\documentclass[12pt,a4paper,oneside]{article}
\usepackage[top=3cm, bottom=3cm, left=2cm, right=2cm]{geometry}
\usepackage{geometry, booktabs, subcaption}
\usepackage{titlesec}
\usepackage[T1]{fontenc}
\usepackage[utf8]{inputenc}
\usepackage{authblk}
\usepackage{amssymb,amsmath}
\usepackage{graphics}

\RequirePackage{amsthm,amsmath,amsfonts,mathrsfs,amssymb}
\RequirePackage[authoryear]{natbib}
\RequirePackage[colorlinks,citecolor=blue,urlcolor=blue]{hyperref}
\RequirePackage{graphicx}

\usepackage{latexsym}
\usepackage{comment}
\usepackage[inline]{enumitem}
\usepackage{hyperref}
\usepackage{graphicx}
\usepackage{placeins}
\usepackage{bbm}
\usepackage{bm}
\usepackage{cleveref}
\usepackage{subcaption}
\usepackage{soul}

\usepackage{xcolor}
\hypersetup{
    colorlinks,
    linkcolor={black!50!black},
    citecolor={black!50!black},
    urlcolor={black!80!black}
}

\usepackage{tikz}
\DeclareRobustCommand\full  {\tikz[baseline=-0.6ex]\draw[thick] (0,0)--(0.50,0);}
\DeclareRobustCommand\dotted{\tikz[baseline=-0.6ex]\draw[thick,dotted] (0,0)--(0.50,0);}
\DeclareRobustCommand\dashed{\tikz[baseline=-0.6ex]\draw[thick,dashed] (0,0)--(0.50,0);}
\DeclareRobustCommand\chain {\tikz[baseline=-0.6ex]\draw[thick,dash dot] (0,0)--(0.50,0);}

\theoremstyle{thmstyleone}%
\newtheorem{theorem}{Theorem}
\newtheorem{corollary}[theorem]{Corollary}
\newtheorem{lemma}[theorem]{Lemma}
\newtheorem{proposition}[theorem]{Proposition}

\theoremstyle{thmstyletwo}%
\newtheorem{example}{Example}%
\newtheorem{remark}{Remark}%
\theoremstyle{thmstylethree}%

\newcommand{\e}{{\rm e}}
\newcommand{\iid}{\stackrel{\mbox{\scriptsize iid}}{\sim}}
\newcommand{\ind}{\stackrel{\mbox{\scriptsize ind}}{\sim}}
\newcommand{\deq}{\stackrel{\mbox{\scriptsize d}}{=}}
\newcommand{\indicator}{\ensuremath{\mathbbm{1}}}

\newcommand{\calC}{\mathcal{C}}
\newcommand{\calM}{\mathcal{M}}
\newcommand{\calN}{\mathcal{N}}

\newcommand{\calX}{\mathcal{X}}
\newcommand{\calS}{\mathcal{S}}

\newcommand{\ptilde}{\tilde{p}}
\newcommand{\mutilde}{\tilde{\mu}}

\newcommand{\dd}{\mathrm d}
\newcommand{\Ga}{\mbox{Gamma}}
\newcommand{\Cov}{\mbox{Cov}}
\newcommand{\plaw}{\mathbf{P}}
\newcommand{\partder}[1]{\frac{\partial}{ \partial #1}}

\newcommand{\E}{\mathsf{E}}

\newcommand{\X}{\mathbb{X}}
\newcommand{\Z}{\mathbb{Z}}
\newcommand{\W}{\mathbb{W}}
\newcommand{\M}{\mathbb{M}}

\newcommand{\R}{\mathbb{R}}
\newcommand{\Ss}{\mathbb{S}}
\newcommand{\prob}{\mathsf{P}}

\renewcommand{\mid}{\ensuremath{\,|\,}}
\newcommand{\midd}{\ensuremath{\,\middle|\,}}

\usepackage[normalem]{ulem}
\newcommand{\MBtext}[1]{#1} 
\title{Bayesian Mixture Models with Repulsive and Attractive Atoms}

\author[1]{Mario Beraha}
\author[2]{Raffaele Argiento}
\author[1]{Federico Camerlenghi}
\author[3]{Alessandra Guglielmi}
\affil[1]{\normalsize{Department of Economics, Management and Statistics, University of Milano-Bicocca}}
\affil[2]{\normalsize{Department of Economics, University of Bergamo}}
\affil[3]{\normalsize{Department of Mathematics, Politecnico di Milano}}

\begin{document}

\maketitle

\begin{abstract}
The study of almost surely discrete random probability measures is an active line of research in Bayesian nonparametrics. 
The idea of assuming interaction across the atoms of the random probability measure has recently spurred significant interest in the context of Bayesian mixture models. 
This allows the definition of priors that encourage well-separated and interpretable clusters.
In this work, we provide a unified framework for the construction and the Bayesian analysis of random probability measures with interacting atoms, encompassing both repulsive and attractive behaviours.
Specifically, we derive closed-form expressions for the posterior distribution, the marginal and predictive distributions, which were not previously available except for the case of measures with i.i.d. atoms.
We show how these quantities are fundamental both for prior elicitation and to develop new posterior simulation algorithms for hierarchical mixture models.
Our results are obtained without any assumption on the finite point process that governs the atoms of the random measure. Their proofs rely on analytical tools borrowed from the Palm calculus theory, which might be of independent interest. 
We specialise our treatment to the classes of Poisson, Gibbs, and determinantal point processes, as well as in the case of shot-noise Cox processes. Finally, we illustrate the performance of different modelling strategies on simulated and real datasets.
\end{abstract}


\section{Introduction}

Clustering is a fundamental problem in statistics and machine learning, being one of the workhorses of unsupervised learning, aiming at dividing datapoints into similar groups.
The Bayesian approach to clustering offers substantial advantages over traditional algorithms, allowing for straightforward uncertainty quantification around point estimates. See \cite{wade23} and \cite{grazian23} for two recent and comprehensive reviews. 
The most common formulations of Bayesian clustering revolve around mixture models \citep{fruhwirth2019handbook}. Mixture models assume that observations belong to one of a potentially infinite number of groups or components, and each group is suitably modelled by a parametric density $f(\cdot \mid Y)$ for some parameter $Y$; their relative prevalence is specified by random weights $(p_j)_{j \geq 1}$ (such that $p_j > 0$ and $\sum_{j \geq 1} p_j = 1$ almost surely).

Recent contributions focused on the robustness of the cluster estimate to model misspecification, establishing a lack thereof in commonly adopted models, both empirically \citep{miller2018robust} and theoretically \citep{cai2021finite, guha2021posterior}. In particular, \cite{cai2021finite} showed how, if the mixture kernel $f(\cdot \mid \cdot)$ does not agree with the true data generating process, the estimated number of clusters by Bayesian nonparametric mixtures diverges as the sample size increases.
The root of this issue can be traced back to a fundamental trade-off between density estimation and clustering in the misspecified regime. To exemplify this, consider the case when data are generated by a mixture of Student $t$ distributions, but the model assumes a mixture of Gaussian distributions. The consistency of Gaussian mixture models for the density estimates \citep{ghosal1999posterior} necessarily entails that the posterior distribution will identify an ever increasing number of clusters, thus producing an inconsistent cluster estimate.

Repulsive mixture models 
\citep{PeRaDu12, xu2016bayesian, xie2019bayesian, bianchini2018determinantal, quinlan2017parsimonious, beraha21, cremaschi2023repulsion} 
offer a practical solution to this problem by forcing the mixture components to be well-separated by assuming a repulsive point process prior for the mixture locations.
However, despite their recent popularity, the mathematical properties of repulsive mixtures have yet to be thoroughly investigated. As a result, prior elicitation strategies and posterior inference algorithms have been derived on a case-by-case basis, often based on heuristics in the case of prior elicitation and inefficient Markov chain Monte Carlo updates for posterior inference.
In sharp contrast, a unified treatment of the main mathematical properties of traditional (non-repulsive) mixtures can be found in the cornerstone paper by \cite{JaLiPr09}, which led to a dramatic increase in the popularity of Bayesian nonparametrics and fostered the development of novel methods, algorithms, and applications. 

The aim of this paper is two-fold. First, we provide a framework for the Bayesian analysis of mixture models encompassing repulsiveness as well as other forms of dependence in the locations, such as attractiveness.
In particular, we fill a gap in the literature by providing a unified framework allowing for the analysis of the associated mixing measure and establishing general results characterising the distribution (both a priori and a posteriori) of several functionals of interest.  Our results do not merely generalise existing theory, but they shed light on large classes of processes which are used in applications without appropriate theoretical investigation.
Secondly, we propose to sidestep the aforementioned trade-off between density and cluster estimation by relying on a novel definition of cluster that naturally arises when assuming a shot-noise Cox process prior \citep{Mo03Cox} for the mixture locations, whereby a cluster can consist of multiple mixture components with similar parameters. In particular, this is the first time such a prior is considered in mixture modelling.

\subsection{Mixture models and almost surely discrete random probability measures}

To formalise the notation, let $Z = (Z_1, \ldots, Z_n)$ be the observed sample, which is assumed to be distributed~as
\begin{equation}\label{eq:mixture1}
    Z_i \mid \ptilde \iid \int f(\cdot \mid y) \ptilde(\dd y),
\end{equation}
with the mixing distribution $\ptilde$ being an almost surely discrete random probability measure (RPM), $\ptilde =\sum_{j \geq 1}   p_j \delta_{X_j}$.
From the seminal work of  \cite{Fer73} on the Dirichlet process, various approaches for constructing RPMs have been introduced.
A fruitful strategy is based on normalising completely random measures with infinite activity, i.e., whose number of support points $( X_j)_{j \geq 1}$ is countably infinite. 
This idea, systematically introduced in \cite{ReLiPr03} for measures on $\R$ with the name of normalised random measures with independent increments (NRMIs), has been extended later to more general spaces \citep[see, e.g.,][]{JaLiPr09}.
More recently, \cite{ArDeInf19} have exploited the same ideas to construct random probability measures with a random finite number of support points, named normalised independent finite point processes.
Both approaches build a discrete RPM by normalising a marked point process where the points of the process define the jumps $(p_j)_{j \geq 1}$ of the RPM and i.i.d. marks $(X_j)_{j \geq 1}$ define the atoms. See also  \cite{lijoi22} for an allied approach.

The direct study of model \eqref{eq:mixture1} is challenging, and it is customary to augment the parameter space via the introduction of latent parameters $Y_i$, $i=1, \ldots, n$ such that
\begin{equation}\label{eq:mix_latent}
    Z_i \mid Y_i \ind f(\cdot \mid Y_i), \qquad Y_i \mid \ptilde \iid \ptilde, \qquad i=1, \ldots, n.
\end{equation}
Then, a fundamental preliminary step for Bayesian inference in the mixture model is to study the distributional properties of the latent sample $Y_1, \ldots, Y_n$.
Specifically, for prior elicitation purposes, it is interesting to investigate the moments of $\ptilde$,  the
prior distribution of the distinct values in the sample, and the marginal distribution of the $Y_i$'s.  For inferential purposes instead, both from a methodological and computational perspective, it is central to derive the conditional distribution of $\ptilde$ given $Y_1, \ldots, Y_n$, as well as the predictive distribution of $Y_{n+1}$ given the latent sample.
See, e.g., \cite{JaLiPr09} and \cite{ArDeInf19} for the expression of such quantities in the case of normalised completely random measures and normalised independent finite point processes, respectively.

In this paper, we consider model \eqref{eq:mix_latent} and propose a construction of RPMs via normalisation of marked point processes with inverted roles with respect to \cite{ReLiPr03} and \cite{ArDeInf19}. In our construction, indeed, the atoms $(X_j)_{j \geq 1}$ are the points of a general simple point process, while the weights $( p_j)_{j \geq 1}$ are obtained by normalising i.i.d. positive marks associated with the atoms.
Through the law of the point process, we are able to encourage different behaviours among the support points of the random probability measure, such as independence (when the point process is Poisson or the class of independent finite point processes), separation (i.e., the support points are well separated, when the point process is repulsive, such as Gibbs or determinantal point processes), and also random aggregation (i.e., the support points are clustered together, when the point process is of Cox type).

\subsection{Our contributions and outline of the paper}

We begin by introducing the model for the $Y_i$'s and the general construction for normalised random measures based on marked point processes in Section~\ref{sec:model}.
In Section~\ref{sec:analysis}, we study the properties of the latent sample $Y_1, \ldots, Y_n$, providing closed-form expressions for several quantities of interest, allowing for a universal theory of very different dependence behaviours such as repulsion or attractiveness.
This is possible thanks to the introduction of technical tools, mainly based on Palm calculus, which might be of independent interest. Palm calculus, beyond notable exceptions, is not typically known in the Bayesian nonparametric literature.
Though the closed-form expressions that we obtain can be generally challenging to evaluate, they drastically simplify in the case of well-known point processes. The manuscript illustrates the general theory using determinantal and shot-noise Cox point processes.
In the appendix we also discuss the case of Poisson and Gibbs point processes.
Then, we resume the study of the Bayesian mixture model  \eqref{eq:mix_latent} in Section~\ref{sec:mix} and show that our analyses can be used as the building block for two Markov chain Monte Carlo algorithms to approximate the posterior distribution.
In Section~\ref{sec:numeric}, we discuss two simulation studies and an application to real data with large sample size, highlighting the main difference between traditional, repulsive, and attractive mixtures. In particular, we show that in a setting with \emph{corrupted} observations, repulsive mixtures are useful to recover the underlying signal, while traditional and attractive mixtures tend to overfit. On the other hand, in more classical misspecification settings where data exhibit heavy tails, we find that attractive mixtures provide a better clustering and density estimate than repulsive mixtures.
We conclude with a discussion in Section~\ref{sec:discussion}.
In the appendix, we report the proofs of the main results, as well as a discussion on the properties of the prior distribution of our random measures and more details on the examples using Poisson, Gibbs, determinantal and shot-noise point processes.   Moreover,  
Appendix~\ref{app:numerical_examples} provides further numerical illustrations, allowing us to explain how we fixed the hyperparameters. We also include a quantitative discussion on the sensitivity of these hyperparameters on density and cluster estimation.

From a technical point of view, our results are based on applying Palm calculus, a fundamental tool in studying point processes.
Essentially, Palm calculus can be regarded as an extension of Fubini's theorem. It allows the exchange of the expectation and integral signs when both are with respect to a point process, with the subtle difference that (in general) the expectation is now taken with respect to the law of another point process, i.e., the \emph{reduced Palm version} of the original process. 
In the Poisson process case (and therefore in the case of completely random measures), the reduced Palm version coincides with the law of the original Poisson process, for which computations are usually easily manageable. 
This property characterises Poisson processes \citep{last2017lectures}.

\subsection{Related works}

This work extends the treatment of  \cite{ArDeInf19} and \cite{beraha21}, who consider finite mixtures with a random number of components. Contrary to our work,  \cite{ArDeInf19} focus on the case of independent and identically distributed atoms of the mixing distribution, while we allow for dependence (general interaction, e.g., repulsiveness or attraction) among the atoms.
\cite{beraha21} consider only repulsive mixtures. The approach there lacks the thorough theoretical background that is instead provided in this paper, though they present a general
framework for repulsive mixture models, and, more importantly, derive an efficient MCMC algorithm which avoids the difficulties of the reversible jump MCMC
computation. 

Another degree of novelty of the present paper concerns the use of Palm calculus for Bayesian analysis.
In the context of BNP models, Palm calculus was introduced in \cite{james2002poisson, james_levy_moving} for the analysis of random probability measures built from Poisson processes.
See also \cite{james_ntr} and \cite{JaLiPr09} for applications to the Bayesian analysis of neutral to the right processes and mixtures of normalised completely random measures, respectively. 
In particular, Section~8 in \cite{james2002poisson} deals with the broad class of weighted Poisson random measures, which are functionals of Gibbs point processes using the terminology of point processes.
Our approach is substantially different. 
First, we do not require that our processes are absolutely continuous with respect to the Poisson one, being able to deal with important processes, such as a large class of determinantal point processes 
and the shot-noise Cox processes. 
Secondly, our approach has the merit of being more direct than  \cite{james2002poisson} because we work directly on the process of interest rather than on the dominating Poisson process. This is possible since we rely on novel technical tools for the analysis of Bayesian nonparametric models borrowed from point process theory.
Consequently, in the mathematical expressions we obtain, it is possible to recognise well-known objects in the study of point processes, which might be more challenging to see in \cite{james2002poisson}. This also allows us to borrow ideas from the literature on the simulation of spatial point processes \citep{MoWaBook03} to address posterior inference.

\section{A general construction for normalized random measures}\label{sec:model}

We consider a sequence of random variables $(Y_i)_{i \geq 1}$ defined on the probability space $(\Omega, \mathcal{A}, \prob)$ and taking values in the Polish space $(\X, \mathcal{X})$, endowed with its Borel $\sigma$-algebra. Moreover, we denote by $\mathbb P(\X)$ the space of all probability measures over $(\X, \mathcal{X})$, and  $\mathcal{P}(\X)$ stands for the corresponding Borel $\sigma$-algebra.
We suppose that 
\begin{equation}\label{eq:model}
    \begin{aligned}
        Y_i \mid \tilde{p} & \iid \tilde p \quad  i \geq 1 \\
        \tilde p & \sim Q,
    \end{aligned}
\end{equation}
where $Q$ is a distribution over $\left(\mathbb P(\X), \mathcal{P}(\X) \right)$.
Note that, by de Finetti's theorem \citep{defin37}, \eqref{eq:model} is equivalent to assuming that the $Y_i$'s are exchangeable, which justifies the Bayesian approach to inference.
{A sample of size $n$ from \eqref{eq:model} consists of the  first $n$ terms of the sequence $(Y_i)_{i \geq 1}$. We will denote this sample by $\bm Y = (Y_1, \ldots, Y_n)$. In this and the following sections, $\bm Y$ represents the observed data points.

\subsection{Random measures as functionals of point processes}
\label{sec:basics_PP}

In this section, we define the class of priors $Q$ in \eqref{eq:model} that we will deal with throughout the paper.
As standard in Bayesian nonparametrics \citep[see, e.g,][]{lijoi2010models} we assume that $\tilde p$ is (almost surely) a purely atomic probability measure, $\tilde p = \sum_{j \geq 1} w_j \delta_{X_j}$ where $(w_j)_{j \geq 1}$ is a collection of positive random variables summing to one and the random atoms $X_j$'s take value in $\X$.
To define such a $\ptilde$, we follow the general approach set forth in \cite{ReLiPr03} and assume that $\ptilde$ arises as the normalisation of a finite random measure $\mutilde = \sum_{j \ge 1} S_j \delta_{X_j}$ built as follows.
Start by considering a \emph{simple point process} $\Phi$ on $\X$, which will define the atoms of $\mutilde$. That is, $\Phi$ is a random counting measure of the kind
\begin{equation}
\Phi(B)= \sum_{j\geq 1}\delta_{X_j} (B), \qquad B \in \calX,
\label{eq:simple_pp}
\end{equation}
where $(X_j)_{j\geq 1}$, the \textit{points} of the process, form a random countable subset of $\X$ such that $\prob(X_i=X_j)=0$ for all $i\neq j$. We denote by $\plaw_\Phi$ the distribution of $\Phi$.
Then, we associate with each point of $\Phi$ independent \emph{marks} $S_j \in \R_+$, such that each $S_j$ has marginal distribution $H$, to define the \emph{marked} point process $\Psi = \sum_{j \geq 1} \delta_{(X_j, S_j)}$ and set 
\begin{equation}\label{eq:mu_def}
    \tilde \mu(B) = \int_{ B\times \R_+ } s \Psi(\dd x \, \dd s ) = \sum_{j \geq 1} S_j \delta_{X_j}(B), \qquad B \in \mathcal{X}.
\end{equation}

Our construction of $\mutilde$ allows us to recover 
the random measures built from the independent finite point processes (also known as mixed binomial processes) analysed in  \cite{ArDeInf19}. In addition, Equation \eqref{eq:mu_def} defines a random measure as a functional of a point process $\Psi$, similarly as in the case of CRMs.
More importantly, our construction allows us to consider measures with \emph{interacting} atoms that result in more informative priors in Bayesian mixture models, for instance by assuming that $\Phi$ is a repulsive Gibbs point process (so that the support points $(X_j)_{j \geq 1}$ are encouraged to be well separated) or a shot-noise Cox process (which, instead, results in the $X_j$'s being randomly clustered together).

Defining $\tilde p$ via the normalisation of $\tilde \mu$ in \eqref{eq:mu_def} requires some care to ensure that $\tilde p$ is well-defined. In particular, we assume that $\Phi$ in \eqref{eq:simple_pp} 
is a finite point process, {i.e., $\Phi(\X) < +\infty$ almost surely (a.s.), which clearly implies $\tilde\mu(\X)<+\infty$ a.s.. 
Then, we set
\begin{equation}\label{eq:pdef}
    \tilde p(B) = \begin{cases}
        \displaystyle\frac{\mutilde(B)}{\mutilde(\X)} & \text{ if } \mutilde(\X) > 0 \\
        0  & \text{ if } \mutilde(\X) = 0
    \end{cases}, \quad B \in \mathcal X
\end{equation}
where, when $\tilde p \equiv 0$,  we intend that the model does not generate any observation $Y_i$. A similar agreement is adopted in
\cite{zhou_fof}.
If we do observe datapoints, then
$\prob(\tilde \mu(\X) = 0 \mid Y_1, \ldots, Y_n) = 0$, as shown in Theorem~\ref{teo:post} below. Hence, a posteriori, the usual assumption $\tilde \mu(\X) > 0$ a.s. is guaranteed. Alternatively, one can assume $\Phi$ is always non-empty, which also fits the definition in \eqref{eq:pdef}. However, most well-studied (repulsive) point processes in the literature assume that $\prob(\Phi(\X) = 0) > 0$. Moreover,
we point out that the requirement $\tilde p =0$ when $\tilde \mu (\X)=0$ makes no issue in any of the proofs; see, e.g., the proof of Proposition \ref{prop:prior_moms}. 
We write $\tilde{p} \sim  {\rm nRM} (\plaw_\Phi, H)$ to denote the distribution of the normalised random measure, while $\tilde{\mu} \sim 
{\rm RM} (\plaw_\Phi, H)$ stands for the distribution of the associated unnormalised random measure. With a little abuse of notation, we will also denote by $\plaw_\Psi$ the law of $\tilde\mu$, where $\Psi$  is $\Phi$ with marks from $H$. 
Note that, as it is clear from the definition, $\ptilde \sim \mbox{nRM}(\plaw_\Phi, H)$ does not lie in the class of \emph{species sampling models} \citep{pitman1996}, since, for instance, the support points are not i.i.d..

\subsection{Background on Point Processes}
\label{sec:back_on_PP}
We now give the necessary background material on point processes and introduce the two processes that will be discussed throughout the paper. To keep the discussion light, we will present here the main mathematical objects in an intuitive and non-formal way. A more technical treatment is deferred to the appendix.

As mentioned in the introduction, our analyses are based on Palm calculus. A central concept of our treatment is the Palm kernel, which may be regarded as an extension of regular conditional distributions to the case of point processes \citep{Kallenberg2021, BaBlaKa}.
Informally speaking, the Palm version $\Phi_x$ of $\Phi$ at $x \in \mathbb X$ is the random measure $\Phi$ conditionally to the event ``$\Phi$ has an atom in $x$''.
Its law is denoted by $\plaw_{\Phi}^{x}$. The latter is referred to as the Palm kernel of $\Phi$ at $x$.
Since $x$ is a \emph{trivial} atom of $\Phi_x$ we can safely discard it and consider the \emph{reduced} Palm version of $\Phi$, $\Phi^!_x := \Phi_x - \delta_x$ with the associated reduced Palm kernel denoted by $\plaw_{\Phi^!}^x$. The argument outlined above can be extended to the case of multiple pairwise different points $\bm x = (x_1, \ldots, x_k)$, leading to the $k$-th Palm distribution $\{\plaw_\Phi^{\bm{x}}\}_{\bm{x} \in \X^k}$.
Again, $\Phi_{\bm x}$ can be understood as the law of $\Phi$ conditional to $\Phi$ having atoms at $\{x_1, \ldots, x_k\}$ and removing the trivial atoms yields the reduced Palm distribution, that is, the law of 
\[
    \Phi^!_{\bm x} := \Phi_{\bm x} - \sum_{j=1}^k \delta_{x_j}.
\]

The other central quantities needed for our subsequent discussion are the so-called moment measures. The moment measure of order one, $M_{\Phi}$, is defined as $M_{\Phi}(B) = \E[\Phi(B)]$ for all $B \in \calX$. The \emph{factorial moment measure} $M_{\Phi^{(k)}}$ of order $k$, instead, is the only measure such that
\[
    \E\left[ \sum_{(x_1, \ldots, x_k) \in \Phi}^{\neq} g(x_1, \ldots, x_k)\right] = \int_{\X^k} g(x_1, \ldots, x_k) M_{\Phi^{(k)}}(\dd x_1 \cdots \dd x_k)
\]
for any measurable function $g: \X^k \rightarrow \R_+$, where the symbol $\neq$ over the summation sign means that $(x_1, \ldots, x_k)$ are pairwise distinct.
Observe that when $k=1$,
$M_{\Phi^{(1)}} \equiv M_{\Phi}$, i.e., the factorial moment measure coincides with the mean measure of $\Phi$.

For a marked point process $\Psi$, with independent marks, the Palm distribution $\Psi_{\bm x, \bm s}$, with  $\bm x=(x_1, \ldots, x_k)$ and $\bm s=(s_1, \ldots, s_k)$, does not depend on $\bm s$.
Moreover, $\Psi^{!}_{\bm x, \bm s}$ has the same law of the point process obtained by considering $\Phi^!_{\bm x}$ and marking it with i.i.d. marks.
See Lemma \ref{prop:marked_palm} in Appendix \ref{app:lemmata}. We write $\Psi^!_{\bm x}$ in place of $\Psi^!_{\bm x, \bm s}$. Since  $\Psi^!_{\bm x}$ is a marked point process, we can define a random probability measure on $\X$ as in the general construction \eqref{eq:mu_def}. Specifically, 
\begin{equation}\label{eq:palm_mu}
    \mutilde^!_{\bm x}(A) := \int_{A \times \R_+} s \Psi^!_{\bm x}(\dd s \, \dd x). 
\end{equation}
We write $ \mutilde^!_{\bm x} \sim \plaw^{\bm x}_{\Psi^!}$.
Note that we can interpret $\plaw^{\bm x}_{\Psi^!}$ as the law of a random measure obtained as follows:
($i$)  take the random measure $\tilde\mu$ as in \eqref{eq:mu_def}, ($ii$) condition to $x_1, \ldots, x_k$ being atoms of $\tilde \mu$  and then ($iii$)  remove $x_1, \ldots, x_k$ from the support. 

We conclude this section by defining the two classes of point processes we will deal with in the rest of the paper. In the supplementary material, we also report the case of the general Gibbs point and Poisson processes.

\begin{example}[Determinantal point processes]\label{ex:dpp_def}
Determinantal point processes \citep[DPPs][]{Macchi75,Hough09,Lav15} are a class of repulsive point processes.
We will restrict our focus to finite DPPs defined as follows. Let $\omega$ be a finite measure on $\X$. Usually, it is the case that $\X$ is a compact subset of $\R^q$ and $\omega$ is the Lebesgue measure, but one could consider more general spaces endowed, e.g., with a Gaussian measure.
Consider a complex-valued covariance function $K: \X \times \X \to \mathbb C$, such that $\int_\X K(x, x) \omega(\dd x) < +\infty$, with spectral representation
\begin{equation}\label{eq:k_mercer}
    K(x, y) = \sum_{h \geq 1} \lambda_h \varphi_h(x) \overline{\varphi_h(y)}, \qquad x, y \in X
\end{equation}
where $(\varphi_h)_{h \geq 1}$ form an orthonormal basis for the space $L^2(\X; \omega)$ of complex-valued functions, $\lambda_h \geq 0$ with $\sum_{h \geq 1} \lambda_h < +\infty$. By Mercer's theorem, the series on the right-hand side of \eqref{eq:k_mercer} converges absolutely and uniformly on $\X^2$.
If $0 \leq \lambda_h \leq 1$, $K$ defines a DPP $\Phi$ on $\X$ \citep{Macchi75, soshnikov2000determinantal}. In particular, the $k$-th factorial moment measure $M_{\Phi^{(k)}}$ admits a density with respect to the product measure $\omega^k$ satisfying
\[
    M_{\Phi^{(k)}}(\dd x_1 \cdots \dd x_k) = \det \left\{ K(x_i, x_j) \right\}_{i, j=1}^n \omega(\dd x_1) \cdots \omega(\dd x_k),
\]
where  $\left\{ K(x_i, x_j) \right\}_{i, j=1}^n$ is the $k \times k$ matrix with $(i,j)$-th entry $K(x_i, x_j)$.
\end{example}

We state below a few remarks on the definition of Determinantal point processes in Example \ref{ex:dpp_def}.

\begin{remark}[Number of points in a DPP]
    The condition $\int K(x, x) \omega(\dd x) < +\infty$ is equivalent to $M_{\Phi}(\X) = \E[\Phi(\X)] < +\infty$, which clearly implies $\Phi(\X) < +\infty$ almost surely. That is, we deal only with finite point processes. This is needed to ensure that $\tilde p$ in \eqref{eq:pdef} is well-defined.
\end{remark}

\begin{remark}[Projection DPPs]
\MBtext{Let $m \in \mathbb N$ (the set of natural numbers);} if the kernel $K$ is such that $\lambda_h = 1$ for $h \leq m$ and $\lambda_h = 0$ for $h > m$, then  the resulting DPP is a \emph{projection} DPP. In particular, it consists of exactly $m$ points, with a joint density with respect to $\omega^m$ given by
    \[
        f_{\Phi}(\nu) \propto \det \left\{K (x, y) \right\}_{x, y \in \{ x_1, \ldots , x_m \}}, \qquad \nu = (x_1, \ldots, x_m).
    \]
\end{remark}

\begin{remark}[DPP densities]
    If the kernel $K $ is such that $\lambda_h < 1$ for all $h$, then $\Phi$ is absolutely continuous with respect to the Poisson process on $\X$ with intensity measure $\omega$, with density
    \[
        f_{\Phi}(\nu)  = \e^{\omega(\X) - D} \det \left\{C(x, y) \right\}_{x, y \in \nu},
    \]
    where $D:=-\sum_{h \geq 1} \log(1 - \lambda_h)$ and $C(x, y) := \sum_{h \geq 1} \frac{\lambda_h}{1 - \lambda_h} \varphi_h(x) \overline{\varphi_h(y)}$ for $x, y \in \X$.
    Such an expression generalizes the one given by \cite{Lav15}, which deals with the case of $\X$ compact set in $\R^q$ and $\omega$ the Lebesgue measure. The proof of such an expression is derived straightforwardly by adapting results from \cite{shirai2003random} and it is provided in Appendix \ref{app:ddp}.
\end{remark}

\begin{remark}[Non finite measures on $\X$]
    If $\omega$ is not finite, then Mercer's theorem does not apply and the spectral representation \eqref{eq:k_mercer} might not hold.
    In such cases, the existence of $\Phi$ requires further conditions on $K$, which are rather technical and beyond the purpose of this paper. \MBtext{See, e.g., \cite{ferreira2009eigenvalues} for further details.}
\end{remark}

\begin{example}[Shot-Noise Cox Processes]\label{ex:sncp_def}
Introduced in \cite{Mo03Cox}, shot-noise Cox processes (SNCPs) are a prominent example of Cox processes \citep{Cox1955}.
For simplicity, we assume that $\X \subset \R^q$. A SNCP $\Phi$ is defined as
\begin{equation}
\label{eq:shot-noise}
        \Phi \mid \Lambda \sim \mathrm{PP}(\omega_\Lambda),\quad
        \omega_\Lambda(\dd x) = \gamma \left(\int_\X k_\alpha(x - v) \Lambda(\dd v)\right) \dd x , \quad
        \Lambda \sim \mathrm{PP}(\omega)
\end{equation}
where $\mathrm{PP}(\omega_\Lambda)$ and $\mathrm{PP}(\omega)$ denote Poisson processes with intensity $\omega_\Lambda$ and $\omega$, respectively.
Here, $k_\alpha$ is a probability density with respect to the $q$-dimensional Lebesgue measure and $\gamma > 0$. In particular, by assuming that $x \mapsto k_\alpha(x - v)$ is continuous for any $v$, we get that the point process $\Phi$ is simple.
We write $\Phi \sim \mathrm{SNCP}(\gamma, k_\alpha, \omega)$, and refer to $\omega$ as the base intensity of the process $\Phi$.
We also introduce the following notation that will be useful in later examples, $\eta(x_1, \ldots, x_{l}): = \int \prod_{i=1}^l k_{\alpha}(x_i - v) \omega(\dd v)$. Observe that if $\omega(\X) = 1$, $\eta$ takes the interpretation of the marginal distribution of a model of the kind $x_1, \ldots, x_k \mid v \iid k_{\alpha}(\cdot - v)$, $v \sim \omega$. \\
Thanks to the coloring theorem of Poisson processes \citep{kingman1992poisson}, \eqref{eq:shot-noise} is equivalent to the following model:
    \begin{equation}\label{eq:shot-noise2}
        \Phi \mid \Lambda = \sum_{h=1}^{n(\Lambda)} \Phi_h, \quad \Phi_h \sim \mathrm{PP}(\gamma k_\alpha(x - \zeta_h) \dd x), \quad \Lambda \sim \mathrm{PP}(\omega),
    \end{equation}
   where $\Lambda = \sum_{h=1}^{n(\Lambda)} \delta_{\zeta_h}$.
   Then, a shot-noise Cox process $\Phi = \sum_{j \geq 1} \delta_{X_j}$ admits the double-summation formulation $\Phi = \sum_{h=1}^{n(\Lambda)} \sum_{j \geq 1} \delta_{X_{h, j}}$.
   Clearly, it is possible to map the latter representation to the former by ordering the $X_{h, j}$'s arbitrarily (thanks to exchangeability), e.g., via their lexicographical order.
   From \eqref{eq:shot-noise2} it is also clear why shot-noise Cox processes are often referred to as
   \emph{cluster processes} \citep[see, e.g.,][]{DaVeJo1}.
   Indeed, introducing an indicator variables $T_j$ for each $X_j$, such that $X_j$ belongs to $\Phi_h$ if and only if  $T_j = h$ (i.e., $\Phi_h(X_j) =1$, or, equivalently, $X_j = X_{h, j^\prime}$ for some $j^\prime$), it is possible to group together the different support points of $\Phi$.
   In this context, instead of referring to the $\Phi_h$'s as \emph{clusters}, we denote them as \emph{groups} to avoid misunderstandings with the usual BNP terminology.
\end{example}


\section{Bayesian analysis of normalized random measures}\label{sec:analysis}

This section contains the most relevant results of the paper: posterior, marginal and predictive distributions for the statistical model \eqref{eq:model}, when $\tilde p \sim  {\rm nRM} (\plaw_\Phi; H)$. All the results are available in closed form, and they constitute the backbone to develop computational procedures for the mixture models in Section~\ref{sec:mix}.

All our proofs are based on the  Laplace functional of the random measure $\mutilde$, defined as
\[
    L_{\mutilde}(f) = \E\left[\exp \left( - \int_\X f(x) \mutilde(\dd x)\right)\right]
\]
for any bounded non-negative function $f: \X \rightarrow \R_+$.
We now introduce useful notation and an additional variable $U_n$, which helps us to describe Bayesian inference for model \eqref{eq:model}.  
The joint distribution of $(\bm Y, \mutilde)$ is given by
\begin{equation}\label{eq:joint}
    \prob(\bm Y \in \dd \bm y, \tilde \mu \in \dd \mu) = \frac{1}{\mu(\X)^n} \prod_{i=1}^n \mu(\dd y_j) \plaw_\Psi(\dd \mu),
\end{equation}
where $\plaw_\Psi$ is the law of $\tilde \mu$. 
\MBtext{Posterior inference in \eqref{eq:joint} is complex due to the term $\mu(\X)^{-n}$, which does not allow to separate the atoms of $\tilde \mu$ that have been observed in the sample $\bm Y$ from the atoms that have not been observed. A popular workaround to this issue, originally noted in} in \cite{JaLiPr09}, is to introduce an auxiliary variable $U_n \mid \mutilde \sim \Ga(n, \mutilde(\X))$ and, by a suitable augmentation of the underlying probability space, we consider the joint distribution of $(\bm{Y}, U_n, \tilde{\mu})$
\begin{equation}\label{eq:joint_u}
    \prob(\bm Y \in \dd \bm y, U_n \in \dd u, \tilde \mu \in \dd \mu) = \frac{u^{n-1}}{\Gamma(n)} \e^{- \mu(\X) u} \dd u \prod_{i=1}^n \mu(\dd y_j) \plaw_\Psi(\dd \mu),
\end{equation}
\MBtext{where we have gotten rid of the term $\mu(\X)^{-n}$. By considering \eqref{eq:joint_u}, we can study the conditional law of $\mutilde$ given $(\bm Y, U_n)$. Together with the posterior of $U_n \mid \bm Y$, this allows us to obtain a disintegration of the posterior of $\mutilde$ in \Cref{teo:post} below, which is amenable to analytical and numerical computations.}

Since $\mutilde$ is almost surely discrete, with positive probability, there will be ties within the sample $\bm{Y}=(Y_1, \ldots, Y_n)$. For this reason $\bm{Y}$ is equivalently characterized by 
the couple $(\bm{Y}^*, \tilde \pi)$, where 
 $\bm Y^* = (Y^*_1, \ldots, Y^*_{K_n})$ is the vector of distinct values and $\tilde \pi$ is the random partition of $[n]: =  \{1, \ldots, n\}$ of size $K_n$, identified by the equivalence relation $i \sim j$ if and only if  $Y_i = Y_j$. Given $K_n =k$, we indicate by  $\bm{n}= (n_{1}, \ldots , n_{k})$ 
 the vector of counts, i.e., 
 $n_j$ is the cardinality of the set $\{i \in [n]: Y_i = Y^*_j\}$, as $j =1, \ldots , k$. As a consequence, we may write
\[
	 \prob(\bm Y \in \dd \bm y, U_n \in \dd u, \tilde \mu \in \dd \mu)  = \frac{u^{n-1}}{\Gamma(n)} \e^{-\mu(\X) u} \prod_{j=1}^k \mu(\dd y^*_j)^{n_j} \plaw_\Psi(\dd \mu).
\]
The augmentation of \eqref{eq:joint} through $U_n$ is made possible since the marginal distribution of $U_n$ exists and has a density, with respect to the Lebesgue measure, that is 
$$f_{U_n}(u)= \frac{u^{n-1}}{\Gamma(n)} \int_{0}^{+\infty} t^n \e^{- tu} f_{\mu}(t) \dd t, $$ 
where $f_\mu$ is the density or the random variable $\tilde \mu(\X)$.

\subsection{Posterior characterization}
We first characterise the posterior distribution of $\tilde{p}$ in model \eqref{eq:model}  when a priori $\tilde p\sim {\rm nRM} (\plaw_\Phi; H)$. Since $\tilde{p}$ is obtained by the normalisation of the random measure $\tilde \mu$, it is sufficient to describe the posterior distribution of $\tilde{\mu}$, which is provided in the following theorem. 
\begin{theorem}\label{teo:post}
 Assume that $H(\dd s) = h(s) \dd s$ where $\dd s$ is the Lebesgue measure.
The distribution of $\mutilde$ conditionally on  $\bm Y = \bm y$ and $U_n =u$  is equal to the 
distribution of 
\begin{equation}\label{eq:mu_post}
\sum_{j=1}^k S_j^* \delta_{y^*_j} + \mutilde^\prime
\end{equation}
where:
\begin{itemize}
    \item[(i)]$\bm{S}^*:= (S_1^*, \ldots , S_k^*)$ is a vector of independent random variables,  with density 
\[
    f_{S_j^*}(s) \propto \e^{- u s} s^{n_j} h(s) \text{ for } j=1, \ldots , k;
\]
    \item[(ii)]  $\mutilde'$ is a random measure with Laplace functional
\begin{equation} \label{eq:Laplace_mu_post}
    \E\left[\exp \int_\X - f(z) \mutilde '(\dd z) \right] = \frac{\E \left[ \exp \left\{- \int_\X (f(z) + u) \mutilde^!_{\bm y^*}(\dd z) \right\} \right]}{\E \left[ \exp \left\{- \int_\X u \mutilde^!_{\bm y^*}(\dd z) \right\} \right]},
\end{equation}
and $\mutilde^!_{\bm y^*}$ is as in \eqref{eq:palm_mu} for $\bm x = \bm y^*$.
\end{itemize}
In \eqref{eq:mu_post} above, $\bm{S}^*$ and $\mutilde'$ are independent. 
Moreover, the conditional distribution of $U_n$, given $\bm Y=\bm y$, has a density with respect to the Lebesgue measure, which satisfies 
\begin{equation}\label{eq:post_u}
    f_{U_n \mid \bm Y}(u) \propto u^{n-1} \E\left[\e^{-  u \mutilde^!_{\bm y^*} (\X)} \right] \prod_{j=1}^k \kappa(u, n_j), \quad u>0,
\end{equation}
with
$\kappa(u, n): = \int_{\R_+} \e^{-u s} s^{n} H(\dd s)$.
 \end{theorem}

First, we point out that the Laplace functional of $\mutilde '$ is an exponential tilting of the one of $\mutilde^!_{\bm y^*}$, that is to say the distribution of $\mutilde '$ is absolutely continuous with respect to the distribution of $\mutilde^!_{\bm y^*}$,  with density given by $f_{\mutilde '}(\mu) \propto \e^{-u \mu (\X)}$.
Theorem~\ref{teo:post} resembles the posterior characterisation of $\tilde p$ in the case of NRMIs \citep{JaLiPr09} and normalised IFPPs \citep{ArDeInf19}.  Interestingly, unlike in the previous cases, here the law of $\mutilde '$  depends on the observed $\bm y^*$. Note that the expression \eqref{eq:Laplace_mu_post} is obtained without any specific assumption on the law of the point process $\Phi$. 
Specific expressions for the posterior distribution of $\mutilde$ are obtained when considering particular classes of point processes, as we showcase in the following examples.
For the sequel, we remind that  $\psi(u) := \E[\e^{-uS}]$  is the Laplace transform of a random variable  $S$, with distribution $H$, and we also define
\begin{equation}\label{eq:h_tilt}
    f_{S^\prime} (s; u) := \frac{\e^{-su}h(s)}{\psi(u)}=\frac{\e^{-su } h (s)}{\int_{\R^+}  \e^{-su}  h(s) \dd s}
\end{equation}
to be the density of the exponential tilting of $S$. 

We now specialise Theorem~\ref{teo:post} to our examples, whose posterior distribution was not previously available in the literature.

\begin{example}[Determinantal point process]\label{ex:dpp1}
Assume that $\Phi$ is a DPP on $\X$ as defined in Example \ref{ex:dpp_def}, whose kernel $K$ has eigenvalues $\lambda_j$ in \eqref{eq:k_mercer} all strictly smaller than one. 
Then, the random measure $\mutilde^\prime$ in Theorem~\ref{teo:post}  is such that
\[
   \mutilde'\stackrel{d}= \sum_{j \geq 1} S^\prime_j \delta_{{X}^\prime_j} 
\]
where $\Psi':=\sum_{j \geq 1} \delta_{({X}_j', {S}_j')}$ is a marked point process whose unmarked version $\Phi' := \sum_{j \geq 1} \delta_{{X}_j '}$ is a DPP with density with respect to $\mathrm{PP}(\omega)$ given by $f_{\Phi^\prime}(\nu) \propto \det\{C^\prime(x_i, x_j)\}_{(x_i, x_j) \in \nu}$ and the marks $S_j'$ are i.i.d. with distribution \eqref{eq:h_tilt}.
The operator $C^\prime$ is defined as 
\[
    C^\prime(x, y) = \psi(u) \left[C(x, y) - \sum_{i, j = 1}^k \left(C_{\bm y^*}^{-1} \right)_{i, j} C(x, y^*_i) C(y, y^*_j)\right], \qquad x, y \in R,
\] 
where $C_{\bm y^*}$ is the $k\times k$ matrix with entries $C(y^*_i, y^*_j)$. 
See \Cref{teo:post_dpp} in \Cref{app:ddp} for a proof.
Note that for simulation purposes, it is useful to know or approximate the eigendecomposition of the kernel $C^\prime$:
\[
   C^\prime(x, y) = \sum_{j} \gamma_j \varphi^\prime_j(x)\overline{\varphi^\prime}_j(y).
\]
The $\gamma_j$'s and $\varphi^\prime_j$'s can be approximated numerically using the Nystr{\"o}m method \citep{nystrom}.
\end{example}

In order to specialize Theorem \ref{teo:post}
to the Shot-Noise Cox process, we introduce
$\bm T^\prime=(T^\prime_1,\dots,T^\prime_k)$, the indicator variables describing a partition of $[k]= \{ 1, \ldots , k\}$, $k$ being the number of distinct values in $\bm y$, namely $T^\prime_j=l$ if and only if $j$ belongs to the $l$-th element of partition introduced in Lemma \ref{lemma:moment_cox}.
In practice, $T^\prime_j$, $j=1,\ldots,k$, describe the points of the Poisson process $\Lambda$ associated to the unique value $y^*_1,\ldots,y^*_k$. 
 
\begin{example}[Shot-Noise Cox process]\label{ex:sncp1}
Assume that $\Phi \sim \mathrm{SNCP}(\gamma, k_\alpha, \omega)$. Then, the random measure $\mutilde^\prime$ in \eqref{eq:mu_post} has a mixture representation.  Let $\bm T^\prime=(T^\prime_1,\dots,T^\prime_k)$ be latent indicator variables introduced above and 
denote by $|\bm T^\prime|$ the number of distinct values in $\bm T^\prime$. We have
\[
    \mutilde^\prime \mid \bm T^\prime \deq \mutilde_0 + \sum_{\ell = 1}^{|\bm T^\prime|} \mutilde_{\ell},
\]
where $\mutilde_0 = \sum_{j \geq 1} \tilde S_j \delta_{\tilde X_j}$ and $\mutilde_{\ell} = \sum_{j \geq 1} \tilde S_{\ell, j} \delta_{\tilde X_{\ell, j}}$, as $\ell=1, \ldots, |\bm T^\prime|$, are independent, and all their weights are i.i.d. with distribution \eqref{eq:h_tilt}.
The unmarked point processes $\Phi_0 = \sum_{j \geq 1} \delta_{\tilde X_j}$ is distributed as $\Phi_0 \sim \mathrm{SNCP}(\gamma \psi(u), k_\alpha, \e^{- \gamma (1-\psi(u))} \omega)$.
Each $\Phi_\ell = \sum_{j \geq 1} \delta_{\tilde X_{\ell, j}}$ is a Poisson point process with random intensity $\omega_{\ell} (\dd x) = \gamma \psi(u) k_{\alpha} (\zeta_\ell - x) \dd x$, where
\[
    \zeta_\ell \sim f_{\zeta_\ell}(v) \dd v \propto \prod_{j: T_j^\prime = \ell} k_\alpha(y^*_j - v) \omega(\dd v).
\]
Finally, $\prob(\bm T^\prime = \bm t^\prime) \propto  \exp\left\{\gamma |\bm t^\prime| (\psi(u) - 1)\right\} \prod_{\ell=1}^{|\bm t^\prime|} \eta((y^*_j: t^\prime_j = \ell))$. For the proof and the formal statement, see Theorem \ref{teo:post_cox} in Appendix H.2.
\end{example}

\subsection{Marginal and predictive distributions}

In the present section, we describe the marginal and predictive distributions for a sample from $\tilde{p}$ as in \eqref{eq:model}, when the prior distribution is defined as $\tilde p\sim {\rm nRM} (\plaw_\Phi; H)$.
As mentioned at the beginning of the present section, the almost sure discreteness of $\ptilde$ entails that the sample $\bm{Y}=(Y_1, \ldots, Y_n)$ is equivalently characterised by  $(\bm{Y}^*, \tilde \pi)$, where 
$\bm Y^* = (Y^*_1, \ldots, Y^*_k)$ is the vector of distinct values and $\tilde \pi$ is the random partition of $[n]: =  \{1, \ldots, n\}$ of size $K_n$, 
identified by the equivalence relation $i \sim j$ if and only if  $Y_i = Y_j$. 

\begin{theorem}\label{teo:marg}
The marginal distribution of a sample $\bm Y$ with  $K_n= k$ distinct values $\bm{Y}^*$ and associated counts $n_1, \ldots , n_k$ is
\[
\prob (\bm{Y} \in \dd \bm{y}) = \prob(\bm Y^* \in \dd \bm y^*, \tilde \pi = \pi)  = \int_{\R_+} \frac{u^{n-1}}{\Gamma(n)} \E \left[ \e^{-  u \mutilde^!_{\bm y^*}(\X)} \right] \prod_{j=1}^k \kappa(u, n_j) \dd u \, M_{\Phi^{(k)}}(\dd \bm y^*),
\]
where $\mutilde^!_{\bm y^*}$ is defined in \eqref{eq:palm_mu} and $M_{\Phi^{(k)}}$ is the $k$-th factorial moment measure of $\Phi$
defined in Section~\ref{sec:back_on_PP}.
\end{theorem}

Theorem~\ref{teo:marg} resembles the marginal characterisation of a sample from an NRMI \citep{JaLiPr09} and a normalised IFPP \citep{ArDeInf19}. However, there are key differences: in the case of NRMIs, $\mutilde^!_{\bm y^*} \stackrel{d}{=} \mutilde$ (i.e., it is distributed as the prior and, in particular, it does not depend on $\bm y^*$) so that $\E [ \exp(- u \mutilde^!_{\bm y^*}(\X))]$ can be evaluated explicitly via the L\'evy-Khintchine representation; 
in the case of IFPPs, $\mutilde^!_{\bm y^*}$ depends on $\bm y^*$ only through its cardinality and the latter expectation can be computed analytically.
In both cases, the $k$-th factorial moment measure has a product form: $M_{\Phi^{(k)}}(\dd \bm y^*) = \prod_{j=1}^k \omega(\dd y^*_j)$ where $\omega$ is a finite measure on $\X$.
This allows for the marginalization of the $\bm Y^*$'s, also obtaining an analytical expression for the prior induced on the random partition $\tilde \pi$, the so-called exchangeable partition probability function.
In our case, such a marginalisation is possible only from a formal point of view, and the resulting expression is generally intractable.
We now specialise Theorem~\ref{teo:marg} to important examples.

\begin{example}[Determinantal point process, cont'd]\label{ex:dpp2}
If $\Phi$ is a DPP, for all $x_1, \ldots, x_n$ such that $K(x_j, x_j) > 0$, $\Phi^!_{\bm x}$ is a DPP with kernel 
\[
    K^!_{\bm x}(x, y) = K(x, y) - \sum_{i, j=1}^k (K_{\bm x}^{-1}) K(x, x_i) K(y, x_j)
\]
where $K_{\bm x}$ is the $n\times n$ matrix with entries $K(x_i, x_j)$ \citep{Lav15}. Denote by $\lambda^{\bm x}_j$ the eigenvalues of $K^!_{\bm x}$.
Then,  $\Phi^!_{\bm x}(\X)$ is a Poisson-Binomial random variable with parameter $(\lambda^{\bm x}_j)_{j \geq 1}$, which entails $\E [ \exp(- u \mutilde^!_{\bm y^*}(\X))] = \prod_{j \geq 1} (1 - \lambda^{\bm x^*}_j + \lambda^{\bm x^*_j} \psi(u))$.
Moreover, the factorial moment measure equals
\[
    M_{\Phi^{(k)}}(\dd \bm y^*) = \det\{K(y^*_i, y^*_j)\}_{i, j =1}^k \omega^k(\dd \bm y^*).
\]
\MBtext{As an alternative to numerically computing the $\lambda^{\bm x}_j$'s, \citet{GhiloFeatures} proposed using Le Cam's approximation \citep{steele1994cam}, i.e., approximating $\Phi^!_{\bm x}(\X)$ with a Poisson random variable of mean $\lambda_{\bm x^*} := \int_{\X} K^!_{\bm x^*}(x, x) \dd x$, showing that the approximation error is small. 
In such a case  $\E [ \exp(- u \mutilde^!_{\bm y^*}(\X))] \approx \exp(\lambda_{\bm x^*}(\psi(u) - 1))$.}
\end{example}

\begin{example}[Shot-Noise Cox process, cont'd]\label{ex:sncp2}
If $\Phi$ is a SNCP, then $\Phi^!_{\bm x}$ can be written as in Lemma \ref{lem:palm_cox} of Appendix \ref{app:sncp}, while the moment measure is described in Lemma \ref{lemma:moment_cox}, which lead to the formula of the marginal distribution. However, a more transparent expression is obtained by considering also the variables $\bm T^\prime$ introduced in Example \ref{ex:sncp1}. Indeed, we have 
\[
    \prob(\bm Y \in \dd \bm y , \bm T^\prime= \bm t^\prime \mid U_n=u) \propto \gamma^k \prod_{j=1}^k \kappa(u, n_j)  \prod_{\ell =1}^{|\bm t^\prime|} \eta( \bm y^*_\ell ) \e^{\gamma |\bm t^\prime| (\psi(u) - 1)} \dd \bm y^*,
\]
where $\bm y^*_\ell = (y^*_j \colon t^\prime_j = \ell)$.
\end{example}

We now focus on the predictive distribution, i.e., the distribution of  $Y_{n+1}$, conditionally to the observable sample $\bm{Y}$.
\begin{theorem}\label{teo:pred_conditional}
Assume that the factorial moment measure  $M_{\Phi^{(k)}}$ is absolutely continuous with respect to the product measure $P_0^k$, where $P_0$ is a non-atomic probability on $(\X,\mathcal{X})$ and let $m_{\Phi^{k}}$ be the associated Radon-Nikodym derivative. Let $U_n$ be a random variable with density  $f_{U_n \mid \bm Y}$ as in  \eqref{eq:post_u}.
Then, conditionally on $\bm Y = \bm y$ and $U_n=u$, the predictive distribution of $Y_{n+1}$ is such that 
\begin{equation} \label{eq:prediction_rule}
    \begin{split}
           \prob(Y_{n+1} \in A \mid \bm Y = \bm y, U_n=u)  &\propto \sum_{j=1}^k \frac{\kappa(u, n_j + 1)}{\kappa(u, n_j)}   \delta_{y_j^*} (A) \\
   & \qquad\qquad +\int_A \kappa(u, 1) \frac{\E \left[ \e^{- u \mutilde_{(\bm y^*, y)}^!(\X)} \right]}{\E \left[ \e^{- u \mutilde_{\bm y^*}^!(\X)} \right]} \frac{m_{\Phi^{k+1}}(\bm y^*, y)}{m_{\Phi^{k}}(\bm y^*)} P_0(\dd y) ,
    \end{split}
\end{equation}
where $A \in \mathcal{X}$. 
\end{theorem}
The predictive distribution from Theorem~\ref{teo:pred_conditional} can be interpreted using a generalised Chinese restaurant metaphor.
As in the traditional Chinese Restaurant process, customers sitting at the same table eat the same dish, whereas the same dish cannot be served at different tables.
The first customer arrives at the restaurant and sits at the first table, eating dish $Y_1 = Y^*_1$ such that
$\prob(Y^*_1 \in \dd y) \propto M_{\Phi}(\dd y)$.
The second customer sits at the same table as the first customer with probability proportional to $\kappa(U_1, 2) / \kappa(U_1, 1)$, or sits at a  new table and eats a new dish  with probability proportional to
\[
   \prob( Y_2 \in \X \setminus \{ y_1^*\} \mid Y^*_1 = y^*_1) 
    \propto
    \frac{ \kappa(U_1, 1)}{\E \left[ \e^{- U_1 \mutilde_{y^*_1}^!(\X)} \right]m_{\Phi}( y^*_1)}
    \int_\X 
     \E \left[ \e^{-  U_1 \mutilde_{(y^*_1, y)}^!(\X)} \right] m_{\Phi^{2}}(y^*_1, y) P_0(\dd y),
\]
where $U_1 \sim f_{U_1 \mid y_1}$ is defined in \eqref{eq:post_u}.
The distribution of the new dish is proportional to 
\[
    \mathsf{P}( Y_2 \in \dd y | Y_1^* = y^*_1) \propto \E \left[ \e^{- U_1 \mutilde_{(y^*_1, y)}^!(\X)} \right] m_{\Phi^{2}}(y^*_1, y) P_0(\dd y).
\]
The metaphor proceeds as usual for a new customer entering the restaurant at time $n+1$: they either sit at one of the $K_n = k$ previously occupied tables, with probability proportional to $\kappa(U_{n}, n_j + 1) / \kappa(U_{n}, n_{j})$, $j=1, \ldots, k$ or sit at a new table eating dish $Y^*_{k + 1} = y$ with probability proportional to
\[
    \kappa(U_n, 1) \frac{\E \left[ \e^{- U_n \mutilde_{(\bm y^*, y)}^!(\X)} \right]}{\E \left[ \e^{- U_n \mutilde_{\bm y^*}^!(\X)} \right]} \frac{m_{\Phi^{k+1}}(\bm y^*, y)}{m_{\Phi^{k}}(\bm y^*)} P_0(\dd y).
\]
Observe that the original formulation of the Chinese restaurant process \citep{Aldous85} is stated only in terms of the seating arrangements in a restaurant with infinite tables, i.e., the random partition. 
Then, $K_n$ ``dishes'' are sampled i.i.d. from a distribution on $\X$ and assigned to each of the occupied tables.
Instead, in our process, the dishes are not i.i.d. and are chosen as soon as a new table is occupied.
In particular, note that the probability of occupying a new table depends on $n$, $K_n$, and the unique values $\bm Y^*$. 
This also sheds light on how our model differs from \emph{Gibbs type priors} \citep{deblasi2013gibbs}, where the probability of occupying a new table depends only on $n$ and $K_n$.

We now specialise the predictive distribution in some examples of interest.

\begin{example}[Determinantal point process, cont'd]\label{ex:dpp3}
Let $K_{\bm x}$ be defined as in Example \ref{ex:dpp2}. Following Example \ref{ex:dpp2}, the ratio of the expected values in \eqref{eq:prediction_rule} equals 
\[
    \prod_{j \geq 1} \frac{1 - \lambda^{\bm x, y}_j + \lambda^{\bm x, y}_j \psi(u)}{1 - \lambda^{\bm x}_j + \lambda^{\bm x}_j \psi(u)},
\]
which can be numerically computed via  Le Cam's approximation.
Moreover, by the Schur determinant identity, denoting by $K_{y, y} := K(y, y)$, $K_{\bm y^*, y} := (K(y^*_1, y), \ldots, K(y^*_k, y))^\top$,  and by $K_{\bm y^*}$ the $k \times k$ matrix with entries $(K(y^*_i, y^*_j))_{i,j=1}^k$, we have
\[
    \frac{m_{\Phi^{k+1}}(\bm y^*, y)}{m_{\Phi^{k}}(\bm y^*)}P_0(\dd y) = \left(K_{y, y} - K_{\bm y^*, y}^\top K_{\bm y^*, \bm y^*}^{-1} K_{\bm y^*, y}\right) \dd y.
\]
\end{example}

\begin{example}[Shot-Noise Cox process, cont'd.]
As in Example \ref{ex:sncp2}, we consider the auxiliary latent vector $\bm T^\prime$, and its corresponding realization $\bm t^\prime$, to describe the predictive distribution. In particular, as discussed in Appendix \ref{app:sncp_pred}, we can think of a restaurant metaphor whereby tables are divided into rooms, such that $t^\prime_j = \ell$ if table $j$, serving dish $y^*_j$, is located in room $\ell$. If, after $n$ customers have entered the restaurant, having occupied $k$ tables in $|\bm T^\prime|=|\bm t^\prime|$ rooms, customer $n+1$ can do one of the following choices.
\begin{itemize}
    \item[(i)] Customer $n+1$ sits at one of the previously occupied tables, say table $j$, with probability proportional to $\kappa(u, n_j + 1) / \kappa(u, n_j)$.
    \item[(ii)] Customer $n+1$ sits at a new table in room $\ell \in \{1, \ldots, |\bm t^\prime_k|\}$ with probability proportional to $\kappa(u, 1) \int_\X  \eta(\bm y^*_\ell, y) / \eta(\bm y^*_\ell) \dd y$, eating a dish $y$ generated from a probability density proportional to $ \eta(\bm y^*_\ell, y) / \eta(\bm y^*_\ell) \dd y$. 
    \item[(iii)]  Customer $n+1$ enters an empty room (sitting, a fortiori, in a new table) with probability proportional to $\kappa(u, 1) \e^{\gamma (\psi(u)- 1))} \int_\X  \eta(y) \dd y $. The customer eats a new dish $y$ generated from a  probability density function proportional to $\eta(y) \dd y $. 
\end{itemize}

The process described above shares similarity with the Chinese restaurant franchise metaphor of \cite{Teh06} in the case of a single restaurant \citep[see][for further details]{Cam18SJS}, the main difference being that if a new table is chosen, the new dish will be different from previously eaten dishes almost surely.
\end{example}

\subsection{Distribution of the distinct values}

From the marginal characterisation in Theorem~\ref{teo:marg}, it is easy to derive the joint distribution of $(K_n, \bm Y^*)$, that is, the joint distribution of the number of the distinct values and their position in a sample of size $n$.

\begin{proposition}\label{prop:joint_nclus}
  Given a set of distinct points $\bm y^* = (y^*_1, \ldots, y^*_k)$, let $(q_r)_{r \geq 0}$ be the probability mass function of the number of points in $\Phi^!_{\bm y^*}$, i.e., $q_r := \prob \left(\Phi^!_{\bm y^*} (\X) = r \right)$. Define
    \[
        V(n_1, \ldots, n_k; r) := \int_{\R_+} \frac{u^{n-1}}{\Gamma(n)}  \psi(u)^r  \prod_{j=1}^k \kappa(u, n_j) \dd u. 
    \]
   Then, the joint distribution of $K_n$ and $\bm Y^*$ equals
   \begin{align*}
        \prob(K_n = k, \bm Y^* \in \dd \bm y^*) = \frac{1}{k!} \sum_{r=0}^\infty \left(\sum_{(n_1, \ldots, n_k) \in \mathbb S_n^k} \binom{n}{n_1 \cdots n_k}  V(n_1, \ldots, n_k; r) \right) q_r \     M_{\Phi^{(k)}}(\dd \bm y^*),
    \end{align*}
    where $\mathbb S_n^k$ denotes the set of $k$-compositions, i.e. $\mathbb S_n^k = \{(n_1, \ldots, n_k) \ : n_i \geq 1, n_1 + \cdots + n_k = n\} \subset \mathbb N^k$.
\end{proposition}
Note that the joint distribution of $K_n$ and $\bm Y^*$ does not factorise. Moreover, it is not easy to marginalise $\bm Y^*$ since each $q_r$ in general depends on $\bm Y^*$. 
Now, it is interesting to specialise Proposition \ref{prop:joint_nclus} for a special choice of the distribution of the jumps $S_j$.
\begin{corollary}\label{cor:uniques_gamma}
    Under the same assumptions as in Proposition \ref{prop:joint_nclus}, if $S_j\iid \mbox{Gamma}(a, 1)$, then
\begin{equation*}\label{eq:distinct_gamma}
   \prob(K_n = k, \bm Y^* \in \dd \bm y^*) = (-1)^n\mathcal{C}(n, k; -a) \left(\sum_{r \geq 0}  q_r \frac{\Gamma\left(k + r) a\right)}{\Gamma\left((k + r)a + n\right)} \right) M_{\Phi^{(k)}}(\dd \bm y^*)
\end{equation*}
where for any non-negative integer $n\ge 0$, $0\le k\le n$, $\mathcal{C}(n,k;a)$ denotes the central generalized factorial coefficient. See \cite{charalambides_enumerativeC}.
\end{corollary}
Observe that these $\mathcal{C}(n,k;a)$ can be computed using the recursive formula
\begin{equation*}
    \mathcal{C}(n,k;a) = a \mathcal{C}(n-1,k-1;a)
    +(k a-n+1)\mathcal{C}(n-1,k;a)   
\end{equation*}
with the initial conditions $\mathcal{C}(0,0,a)=1$, $\mathcal{C}(n, 0,a) = 0$ for any $n > 0$ and $\mathcal C(n, k, a) = 0$ for $k > n$.
See \cite{charalambides_enumerativeC} for further details.

Using Corollary \ref{cor:uniques_gamma}, we now consider a concrete example highlighting the difference between a repulsive and a non-repulsive point process to show the great flexibility of our prior.
To this end, we compute $\prob(K_n = k, \bm Y^* \in \dd \bm y^*)$ for $n=5$ and $a=1$
under two possible priors for $\Phi$, i.e., a Poisson process and a DPP.
In particular, the Poisson process prior has intensity $\omega(\dd x) = \indicator_{D}(x) \dd x$ where $D = (-1/2, 1/2)$.
The DPP prior is defined on $D$ as well and is characterised by a Gaussian covariance function $K(x, y) = 5 \exp(- (x - y)^2 / 0.3)$.
We consider different settings: in the first (I) $\bm y^* = (-x, x)$, in the second (II) $\bm y^* = (-0.3, -0.3 + 2x)$ and in the third (III) $\bm y^* = (-x, 0, x)$. Moreover, we assume that $x$ varies in the interval $(0, 0.4)$.
Figure~\ref{fig:prior_kn} shows the joint probability of $K_n$ and $\bm Y^*$ in the different scenarios. 
Note that under the Poisson process prior (solid line, left plot), the probability does not depend on $\bm y^*$.
Instead, under the DPP prior, the three panels show that the probability increases when the points in $\bm y^*$ are well separated. 
The central panel shows the comparison between the probabilities under setting (I) (\dashed) and setting (II) (\chain) for the DPP prior. This shows that these probabilities depend not only on the distance between the 
$\bm y^*$, but also on their position in $D$. 
The right panel displays setting (I) (\dashed) and setting (III) (\dotted) under the DPP prior, showing the effect of observing 0 beyond observing $-x,x$.

\begin{figure}
    \centering
    \includegraphics[width=\linewidth]{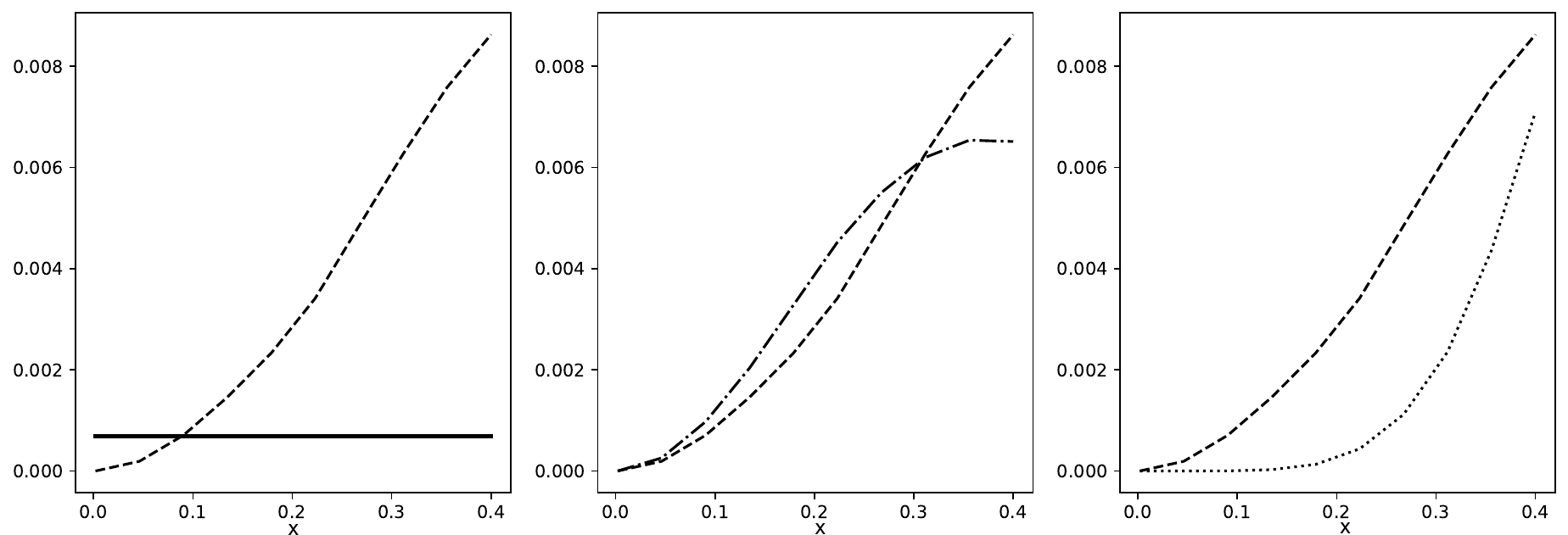}
    \caption{$\prob(K_n = k, \bm Y^* \in \dd \bm y^*)$ when $n=5$, $a=1$, as a function of the parameter $x$ under different settings.
    Left plot, setting (I) under the Poisson process (\full) and DPP (\dashed) prior.
    Middle plot: setting (I) (\dashed) and setting (II) (\chain) under the DPP prior.
    Right plot setting (I) (\dashed) and setting (III) (\dotted) under the DPP prior.}
    \label{fig:prior_kn}
\end{figure}

\section{Bayesian Hierarchical Mixture Models}\label{sec:mix}

Discrete random probability measures are commonly employed as mixing measures in Bayesian mixture models to address clustering and density estimation.
In a mixture model, instead of modelling observations via \eqref{eq:model}, we use the de Finetti measure $Q$ in \eqref{eq:model} as a prior for latent variables $Y_1, \ldots, Y_n$. 

It is convenient to consider random measures on an extended space $\X \times \W$ instead of $\X$ to encompass location-scale mixtures. Hence, we deal with the random measure
\begin{equation}\label{eq:mixing_meas}
    \mutilde(\cdot) = \sum_{j \geq 1 }  S_j \delta_{(X_j, W_j)},
\end{equation}
obtained by marking the points in the point process $\Psi$ with i.i.d. marks $(W_j)_{j \geq 1}$ from an absolutely continuous probability distribution over $\W$, whose density we will denote by $f_W(\cdot)$. 

To formalise the mixture model, consider $\Z$-valued observations $Z_1, \ldots, Z_n$ and a probability kernel $f: \mathbb Z \times \X \times \W \rightarrow \R_+$, such that $z \mapsto f(z \mid y, v)$ is a probability density over $\mathbb Z$ for any $(y, v ) \in \mathbb \X \times \W$. We assume both  
$\mathbb Z$ and $\W$ to be Polish spaces, endowed with the associated Borel $\sigma$-algebras. 
More precisely, the statistical model we are dealing with is the following:
\begin{equation}\label{eq:mixture}
\begin{aligned}
    Z_i \mid Y_i, V_i & \ind f(\cdot \mid Y_i, V_i), \qquad i=1, \ldots, n \\
    Y_i, V_i \mid \mutilde & \iid \frac{\mutilde}{\mutilde(\X \times \W)} \\
    \mutilde & \sim \plaw_\mu .
\end{aligned}
\end{equation}
Usually, the kernel $f(\cdot \mid Y_i, V_i)$ is the Gaussian density with mean $Y_i$ and variance (if data are univariate) or covariance matrix (if data are multivariate) $V_i$, where the points $X_j$'s are the component-specific means and the points $W_j$'s the component-specific variances in a Gaussian mixture model.
Note that, in accordance with previous literature on repulsive mixtures, we assume that the component-specific variances $V_i$'s are i.i.d. and that interaction is on the locations $Y_i$'s.

Let $M$ be the number of support points in $\mutilde$. It is convenient to introduce auxiliary component indicator variables $C_i$, $i=1, \ldots, n$  taking values in $\{1, \ldots, M\}$, such that $\prob(C_i = h \mid \mutilde) \propto S_h$.
By noticing that $(Y_i, V_i)_{i \geq 1} = (X_{C_i}, W_{C_i})_{i \geq 1}$,
\eqref{eq:mixture} is equivalent to
\begin{equation}\label{eq:mixture_clus}
\begin{aligned}
    Z_i \mid \mutilde, C_i & \ind f(\cdot \mid X_{C_i},  W_{C_i}), \qquad i=1, \ldots, n \\
    C_i \mid \mutilde & \iid \mbox{Categorical}(S_1 / S_\bullet, \ldots, S_M / S_\bullet) \\
    \mutilde & \sim \plaw_\mu
\end{aligned}
\end{equation}
where $S_\bullet := \sum_{j = 1}^M S_j$.
Following the standard terminology of \cite{ArDeInf19} and \cite{griffin2011posterior}, we define 
$\bm{S}^{(a)} = \{S^{(a)}_1, \ldots, S^{(a)}_{K_n}\}$ as the distinct values in $\{S_{C_i}\}_{i \geq 1}$ and refer to them as \emph{active} jumps.
Moreover, we set the \emph{non-active} jumps $\bm { S}^{(na)} := \{S_1, \ldots, S_M\} \setminus \bm{ S}^{(a)}$.
Analogously, we define $\bm{X}^{(a)}, \bm{ X}^{(na)}$ and $\bm{ W}^{(a)}, \bm{ W}^{(na)}$ and refer to them as the active and non-active atoms, respectively.
Note that, according to our notation, $\bm X^{(a)}$ coincide with the unique values in the latent variables $Y_1, \ldots, Y_n$ in \eqref{eq:mixture}.

In the rest of this section, we describe two Markov chain Monte Carlo algorithms for posterior inference under model \eqref{eq:mixture}. These are based on the posterior characterization of $\mutilde$ given $Y_1, \ldots, Y_n$ in \Cref{teo:post} and the predictive distribution of $Y_{n+1}$ in \Cref{teo:pred_conditional}.
Following \cite{papaspiliopoulos2008retrospective}, we term them \emph{conditional} and \emph{marginal} algorithms, respectively. In fact, in the former, the random measure $\mutilde$ is part of the
state space of the algorithm, while $\mutilde$ is integrated out in the latter.
We assume that, conditionally to the number of points $M$, the distribution of the vector $(X_1, \ldots, X_M) \in \X^M$ defining the support of $\Phi$ has a density. 
With an abuse of notation, we denote this density by $f_\Phi$;  the existence of the density of the point process $\Phi$  guarantees that the  $f_\Phi$ above is indeed proportional to the density of the point process itself \citep[see][for more details]{MoWaBook03}.

\subsection{A conditional MCMC algorithm}\label{sec:cond_algo}

Theorem~\ref{teo:post} can be easily rephrased to encompass the case of a point process $\Phi$ with i.i.d. marks $(W_j)_{j \geq 1}$: this is useful to derive the full-conditional of $\mutilde$ given $C_1, \ldots, C_n, \bm{ X}^{(a)}, \bm{ W}^{(a)}$ and $U_n$.
\begin{corollary}\label{cor:post_mix}
    Consider the model \eqref{eq:mixture_clus}, and define $n_h = \sum_{i=1}^n \indicator_{h}(C_i)$. Then, conditionally on $C_1, \ldots, C_n, K_n = k, \bm{X}^{(a)} = \bm x^{(a)}, \bm{W^{(a)}} = \bm{w}^{(a)}$ and $U_n =u$, $\mutilde$ is equal in distribution to
    \[
        \sum_{h=1}^k S^{(a)}_h \delta_{(x^{(a)}_h, w^{(a)}_h)} + \mutilde^\prime
    \]
    where $S^{(a)}_h \sim f_{S^{(a)}_h}(s) \propto s^{n_h} \e^{-us} H(\dd s)$, and $\mutilde^\prime := \sum_{h \geq 1} S^{(na)}_h \delta_{( X^{(na)}_h,  W^{(na)}_h)}$ is such that $\mutilde^\prime_X := \sum_{h \geq 1} S^{(na)}_h \delta_{X^{(na)}_h}$ has Laplace functional 
    \[
         \E\left[\exp \int_\X - f(z) \mutilde^\prime_X(\dd z) \right] = \frac{\E \left[ \exp \left\{- \int_\X (f(z) + u) \mutilde^!_{\bm x^{(a)}}(\dd z) \right\} \right]}{\E \left[ \exp \left\{- \int_\X u \mutilde^!_{\bm x^{(a)}}(\dd z) \right\} \right]},
    \]
    where $\mutilde^!_{\bm x^{(a)}}$ is as in \eqref{eq:palm_mu} for $\bm x = \bm x^{(a)}$, and $ W^{(na)}_h \iid f_W$.
\end{corollary}

Corollary \ref{cor:post_mix}  leads to the algorithm below, where we write  ``$\cdot \mid \text{rest}$'' to mean that we are conditioning with respect to all the variables but the ones appearing on the left-hand side of the conditioning symbol.

\begin{enumerate}
    \item Sample $U_n \mid \text{rest} \sim \mbox{Gamma}(n, S_\bullet)$, where $S_\bullet := \sum_{j=1}^M S_j$.
    
    \item Sample each $C_i$ independently from a discrete distribution over $\{1, \ldots, M\}$ such that
    \[  
        \prob(C_i = h \mid \text{rest}) \propto S_h f(Z_i \mid X_h, W_h).
    \]
    Let $k$ be  the number of unique values in $\bm X^{(a)} := \{X_{C_i}, \ i=1, \ldots, n\}$.
    Define $\bm W^{(a)}$ and $\bm S^{(a)}$ analogously.
    
    \item\label{item:complex} Sample $\mutilde$ using the distribution in Corollary \ref{cor:post_mix}.  For the \emph{active} part, sample each $S^{(a)}_h$ independently from a distribution on $\R_+$ with density
    \[
        f_{S_h}(s) \propto \e^{-U_n s} s^{n_h} H(\dd s).
    \]
    For the \emph{non-active} part, first sample 
    $\mu'=\sum_{j \geq 1} S^{(na)}_j \delta_{X^{(na)}_j}$ from the law of a random measure with Laplace transform \eqref{eq:Laplace_mu_post}. See below for some guidelines on how to perform this step.
    Then, sample the corresponding marks $W^{(na)}_j \iid f_W$.
    
    \item Let $\bm X^{(a)} = \{X^{(a)}_1, \ldots, X_k^{(a)}\}$ where $\tilde X^{(a)} = (X^{(a)}_1, \ldots, X_k^{(a)})$ is sampled from the density on $\X^k$ given by
    \[
        \prob(\tilde X^{(a)} \in \dd \bm x \mid \text{rest}) \propto f_{\Phi}(\bm x, \bm X^{(na)}) \prod_{h=1}^k \prod_{i: C_i = h} f(Z_i \mid x_h, W_h), \quad \bm x \in \X^k.
    \]

    \item Sample each entry in $\bm W^{(a)}$ independently from a density on $\R_+$ proportional to
    \[
        f_W(w) \prod_{i: C_i = h} f(Z_i \mid X^{(a)}_h, w).
    \]

    \item Set $\bm X = \bm X^{(a)} \cup X^{(na)}$, $\bm W = \bm W^{(a)} \cup \bm W^{(na)}$ $\bm S = \bm S^{(a)} \cup \bm S^{(na)}$ and $m$ equal to the cardinality of these sets.
\end{enumerate}

Among the steps in the algorithm, the most complex is the third, which involves sampling the random measure $\mutilde^\prime$ with a given Laplace transform. All the remaining ones can be handled either 
by closed-form full-conditionals (depending, for instance, on the law of the jumps $H(\dd s)$ and the prior $f_W$) or by simple Metropolis-Hastings steps, yielding a Metropolis-within-Gibbs
algorithm.
To sample $\mutilde^\prime$,  we need to sample sequentially as follows: ($3.i$) the support points $\Phi^\prime = \sum_{j \geq 1} \delta_{X^{(na)}_j}$, ($3.ii$) the jumps $\{ S^{(na)}_j\}$, and ($3.iii$) the marks $W^{(na)}_j \iid f_W$, which is trivial.  Step ($3.i$) requires simulating a point process, and we refer to the literature on this topic, e.g., via perfect sampling \citep[see Algorithm 11.7 in][]{MoWaBook03} or via a birth-death Metropolis-Hastings step \citep[see Algorithm 11.3 in][]{MoWaBook03}.  A particular case of our Algorithm 1-6, when $\Phi$ is a Gibbs point process with a density, has been designed in 
 \cite{beraha21}. 
Step
 ($3.ii$) needs to iid sample from the tilted density 
 $f_{S_j} \propto \e^{-Us} H(\dd s) $. When $H(\dd s)$ is a gamma density, this is straightforward; for different $H(\dd s)$, we refer to \cite{ArDeInf19}.
On the other hand, when considering a DPP prior, we can use Algorithm 1 in \cite{Lav15} \MBtext{(adapted from \citealp{Hough:etal:06})} to obtain a perfect sample from the law of $\bm X^{(na)}$.

\subsection{A marginal MCMC algorithm}

When integrating out $\mutilde$ from \eqref{eq:mixture}, Theorem~\ref{teo:pred_conditional} can be used to devise a marginal MCMC strategy, i.e., an instance of \emph{collapsed} Gibbs sampler.
In such an algorithm, we alternate between updating the observation-specific parameters $\{(Y_i, V_i)\}_{i=1}^n \mid U_n, Z_1, \ldots, Z_n$ and sampling $U_n \mid \{(Y_i, V_i)\}_{i=1}^n, Z_1, \ldots, Z_n$, which is as \eqref{eq:post_u}. To sample $\{(Y_i, V_i)\}_{i=1}^n \mid U_n, Z_1, \ldots, Z_n$ we perform a Gibbs scan, whereby we sample $(Y_i, V_i)$ from its full conditional distribution, i.e. $(Y_i, V_i) \mid (\bm Y, \bm V)_{-i}, U_n, Z_1, \ldots, Z_n$, where $ (\bm Y, \bm V)_{-i} = (Y_1, V_1), \ldots, (Y_{i-1}, V_{i-1}), (Y_{i+1}, V_{i+1}), (Y_{n}, V_{n})$. It can be easily seen \citep{neal2000markov} that $\prob((Y_i, V_i) \in \dd y \,  \dd v \mid (\bm Y, \bm V)_{-i}, U_n, Z_1, \ldots, Z_n) \propto \prob(Y_i \in \dd y \mid \bm Y_{-i}, U_n) f_W(v) \dd v f(Z_i \mid y, v)$ where  $\prob(Y_i \in \dd y \mid \bm Y_{-i}, U_n)$ is the conditional prior law of $Y_i$. In particular, by exchangeability, for any $\bm y_{-i} \in \X^{n-1}$ the conditional prior for $Y_i$ satisfies $\prob(\bm Y_i \in \dd y \mid \bm Y_{-i} = \bm y_{-i}, U_{n}) = \prob(\bm Y_n \in \dd y \mid \bm Y_{-n} = \bm y_{-i}, U_{n})$, where the latter is the law described Theorem~\ref{teo:pred_conditional}.
The plain application of this result yields the following MCMC algorithm.
\begin{enumerate}
    \item Sample $U_n \mid \text{rest}$ from the full conditional distribution in \eqref{eq:post_u}.
    \item For each observation, sample $(Y_i, V_i)$ from
    \begin{multline*}
        \prob((Y_i, V_i) \in (\dd y \, \dd v ) \mid \text{rest}) \propto \sum_{j=1}^k \frac{\kappa(u, n^{(-i)}_j + 1)}{\kappa(u, n^{(-i)}_j)}  f(Z_i \mid Y_j^*, V^*_j) \delta_{(Y_j^*, V^*_j)}(\dd y \, \dd v) + \\
        \kappa(u, 1) \frac{\E \left[ \e^{-  u \mutilde_{\bm y^{*(-i)}, y}^!(\X)} \right]}{\E \left[ \e^{- u \mutilde_{\bm y^{*(-i)}}^!(\X)} \right]} \frac{m_{\Phi^{k+1}}(\bm y^{*(-i)}, y)}{m_{\Phi^{k}}( \bm y^{*(-i)})} f(Z_i \mid y, v) P_0(\dd y) f_W(v) \dd v , 
    \end{multline*}
    where the superscript $(-i)$ means that the $i$-th observation is removed from the state for the computations. Here, $\bm Y^*$ and $\bm V^*$ are the vectors of unique values in $(Y_1, \ldots, Y_n)$ and $(V_1, \ldots, V_n)$ respectively.
    \item Sample the unique values $\bm Y^*$ from a joint distribution on $\X^k$ proportional to
    \[
        \prob(\bm Y^* \in \dd \bm y^*) \propto m_{\Phi^k}(\bm y^*) \prod_{h=1}^k \prod_{\{i:\,  Y_i = Y^*_h\}} f(Z_i \mid y^*_h, V^*_h)P_0^k(\dd \bm y^*), \quad \bm y^* \in \X^k.
    \]
    \item 
    Sample each of unique values $\bm V^*$ independently from
    \[
        \prob(V^*_h \in \dd v^*) \propto f_W(v^*) \prod_{i: C_i = h} f(Z_i \mid Y^*_h, v^*) \dd v^*, \quad v^* \in W.
    \]
\end{enumerate}
Steps 3 and 4 of the above are not strictly necessary to produce an ergodic MCMC sampler that targets the posterior distribution. However, it is common practice in MCMC for mixture models to add these steps \MBtext{as it has been empirically demonstrated that they speed up the convergence of the algorithm and improve its mixing, see, e.g, \cite{neal2000markov}}.
Step 2 of the algorithm above requires the computation of
\[
    \int_{\X \times \W} \frac{m_{\Phi^{k+1}}(\bm y^{*(-i)}, y)}{m_{\Phi^{k}}(\bm y^{*(-i)})} f(Z_i \mid y, v) P_0(\dd y) f_W(v) \dd v
\]
which might be challenging in situations where data are multidimensional.
However, \cite{xie2019bayesian} shows how this can be overcome using numerical quadrature.
Moreover, in the higher dimensional setting, we can adapt the strategy devised by Neal in his Algorithm 8 by introducing $L$ auxiliary variables and replacing Step 2 with:
\begin{enumerate}
    \item[2'.a ] For $\ell = 1, \ldots, L$, sample $Y^*_{k + \ell}$ from 
    \[
        \prob(Y^*_{k + \ell} \in \dd y \mid \text{rest}) \propto \frac{m_{\Phi^{k+1}}(\bm y^{*(-i)}, y)}{m_{\Phi^{k}}( \bm y^{*(-i)})} P_0(\dd y)
    \]
    and $V^*_{k + \ell} \iid f_W$.
    \item[2'.b ] Set $(Y_i, V_i)$ equal to $(Y^*_h, V^*_h)$ with probability proportional to
    \begin{equation*}
    \begin{cases}
        \frac{\kappa(u, n^{(-i)}_j + 1)}{\kappa(u, n^{(-i)}_j)}  f(Z_i \mid Y_h^*, V^*_h), & \text{ for } h=1, \ldots, k \\
        \frac{1}{L} \kappa(u, 1) \frac{\E \left[ \exp(- u \mutilde_{\bm y^{*(-i)}, Y^*_k}^!(\X)) \right]}{\E \left[ \exp(- u \mutilde_{\bm y^{*(-i)}}^!(\X)) \right]} f(Z_i \mid Y_h^*, V^*_h), & \text{ for } h=k+1, \ldots, k + L.
    \end{cases}
    \end{equation*}
   The computation of the ratio of expected values can be challenging if $\Phi$ is a DPP; see Example \ref{ex:dpp2} for further details, while it poses no problems when $\Phi$ is a SNCP.
\end{enumerate}

\subsection{Shot-Noise Cox Process mixtures as mixtures of mixtures}\label{sec:sncp_mixture_of_mixture}

Let us draw an interesting connection between a mixture model driven by a SNCP and the \textit{mixtures of mixtures} model, whereby the density of each mixture component is approximated by a mixture model.
This model was previously considered in \cite{WalliMix}, who proposed a Bayesian methodology and \cite{aragam2020identifiability}, which studied identifiability and estimation within a frequentist setting.

First, recall from \Cref{ex:sncp_def} that we can equivalently write $\Phi \mid \Lambda = \sum_{h=1}^{n(\Lambda)} \Phi_h$, where, with our notation, each $\Phi_h$ is a \emph{group}.
Observe that, for appropriate choices of $k_\alpha$, the points in a \emph{group} 
will be closer than points belonging to different \emph{groups}. 
When we embed a random probability measure built from \eqref{eq:shot-noise} in a Bayesian mixture model as done at the beginning of Section~\ref{sec:mix}, we obtain that the atoms of the mixture are randomly \emph{grouped} together.
Hence, we rewrite the random population density as
\[
  f(z) = \frac{1}{S_\bullet} \sum_{j \geq 1} S_j f(z \mid X_j, W_j) \equiv \frac{1}{S_\bullet} \sum_{h=1}^{n(\Lambda)} \sum_{j \geq 1}  S_{h, j} f(z \mid  X_{h, j},  W_{h, j}).
\]
On the right-hand side, we first sum over the atoms of $\Lambda$ and then over the atoms of each of the \emph{groups} $\Phi_j$'s. With an abuse of notation, we have introduced a second subscript to $S_h$, $X_h$, and $W_h$ to keep track that each point of $X_h$ (and its marks) is assigned to a point in $\Lambda$. It is possible to pass from the single-index summation to the double-index thanks to the variables $T_j$ introduced in \Cref{ex:sncp_def}.
Let us define
\begin{equation}\label{eq:sncp_comp}
  \tilde f_h(z) := \frac{1}{P_h} \sum_{j \geq 1} S_{h, j} f(z \mid  X_{h, j},  W_{h, j}), \ \textrm{where } P_h := \sum_{j \geq 1 } S_{h, j},  
\end{equation}
which is a (random) probability density function over $\X$.
We clearly see that the random population density is equal to $f(z) = S_\bullet^{-1} \sum_{h=1}^{n(\Lambda)} P_h \tilde f_h(z)$.
Hence, a SNCP mixture model can be written as a \emph{mixture of mixtures}, where each component $\tilde f_h(z)$ is expressed as a mixture model itself.
Therefore, we can consider the SNCP mixture model as a nonparametric generalisation of the \emph{mixtures of mixtures} model in \cite{WalliMix}.

Observe that the SNCP mixture induces a two-level clustering. 
At the first level, the parameters $(Y_i, V_i)$ in \eqref{eq:mixture} induce the standard partition of the observations by the equivalence relation $i \sim j$ if and only if $(Y_i, V_i) = (Y_j, V_j) = (X_h, W_h)$ for some $h$. Let $C_i = h$  
if and only if $(Y_i, V_i) = (X_h, W_h)$.
Then, at the second level, each point $(X_h, W_h)$ can be referred to as one of the \emph{groups} $\Phi_j$ thanks to the variables $T_j$ introduced in \Cref{ex:sncp_def}.
Hence, we can refer observations to the \emph{groups} using the variables $(T_{C_1}, \ldots, T_{C_n})$. In the next subsection, we call the partition induced by the $T_{C_i}$'s as the \emph{grouping} of the observations.

\section{Numerical illustrations}
\label{sec:numeric}

Here, we consider two simulated scenarios and a real dataset to highlight the core differences between a mixture model with repulsive, i.i.d., or ``attractive'' atoms.
We focus on mixtures of Gaussian kernels, i.e., $f$ in \eqref{eq:mixture} is the Gaussian density where parameters $(Y_i, V_i)$  are the mean and the variance, respectively.
We compare posterior inference under three prior specifications for the measure $\mutilde$ in \eqref{eq:mixing_meas}; in all cases, the unnormalised weights are $S_j\iid \mbox{Gamma}(2,2)$.
Under the first model, the cluster centres $\{X_j\}_{j \geq 1}$ follow a DPP with kernel $K(x, y) = \rho \exp(-(x - y)^2/\alpha^2)$ for $\rho = 2$ and $\alpha= \sqrt{\pi} / 2$, paired with inverse-Gamma density $f_W$ for the prior of the $W_j$'s with both shape
and scale parameters equal to two. See Appendix \ref{app:numerical_examples} for further details on the DPP prior considered and Appendix \ref{app:dpp_density} for details on the computation of the DPP density with respect to a suitable Poisson process. The second model we consider is the normalised IFPP mixture model by \cite{ArDeInf19}, which corresponds to assuming that $\{(X_j, W_j)\}_{j=1}^K \mid K$ are i.i.d. from a Normal-inverse-Gamma density (i.e., $X_j|W_j\sim \mathcal{N}(0, 10 W_j)$ and $1/W_i\sim \mbox{Gamma}(2,2)$),  and we further assume $K-1 \sim \mbox{Poi(1)}$.
Our third choice of the marginal prior for the $X_j$'s is the SNCP process,  where $k_\alpha(\cdot)$ is the Gaussian density centred at the origin with standard deviation $\alpha$. Unless otherwise specified, we fix $\alpha=1$. We also fix $\gamma = 1$ and assume that the $W_j$'s are distributed as in the DPP case.
Prior elicitation for the SNCP mixture is reported in greater detail in Appendix \ref{app:numerical_examples}.
Moreover, we also discuss there the robustness of posterior inference with respect to hyperparameters in the case of a DPP and a SNCP prior, as well as the inherent difficulty of learning the repulsive parameters via a fully-Bayesian approach, i.e., assuming the parameters controlling the kernel of the DPP random and updating them as part of the MCMC algorithm.

Posterior inference for the model in \cite{ArDeInf19} is addressed via the \texttt{BayesMix} library \citep{bayesmix}.
To fit the SNCP mixture, we use the conditional algorithm in Section~\ref{sec:cond_algo}, where we further add to the MCMC state the points of $\Lambda$; see \Cref{app:scnp_algo} for further details. 
As for the repulsive mixture based on DPP, we use the conditional algorithm described in Section~\ref{sec:cond_algo}. See also \cite{beraha21} for further details.
A Python implementation of the different MCMC algorithms is available at \url{https://github.com/mberaha/interacting_mixtures}.
We run the MCMC algorithms for a total of $100,000$ iterations, discarding the first $50,000$  and keeping one every five iterations for a final size of $1,000$.
In the sequel, we compare the \emph{grouping} of the datapoints induced by the SNCP mixture, as defined above, with the standard clustering resulting from the use of  DPP and normalised IFPP mixtures; we also compare density estimation under the three models.

\subsection[Data from a mixture of t distributions]{Data from a mixture of $t$ distributions}\label{sec:simu1}

In the first scenario, we generate $200$ data points from a two-component mixture of univariate Student's $t$ distributions with three degrees of freedom, centred respectively in $-5$ and $+5$.
Posterior inference is summarised in Figure~\ref{fig:simu1}.
Both the DPP and SNCP mixtures identify two clusters, with the difference that the SNCP mixture activates between 25 and 40 Gaussian components, while the DPP only has two.
On the other hand, the IFPP mixture has between 3 and 8 active components. The co-clustering matrices for the IFPP and SNCP mixtures are shown in Figure~\ref{fig:simu1}, and the one associated with the DPP mixture is identical to that of the SNCP.
We observe that the IFPP model erroneously identifies three clusters:  the two large clusters, we would expect, and a third cluster, which collects the observations at the centre of the domain.
Surprisingly, this last cluster also contains some observations at the farthest right of the domain: this can be explained by the presence of a mixture component located near zero and with an extremely large variance, which produces an overabundant estimated cluster that
groups together data that are far apart.
All the models produce a satisfactory density estimate. However, as it is to be expected, the DPP mixture fails to capture the heavy tails of the $t$ distributions both at the centre of the domain and at the left and right extremities, resulting in a slightly poorer fit.

\begin{figure}
    \centering
    \begin{subfigure}[T]{0.24\textwidth}
        \centering
        \includegraphics[width=\textwidth]{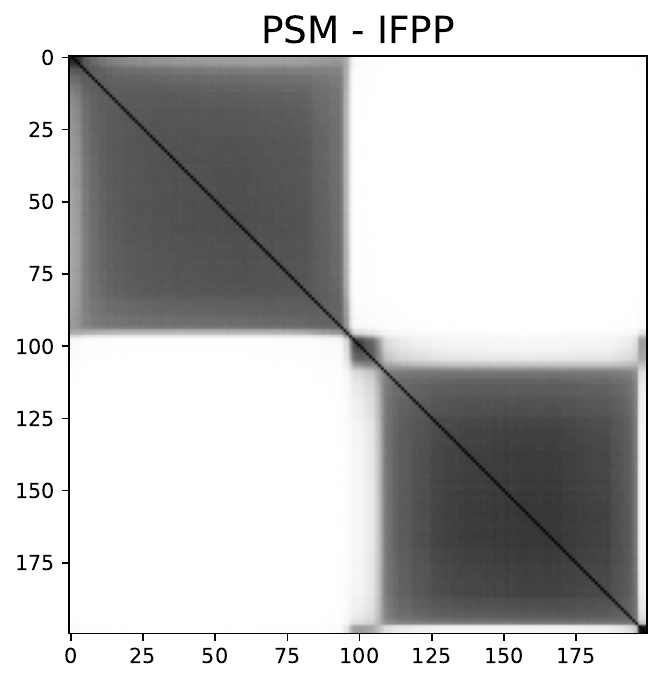}
        \caption{}
    \end{subfigure}
    \begin{subfigure}[T]{0.24\textwidth}
        \centering
        \includegraphics[width=\textwidth]{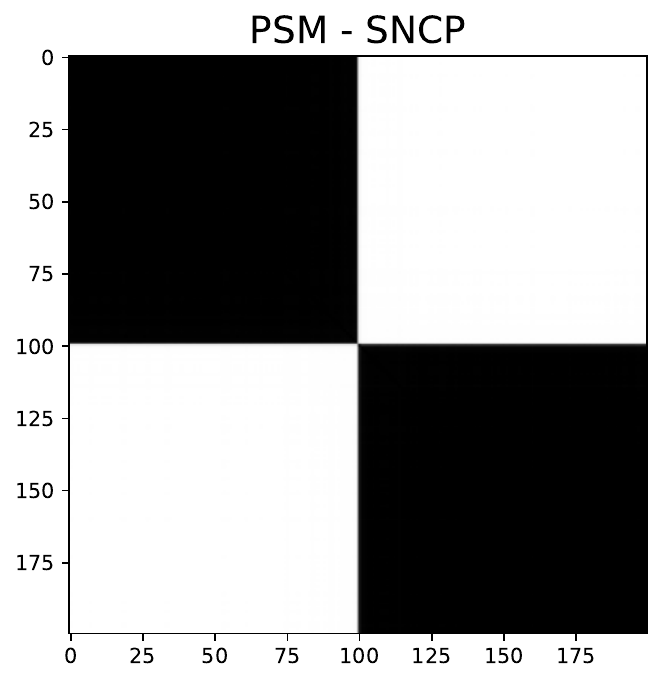}
        \caption{}
    \end{subfigure}
    \begin{subfigure}[T]{0.24\textwidth}
        \centering
        \includegraphics[width=\textwidth]{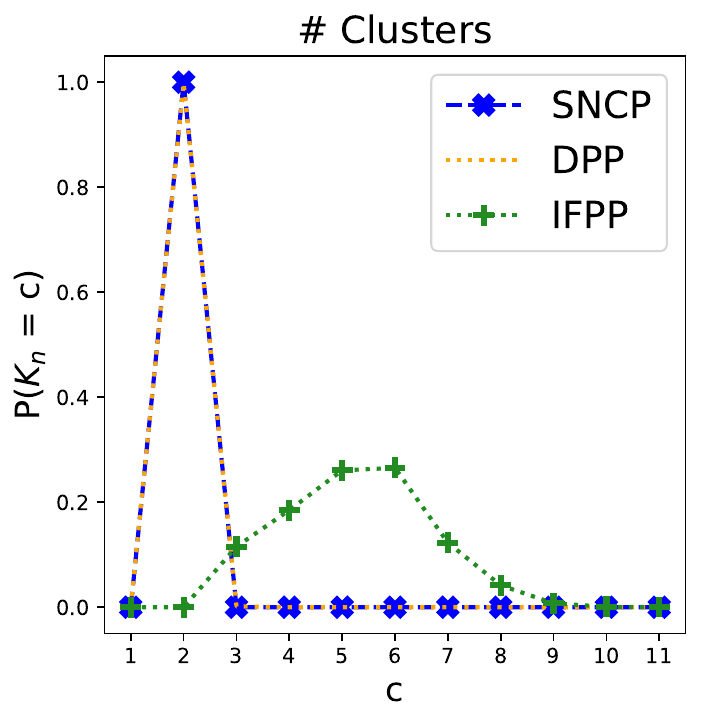}
        \caption{}
    \end{subfigure}
    \begin{subfigure}[T]{0.24\textwidth}
        \centering
        \includegraphics[width=\textwidth]{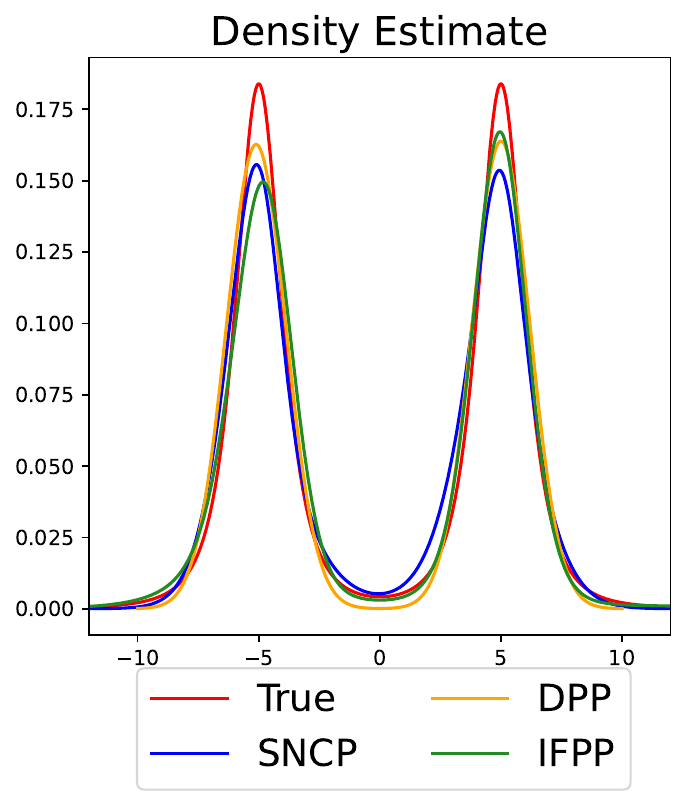}
        \caption{}
    \end{subfigure}
    \caption{From left to right: posterior co-clustering matrix under the normalised IFPP (a) and SNCP (b) mixtures, posterior distribution of the number of clusters (c), density estimates (d) for the simulated example in Section~\ref{sec:simu1}. In the posterior co-clustering matrices, data are sorted from the smallest to the largest.}
    \label{fig:simu1}
\end{figure}

\subsection{Data from a contaminated mixture}\label{sec:simu2}

Our second simulation follows the setup in \cite{miller2018robust}. We assume that a ``true'' data generating process with density $f_0 = 0.25 \mathcal N(-3.5, 0.8^2) + 0.3 \calN(0, 0.4^2) + 0.25 \calN(3, 0.5^2) + 0.2 \calN(6, 0.5^2)$ has been corrupted by some noise, and we observe data distributed as 
\begin{equation*}
        Z_i \mid \tilde p \iid \tilde f(\cdot) = \int_{\R} \mathcal N(\cdot \mid \theta, 0.25^2) \tilde p (\dd \theta), \qquad \tilde p \sim DP(500 f_0)
\end{equation*}
where $DP(500 f_0)$ denotes the Dirichlet process with total mass parameter equal to $500$ and centering probability measure induced by $f_0$.
We interpret $\tilde f$ as a perturbation of $f_0$ so that the goal is then to recover $f_0$ from the observed data.

Figure~\ref{fig:simu2} shows the posterior distribution for the number of clusters and the density estimates when $n=500$. The density estimate under the IFPP mixture is essentially identical to the one under the SNCP and is, therefore, omitted.
The SNCP mixture correctly identifies 4 clusters, while the DPP mixture chooses between 4 and 5 clusters (the point estimate of the partition agrees on 4 clusters) and the IFPP mixture between 4 and 7 clusters.
Hence, although there is no ``true'' number of clusters in the data-generating process, we can argue that the SNCP model produces the most parsimonious and interpretable clustering.
On the other hand, it is clear that the density estimate under SCNP is somewhat worse than under the DPP prior. Indeed, the SNCP mixture picks up all the noise induced by the perturbation.

\begin{figure}
    \centering
    \begin{subfigure}{0.35\textwidth}
        \centering
        \includegraphics[width=\textwidth]{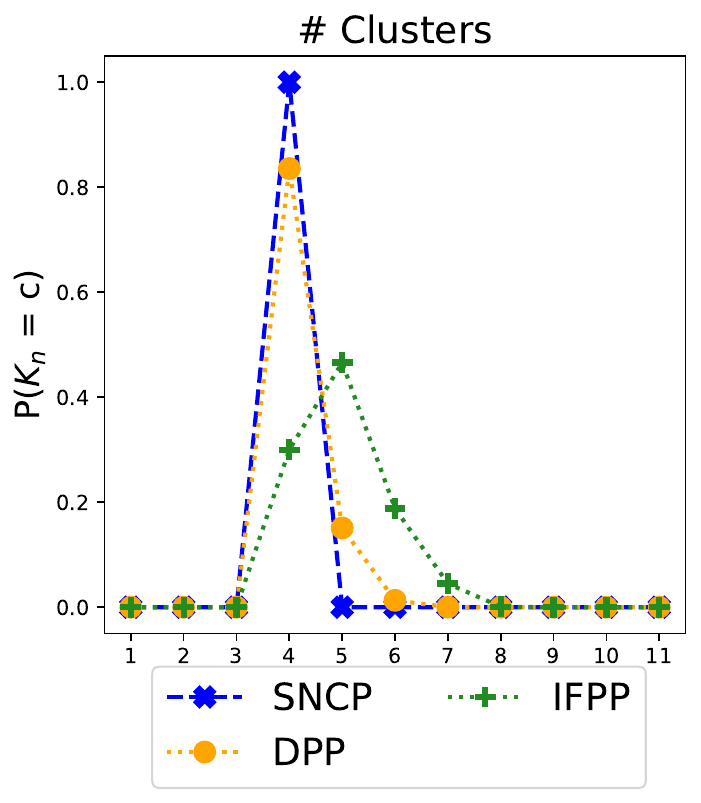}
        \caption{}
    \end{subfigure}
    \begin{subfigure}{0.35\textwidth}
        \centering
        \includegraphics[width=\textwidth]{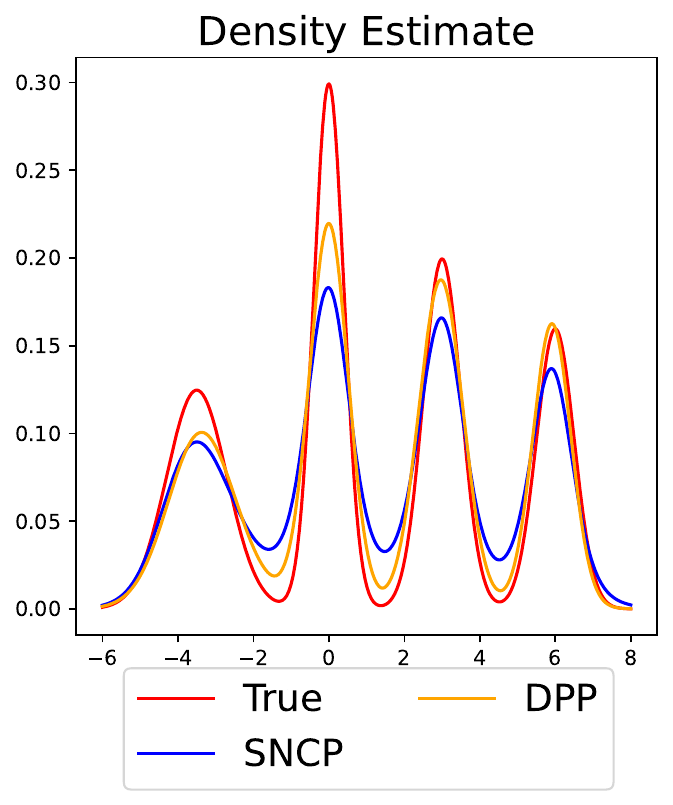}
        \caption{}
    \end{subfigure}
    \caption{Posterior distribution of the number of clusters (a) and density estimates (b) for the simulated example in Section~\ref{sec:simu2}.}
    \label{fig:simu2}
\end{figure}

\subsection{Shapley Galaxy Data}

We consider now a real dataset containing the velocities of $n=4206$ galaxies (measured in $10^5$ km/s) in the Shapley supercluster, available in the \texttt{R} package \texttt{spatstat}. This can be seen as a modern and improved version of the popular galaxy dataset \citep{roeder1990density}.
We compare the density and cluster estimates when using a subsample of $n=100, 500, 2000$ datapoints as well as the whole dataset.

We set the prior {distribution via an empirical Bayes approach; further details and the numerical values chosen} are reported in Appendix~\ref{app:galaxy}.
Posterior inference is summarised in Figure~\ref{fig:galaxy}.
We notice how under all three models, the estimated number of clusters grows with the sample size; however, while under the SNCP and DPP mixtures, it stabilises around $7$ clusters at most, under the IFPP prior, when $n=4206$, the model induces more than 16 clusters for some iterations of the MCMC algorithm. However, remember that for the SNCP mixtures we report the estimated \emph{groups}. 
The density estimates are consistent with the discussion in the examples above. The IFPP and SNCP mixtures induce very similar density estimates, while the DPP mixture tends to oversmooth the density in some regions to obtain well-separated components.
In this particular example, the density estimate under a SNCP prior reflects the empirical histogram carefully, picking up even very subtle ``bumps'' in the density, such as the one near the value $33$ in the $x$-axis. Since the measurements of galaxy velocities are typically accurate (especially at the scale we are considering), we can recommend using the SNCP mixture over the IFPP model for density estimation in this case.

\begin{figure}[t!]
    \centering
    \includegraphics[width=0.32\linewidth]{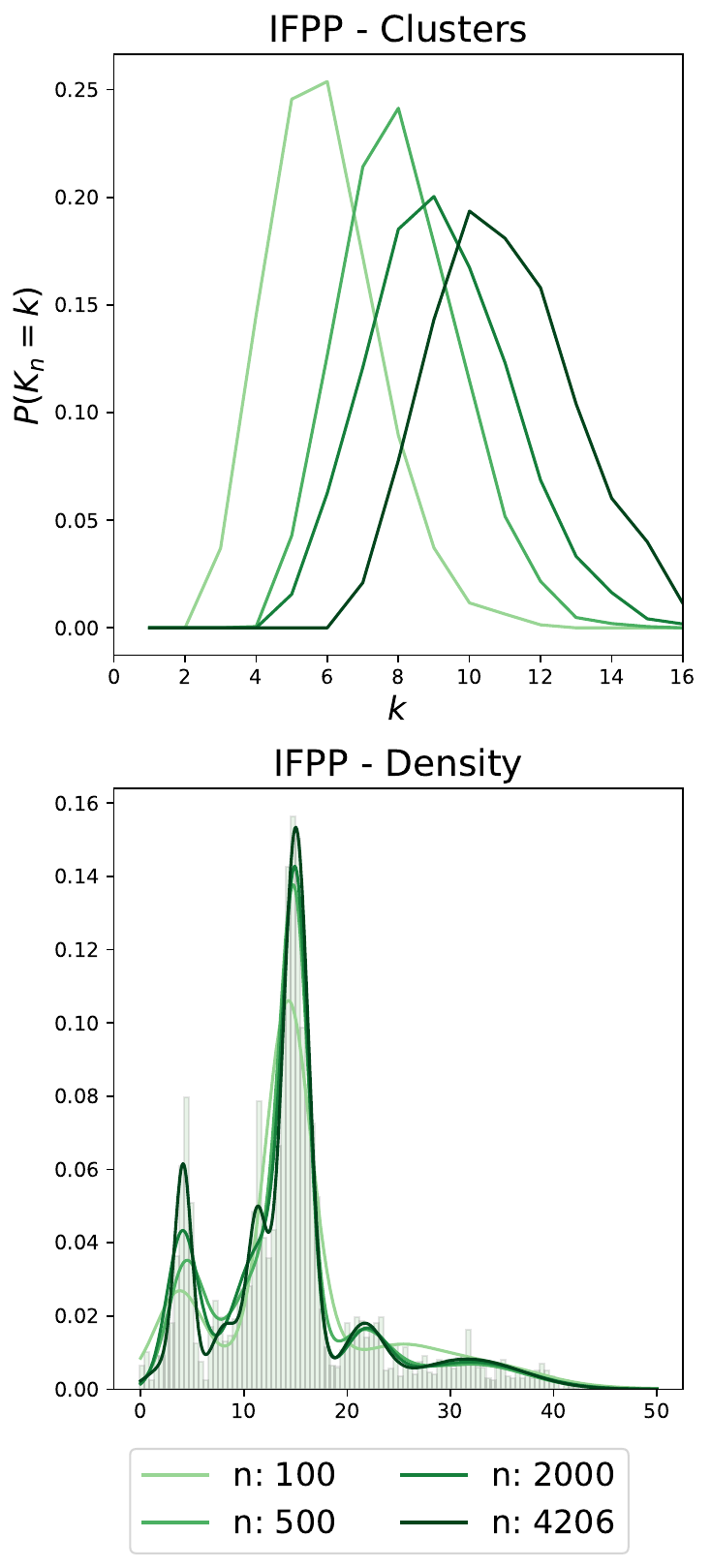}%
     \includegraphics[width=0.32\linewidth]{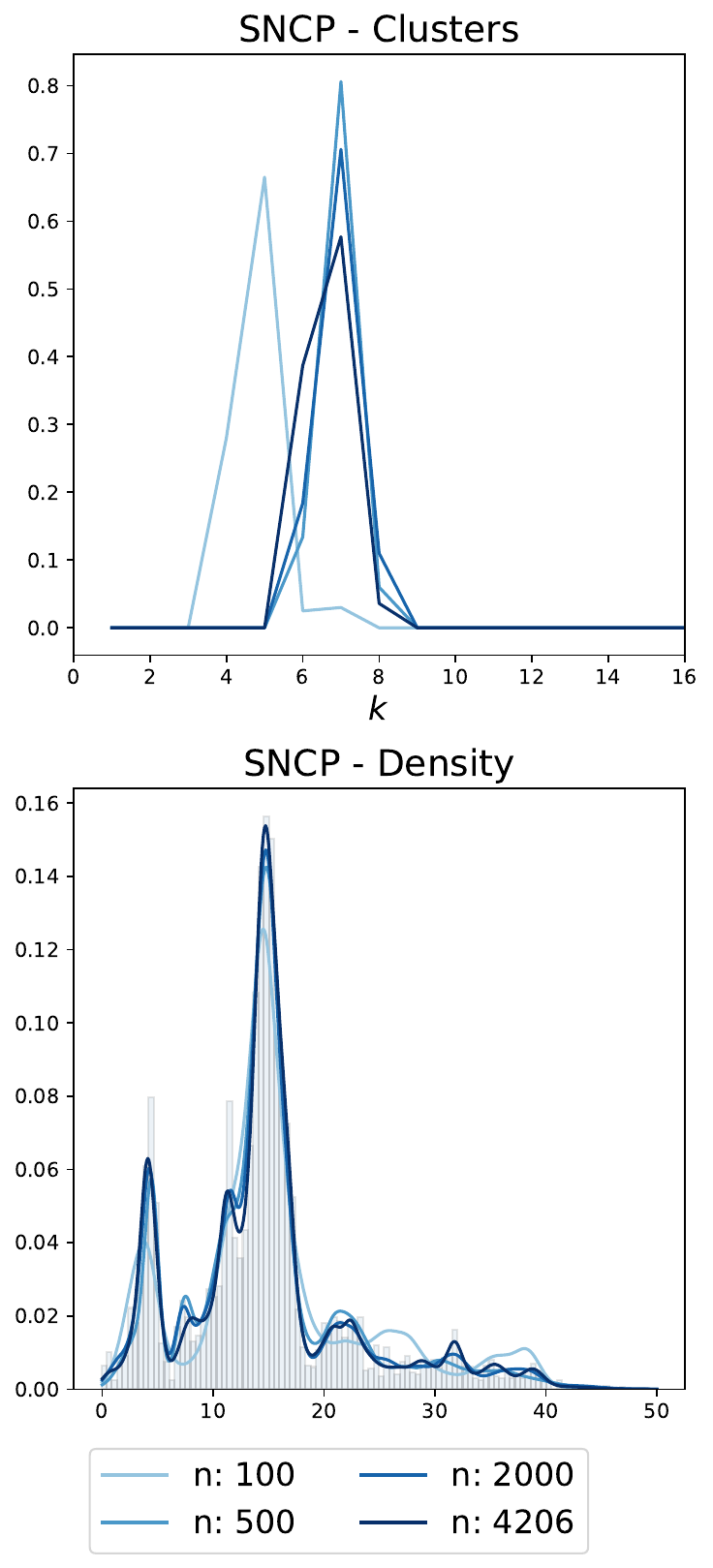}%
     \includegraphics[width=0.32\linewidth]{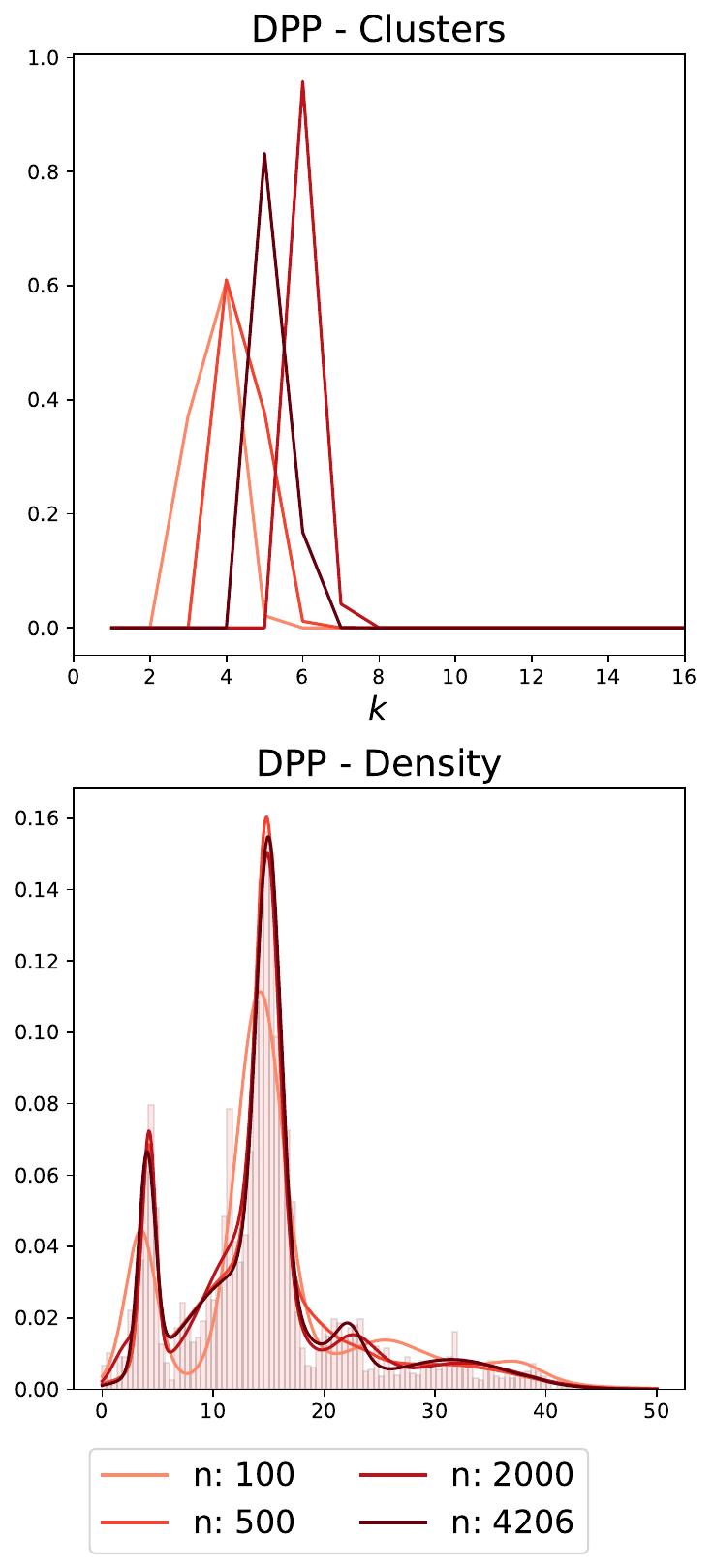}%
    \caption{Analysis of the Shapley galaxy data under the three different models.}
    \label{fig:galaxy}
\end{figure}

We remark here that a crucial parameter for the SNCP mixture is the scale $\alpha$, which controls how spread a component $\tilde f_j$ (defined in \eqref{eq:sncp_comp}) can be.
Indeed, if $\alpha$ is too large, the model will tend to group all the normal components in one single large $\tilde f_j$, thereby estimating only one cluster.
On the other hand, if $\alpha$ is too small, each $\tilde f_j$ is more likely to correspond to a single normal component, thus producing too many clusters. 
In our examples, we did not notice a particular sensibility to this choice for both density estimation and clustering, as long as the values are chosen in an appropriate range.
In Appendix \ref{app:alpha_choice}, we report a sensitivity analysis for the choice of $\alpha$ using the Shapley galaxy data.

\section{Discussion and future works}
\label{sec:discussion}

In this work, we have investigated discrete random probability measures with interaction across support points.
Our study is motivated by the recent popularity of repulsive mixtures 
\citep{PeRaDu12, xu2016bayesian,  bianchini2018determinantal, quinlan2017parsimonious, beraha21, cremaschi2023repulsion} 
for Bayesian model-based clustering. 

The first main contribution of the paper is to propose a general construction of RPMs via the normalisation of marked point processes with repulsive or attractive support points, which can be embedded in a hierarchical mixture model, as shown in Section~\ref{sec:mix}.
The second contribution of this paper is to present a unified mathematical theory by encompassing both repulsive and attractive mixtures, as well as the IFPP mixtures by \cite{ArDeInf19}.
Our results also extend the well-known distributional properties of NRMIs \citep{ReLiPr03}.
Due to the generality of our framework, all the proofs rely on new arguments based on Palm calculus; see \cite{BaBlaKa}.
Thirdly, we discuss the use of shot-noise Cox processes, which were not previously considered in connection with Bayesian mixture models beyond \cite{wang2022spatiotemporal}. 
As a further contribution, we include the applicability of the MCMC algorithms proposed in Section~\ref{sec:mix}
to any choice of the point process prior.
However, they might lack efficiency, especially for moderate or high-dimensional data. More adequate algorithms could be devised for specific classes of point processes; see, for instance, \cite{sun2022bayesian} for an alternative MCMC scheme tailored to the class of Mat\'ern point processes.

Our work paves the way for several contributions with both theoretical and computational flavour. We discuss some ideas below.

\subsection{The trade-off between density and cluster estimates, and related asymptotics}\label{sec:tradeoff}

Consider the setting of a well-specified mixture model, i.e., the data come from a finite mixture model whose kernel agrees with our modelling choices. 
Then, for standard mixtures whose atoms are i.i.d. from a base distribution, a well-established theory ensures convergence of the latent mixing measure (in an appropriate sense) to the true data-generating process under suitable conditions. See, e.g., \cite{guha2021posterior}.
In this setting, it is easy to see that the posterior of repulsive mixtures also converges to the true mixing measure under the additional requirement that the true cluster centers are sufficiently separated.
In particular, a minor modification of Theorem~3.1 in \cite{guha2021posterior} ensures that the posterior of $K_n$ is consistent and the mixing measure converges to the true one at an optimal rate, provided that the point process has a density that is bounded from below when evaluated at sufficiently separated points.

A more interesting scenario occurs when the mixture model is misspecified. This brings to light a fundamental trade-off between cluster and density estimates as shown in two simulation studies of Section~\ref{sec:numeric}.
The general asymptotic theory for misspecified models outlined in \cite{Kle(06)}, specialised to the case of mixtures by \cite{guha2021posterior}, ensures that, in the large sample limit, the mixing measure $\ptilde$ contracts to a measure $p^*$ that minimises the Kullback-Leibler divergence between the true data generating density and the support of the prior. As shown in \cite{cai2021finite}, for traditional mixtures with i.i.d. atoms, $p^*$ has infinite support.
We expect that also in the case of SNCP mixtures, the posterior of the number of components diverges, but, thanks to the decoupling of the notion of cluster (i.e., ``grouping'') from that 
of 
allocated mixture component, the SNCP might estimate the correct number of groups of data even in the large sample limit.
Things are less clear when assuming a repulsive point process prior. To the best of our knowledge, the only result in these directions is Corollary 1 by \cite{xie2019bayesian}, which establishes that, under a specific prior, the number of clusters grows sub-linearly with the sample size.
Of course, we expect different priors to yield different asymptotics. Consider, for instance, the case of a hardcore process, for which the probability of having two points closer than a certain radius $r$ is exactly zero. Then, suppose the data generating process is, e.g., a truncated $t$ distribution over a certain interval and $r$ is sufficiently large. In that case, it is reasonable to expect that $K_n = 1$ for any $n$ (since the prior makes it impossible to add more atoms to the mixture).

In a sense, we might argue that repulsive mixtures are the ``safe option'',  meaning that if the model is well-specified, we expect the same posterior inference as with traditional mixtures. In contrast, if the model is misspecified, empirical works suggest that inference is more robust.
The asymptotic study of repulsive mixtures under misspecification presents a ``dual'' challenge with respect to the typical study of mixture models. Indeed, the challenge here is to explicitly characterise the limiting $p^*$, a problem that needs to be more noticed in the current literature.

\subsection{Repulsive priors or more flexible kernels?}

From the discussion above, it is clear that posterior inference in traditional mixture models is troublesome in misspecified regimes.
Several approaches have been proposed to fix this issue, including generalised Bayes frameworks \citep{rigon23} and distance-based clustering \citep{duan2021bayesian, natarajan2023cohesion}.
These offer robust cluster estimates but not density estimates, and obtaining a robust density estimate might be an equally important problem \citep{miller2018robust}.
As shown in Section~\ref{sec:simu2}, repulsive mixtures do offer robust density estimates in addition to cluster estimates. 

Another natural way to address model misspecification issues is to improve the flexibility of mixture kernels. See the recent contributions of \cite{rodriguez2014univariate, paez2018modeling, mukhopadhyay2020estimating}.
The SNCP mixture proposed in this work constitutes another contribution in this direction, whereby we model a mixture kernel via a mixture of Gaussians. The numerical illustrations of Section~\ref{sec:numeric} show that the choice of the prior distribution (i.e., whether it has repulsive, independent, or attractive atoms) must be carried out on a case-by-case basis.
Indeed, the use of an extremely flexible kernel might result in noisy density estimates if the data-generating process is thought to be corrupted. At the same time, a repulsive mixture might produce unsatisfactory density estimates in other settings. Ultimately, it is up to the practitioner to choose the right model for the job. 

\bibliographystyle{chicago}
\bibliography{references}

\begin{thebibliography}{}

\bibitem[\protect\citeauthoryear{Aldous}{Aldous}{1985}]{Aldous85}
Aldous, D.~J. (1985).
\newblock Exchangeability and related topics.
\newblock In {\em \'{E}cole d'\'{e}t\'{e} de probabilit\'{e}s de
  {S}aint-{F}lour, {XIII}---1983}, Volume 1117 of {\em Lecture Notes in Math.},
  pp.\  1--198. Springer, Berlin.

\bibitem[\protect\citeauthoryear{Anderson, Guionnet, and Zeitouni}{Anderson
  et~al.}{2010}]{anderson2010introduction}
Anderson, G.~W., A.~Guionnet, and O.~Zeitouni (2010).
\newblock {\em An introduction to random matrices}.
\newblock Number 118. Cambridge university press.

\bibitem[\protect\citeauthoryear{Aragam, Dan, Xing, and Ravikumar}{Aragam
  et~al.}{2020}]{aragam2020identifiability}
Aragam, B., C.~Dan, E.~P. Xing, and P.~Ravikumar (2020).
\newblock Identifiability of nonparametric mixture models and {Bayes} optimal
  clustering.
\newblock {\em The Annals of Statistics\/}~{\em 48\/}(4), 2277--2302.

\bibitem[\protect\citeauthoryear{Argiento and De~Iorio}{Argiento and
  De~Iorio}{2022}]{ArDeInf19}
Argiento, R. and M.~De~Iorio (2022).
\newblock Is infinity that far? {A} {Bayesian} nonparametric perspective of
  finite mixture models.
\newblock {\em The Annals of Statistics\/}~{\em 50\/}(5), 2641--2663.

\bibitem[\protect\citeauthoryear{Baccelli, B{\l}aszczyszyn, and
  Karray}{Baccelli et~al.}{2020}]{BaBlaKa}
Baccelli, F., B.~B{\l}aszczyszyn, and M.~Karray (2020).
\newblock Random measures, point processes, and stochastic geometry.
\newblock {\em HAL preprint available at https://hal.inria.fr/hal-02460214/\/}.

\bibitem[\protect\citeauthoryear{Beraha, Argiento, M{\o}ller, and
  Guglielmi}{Beraha et~al.}{2022}]{beraha21}
Beraha, M., R.~Argiento, J.~M{\o}ller, and A.~G. Guglielmi (2022).
\newblock {MCMC} computations for {B}ayesian mixture models using repulsive
  point processes.
\newblock {\em Journal of Computational and Graphical Statistics\/}~{\em
  31\/}(2), 422--435.

\bibitem[\protect\citeauthoryear{Beraha and Camerlenghi}{Beraha and
  Camerlenghi}{2024}]{BerCamPalm}
Beraha, M. and F.~Camerlenghi (2024).
\newblock On the {P}alm distribution of superpositions of point processes.
\newblock {\em arXiv preprint arXiv:2409.14753v1\/}.

\bibitem[\protect\citeauthoryear{Beraha, Camerlenghi, and Ghilotti}{Beraha
  et~al.}{2025}]{GhiloFeatures}
Beraha, M., F.~Camerlenghi, and L.~Ghilotti (2025).
\newblock Bayesian calculus and predictive characterizations of extended
  feature allocation models.
\newblock {\em arXiv preprint arXiv:2502.10257\/}.

\bibitem[\protect\citeauthoryear{Beraha, Guindani, Gianella, and
  Guglielmi}{Beraha et~al.}{2025}]{bayesmix}
Beraha, M., B.~Guindani, M.~Gianella, and A.~Guglielmi (2025).
\newblock Bayesmix: Bayesian mixture models in {C}++.
\newblock {\em Journal of Statistical Software\/}~{\em 112\/}(9), 1–40.

\bibitem[\protect\citeauthoryear{Bianchini, Guglielmi, and Quintana}{Bianchini
  et~al.}{2020}]{bianchini2018determinantal}
Bianchini, I., A.~Guglielmi, and F.~A. Quintana (2020).
\newblock Determinantal point process mixtures via spectral density approach.
\newblock {\em Bayesian Analysis\/}~{\em 15}, 187--214.

\bibitem[\protect\citeauthoryear{Bj{\"o}rk}{Bj{\"o}rk}{2009}]{bjork2009arbitrage}
Bj{\"o}rk, T. (2009).
\newblock {\em Arbitrage theory in continuous time}.
\newblock Oxford university press.

\bibitem[\protect\citeauthoryear{Cai, Campbell, and Broderick}{Cai
  et~al.}{2021}]{cai2021finite}
Cai, D., T.~Campbell, and T.~Broderick (2021).
\newblock Finite mixture models do not reliably learn the number of components.
\newblock In {\em International Conference on Machine Learning}, pp.\
  1158--1169. PMLR.

\bibitem[\protect\citeauthoryear{Camerlenghi, Lijoi, and
  Pr{\"u}nster}{Camerlenghi et~al.}{2018}]{Cam18SJS}
Camerlenghi, F., A.~Lijoi, and I.~Pr{\"u}nster (2018).
\newblock {Bayesian nonparametric inference beyond the Gibbs-type framework}.
\newblock {\em Scandinavian Journal of Statistics\/}~{\em 45\/}(4), 1062--1091.

\bibitem[\protect\citeauthoryear{Charalambides}{Charalambides}{2002}]{charalambides_enumerativeC}
Charalambides, C.~A. (2002).
\newblock {\em Enumerative combinatorics}.
\newblock CRC Press Series on Discrete Mathematics and its Applications.
  Chapman \& Hall/CRC, Boca Raton, FL.

\bibitem[\protect\citeauthoryear{Coeurjolly, M{\o}ller, and
  Waagepetersen}{Coeurjolly et~al.}{2017}]{CaMoWa17}
Coeurjolly, J.-F., J.~M{\o}ller, and R.~Waagepetersen (2017).
\newblock A tutorial on {Palm} distributions for spatial point processes.
\newblock {\em International Statistical Review\/}~{\em 85\/}(3), 404--420.

\bibitem[\protect\citeauthoryear{Cox}{Cox}{1955}]{Cox1955}
Cox, D.~R. (1955, July).
\newblock Some statistical methods connected with series of events.
\newblock {\em Journal of the Royal Statistical Society: Series B
  (Methodological)\/}~{\em 17\/}(2), 129--157.

\bibitem[\protect\citeauthoryear{Cremaschi, Wertz, and De~Iorio}{Cremaschi
  et~al.}{2025}]{cremaschi2023repulsion}
Cremaschi, A., T.~M. Wertz, and M.~De~Iorio (2025).
\newblock Repulsion, chaos, and equilibrium in mixture models.
\newblock {\em Journal of the Royal Statistical Society Series B: Statistical
  Methodology\/}~{\em 87\/}(2), 389--432.

\bibitem[\protect\citeauthoryear{Daley and Vere-Jones}{Daley and
  Vere-Jones}{2003}]{DaVeJo1}
Daley, D.~J. and D.~Vere-Jones (2003).
\newblock {\em An introduction to the theory of point processes. {V}ol. {I}\/}
  (Second ed.).
\newblock Probability and its Applications (New York). New York:
  Springer-Verlag.
\newblock Elementary theory and methods.

\bibitem[\protect\citeauthoryear{Daley and Vere-Jones}{Daley and
  Vere-Jones}{2008}]{DaVeJo2}
Daley, D.~J. and D.~Vere-Jones (2008).
\newblock {\em An introduction to the theory of point processes. {V}ol. {II}\/}
  (Second ed.).
\newblock Probability and its Applications (New York). New York: Springer.
\newblock General theory and structure.

\bibitem[\protect\citeauthoryear{De~Blasi, Favaro, Lijoi, Mena, Pr{\"u}nster,
  and Ruggiero}{De~Blasi et~al.}{2013}]{deblasi2013gibbs}
De~Blasi, P., S.~Favaro, A.~Lijoi, R.~H. Mena, I.~Pr{\"u}nster, and M.~Ruggiero
  (2013).
\newblock Are {Gibbs}-type priors the most natural generalization of the
  {Dirichlet} process?
\newblock {\em IEEE transactions on pattern analysis and machine
  intelligence\/}~{\em 37\/}(2), 212--229.

\bibitem[\protect\citeauthoryear{de~Finetti}{de~Finetti}{1937}]{defin37}
de~Finetti, B. (1937).
\newblock La pr\'{e}vision: ses lois logiques, ses sources subjectives.
\newblock {\em Annales de l'Institut Henri Poincar\'{e}\/}~{\em 7\/}(1), 1--68.

\bibitem[\protect\citeauthoryear{Duan and Dunson}{Duan and
  Dunson}{2021}]{duan2021bayesian}
Duan, L.~L. and D.~B. Dunson (2021).
\newblock Bayesian distance clustering.
\newblock {\em The Journal of Machine Learning Research\/}~{\em 22\/}(1),
  10228--10254.

\bibitem[\protect\citeauthoryear{Ferguson}{Ferguson}{1973}]{Fer73}
Ferguson, T.~S. (1973).
\newblock A {Bayesian} analysis of some nonparametric problems.
\newblock {\em The Annals of Statistics\/}, 209--230.

\bibitem[\protect\citeauthoryear{Ferreira and Menegatto}{Ferreira and
  Menegatto}{2009}]{ferreira2009eigenvalues}
Ferreira, J. and V.~Menegatto (2009).
\newblock Eigenvalues of integral operators defined by smooth positive definite
  kernels.
\newblock {\em Integral Equations and Operator Theory\/}~{\em 64\/}(1), 61--81.

\bibitem[\protect\citeauthoryear{Fruhwirth-Schnatter, Celeux, and
  Robert}{Fruhwirth-Schnatter et~al.}{2019}]{fruhwirth2019handbook}
Fruhwirth-Schnatter, S., G.~Celeux, and C.~P. Robert (2019).
\newblock {\em Handbook of mixture analysis}.
\newblock CRC press.

\bibitem[\protect\citeauthoryear{Georgii and Yoo}{Georgii and
  Yoo}{2005}]{georgii2005conditional}
Georgii, H.-O. and H.~J. Yoo (2005).
\newblock Conditional intensity and {Gibbsianness} of determinantal point
  processes.
\newblock {\em Journal of Statistical Physics\/}~{\em 118}, 55--84.

\bibitem[\protect\citeauthoryear{Ghilotti, Beraha, and Guglielmi}{Ghilotti
  et~al.}{2025}]{ghilotti2023bayesian}
Ghilotti, L., M.~Beraha, and A.~Guglielmi (2025).
\newblock Bayesian clustering of high-dimensional data via latent repulsive
  mixtures.
\newblock {\em Biometrika\/}~{\em 112\/}(2), asae059.

\bibitem[\protect\citeauthoryear{Ghosal, Ghosh, and Ramamoorthi}{Ghosal
  et~al.}{1999}]{ghosal1999posterior}
Ghosal, S., J.~K. Ghosh, and R.~Ramamoorthi (1999).
\newblock Posterior consistency of {Dirichlet} mixtures in density estimation.
\newblock {\em The Annals of Statistics\/}~{\em 27\/}(1), 143--158.

\bibitem[\protect\citeauthoryear{{Grazian}}{{Grazian}}{2023}]{grazian23}
{Grazian}, C. (2023).
\newblock {A review on {B}ayesian model-based clustering}.
\newblock {\em arXiv:2303.17182\/}.

\bibitem[\protect\citeauthoryear{Griffin and Walker}{Griffin and
  Walker}{2011}]{griffin2011posterior}
Griffin, J.~E. and S.~G. Walker (2011).
\newblock Posterior simulation of normalized random measure mixtures.
\newblock {\em Journal of Computational and Graphical Statistics\/}~{\em
  20\/}(1), 241--259.

\bibitem[\protect\citeauthoryear{Guha, Ho, and Nguyen}{Guha
  et~al.}{2021}]{guha2021posterior}
Guha, A., N.~Ho, and X.~Nguyen (2021).
\newblock On posterior contraction of parameters and interpretability in
  {Bayesian} mixture modeling.
\newblock {\em Bernoulli\/}~{\em 27\/}(4), 2159--2188.

\bibitem[\protect\citeauthoryear{Hough, Krishnapur, Peres, and Vir\`{a}g}{Hough
  et~al.}{2006}]{Hough:etal:06}
Hough, J.~B., M.~Krishnapur, Y.~Peres, and B.~Vir\`{a}g (2006).
\newblock Determinantal processes and independence.
\newblock {\em Probability Surveys\/}~{\em 3}, 206--229.

\bibitem[\protect\citeauthoryear{Hough, Krishnapur, Peres, and Vir\'{a}g}{Hough
  et~al.}{2009}]{Hough09}
Hough, J.~B., M.~Krishnapur, Y.~Peres, and B.~Vir\'{a}g (2009).
\newblock {\em Zeros of Gaussian Analytic Functions and Determinantal Point
  Processes}.
\newblock Providence: American Mathematical Society.

\bibitem[\protect\citeauthoryear{James}{James}{2002}]{james2002poisson}
James, L.~F. (2002).
\newblock Poisson process partition calculus with applications to exchangeable
  models and {Bayesian} nonparametrics.
\newblock {\em arXiv preprint math/0205093\/}.

\bibitem[\protect\citeauthoryear{James}{James}{2005}]{james_levy_moving}
James, L.~F. (2005).
\newblock {Bayesian Poisson process partition calculus with an application to
  Bayesian L{\'e}vy moving averages}.
\newblock {\em The Annals of Statistics\/}~{\em 33\/}(4), 1771 -- 1799.

\bibitem[\protect\citeauthoryear{James}{James}{2006}]{james_ntr}
James, L.~F. (2006).
\newblock {Poisson calculus for spatial neutral to the right processes}.
\newblock {\em The Annals of Statistics\/}~{\em 34\/}(1), 416 -- 440.

\bibitem[\protect\citeauthoryear{James, Lijoi, and Pr{\"u}nster}{James
  et~al.}{2009}]{JaLiPr09}
James, L.~F., A.~Lijoi, and I.~Pr{\"u}nster (2009).
\newblock Posterior analysis for normalized random measures with independent
  increments.
\newblock {\em Scandinavian Journal of Statistics\/}~{\em 36\/}(1), 76--97.

\bibitem[\protect\citeauthoryear{Kallenberg}{Kallenberg}{1984}]{Kallenberg84}
Kallenberg, O. (1984).
\newblock An informal guide to the theory of conditioning in point processes.
\newblock {\em International Statistical Review / Revue Internationale de
  Statistique\/}~{\em 52\/}(2), 151--164.

\bibitem[\protect\citeauthoryear{Kallenberg}{Kallenberg}{2017}]{Kallenberg17}
Kallenberg, O. (2017).
\newblock {\em Random measures, theory and applications}, Volume~1.
\newblock Springer.

\bibitem[\protect\citeauthoryear{Kallenberg}{Kallenberg}{2021}]{Kallenberg2021}
Kallenberg, O. (2021).
\newblock {\em Foundations of modern probability}, Volume~99 of {\em
  Probability Theory and Stochastic Modelling}.
\newblock Springer, Cham.
\newblock Third edition.

\bibitem[\protect\citeauthoryear{Kingman}{Kingman}{1992}]{kingman1992poisson}
Kingman, J. F.~C. (1992).
\newblock {\em Poisson processes}, Volume~3.
\newblock Clarendon Press.

\bibitem[\protect\citeauthoryear{Kleijn and van~der Vaart}{Kleijn and van~der
  Vaart}{2006}]{Kle(06)}
Kleijn, B. J.~K. and A.~W. van~der Vaart (2006).
\newblock {Misspecification in infinite-dimensional Bayesian statistics}.
\newblock {\em The Annals of Statistics\/}~{\em 34\/}(2), 837 -- 877.

\bibitem[\protect\citeauthoryear{Last and Penrose}{Last and
  Penrose}{2017}]{last2017lectures}
Last, G. and M.~Penrose (2017).
\newblock {\em {Lectures on the Poisson process}}, Volume~7.
\newblock Cambridge University Press.

\bibitem[\protect\citeauthoryear{Lavancier, M{\o}ller, and Rubak}{Lavancier
  et~al.}{2015}]{Lav15}
Lavancier, F., J.~M{\o}ller, and E.~Rubak (2015).
\newblock Determinantal point process models and statistical inference.
\newblock {\em Journal of the Royal Statistical Society: Series B (Statistical
  Methodology)\/}~{\em 77\/}(4), 853--877.

\bibitem[\protect\citeauthoryear{Lijoi, Pr{\"u}nster, et~al.}{Lijoi
  et~al.}{2010}]{lijoi2010models}
Lijoi, A., I.~Pr{\"u}nster, et~al. (2010).
\newblock {Models beyond the Dirichlet process}.
\newblock {\em Bayesian nonparametrics\/}~{\em 28\/}(80), 342.

\bibitem[\protect\citeauthoryear{Lijoi, Pr\"unster, and Rigon}{Lijoi
  et~al.}{2022}]{lijoi22}
Lijoi, A., I.~Pr\"unster, and T.~Rigon (2022).
\newblock Finite-dimensional discrete random structures and {Bayesian}
  clustering.
\newblock {\em Journal of the American Statistical Association\/}.

\bibitem[\protect\citeauthoryear{Macchi}{Macchi}{1975}]{Macchi75}
Macchi, O. (1975).
\newblock The coincidence approach to stochastic point processes.
\newblock {\em Advances in Applied Probability\/}~{\em 7}, 83--122.

\bibitem[\protect\citeauthoryear{Malsiner-Walli, Fr{\"u}hwirth-Schnatter, and
  Gr{\"u}n}{Malsiner-Walli et~al.}{2017}]{WalliMix}
Malsiner-Walli, G., S.~Fr{\"u}hwirth-Schnatter, and B.~Gr{\"u}n (2017).
\newblock Identifying mixtures of mixtures using {Bayesian} estimation.
\newblock {\em Journal of Computational and Graphical Statistics\/}~{\em
  26\/}(2), 285--295.

\bibitem[\protect\citeauthoryear{Miller and Dunson}{Miller and
  Dunson}{2019}]{miller2018robust}
Miller, J.~W. and D.~B. Dunson (2019).
\newblock {Robust Bayesian inference via coarsening}.
\newblock {\em Journal of the American Statistical Association\/}~{\em
  114\/}(527), 1113--1125.

\bibitem[\protect\citeauthoryear{M{\o}ller}{M{\o}ller}{2003}]{Mo03Cox}
M{\o}ller, J. (2003).
\newblock Shot noise {Cox} processes.
\newblock {\em Advances in Applied Probability\/}~{\em 35\/}(3), 614--640.

\bibitem[\protect\citeauthoryear{M{\o}ller and Waagepetersen}{M{\o}ller and
  Waagepetersen}{2003}]{MoWaBook03}
M{\o}ller, J. and R.~P. Waagepetersen (2003).
\newblock {\em Statistical inference and simulation for spatial point
  processes}.
\newblock CRC Press.

\bibitem[\protect\citeauthoryear{Mukhopadhyay, Li, and Dunson}{Mukhopadhyay
  et~al.}{2020}]{mukhopadhyay2020estimating}
Mukhopadhyay, M., D.~Li, and D.~B. Dunson (2020).
\newblock Estimating densities with non-linear support by using
  {Fisher}--{Gaussian} kernels.
\newblock {\em Journal of the Royal Statistical Society Series B: Statistical
  Methodology\/}~{\em 82\/}(5), 1249--1271.

\bibitem[\protect\citeauthoryear{Natarajan, De~Iorio, Heinecke, Mayer, and
  Glenn}{Natarajan et~al.}{2024}]{natarajan2023cohesion}
Natarajan, A., M.~De~Iorio, A.~Heinecke, E.~Mayer, and S.~Glenn (2024).
\newblock {Cohesion and repulsion in Bayesian distance clustering}.
\newblock {\em Journal of the American Statistical Association\/}~{\em
  119\/}(546), 1374--1384.

\bibitem[\protect\citeauthoryear{Neal}{Neal}{2000}]{neal2000markov}
Neal, R.~M. (2000).
\newblock {Markov chain sampling methods for Dirichlet process mixture models}.
\newblock {\em Journal of computational and graphical statistics\/}~{\em
  9\/}(2), 249--265.

\bibitem[\protect\citeauthoryear{Paez and Walker}{Paez and
  Walker}{2018}]{paez2018modeling}
Paez, M.~S. and S.~G. Walker (2018).
\newblock Modeling with a large class of unimodal multivariate distributions.
\newblock {\em Journal of applied statistics\/}~{\em 45\/}(10), 1823--1845.

\bibitem[\protect\citeauthoryear{Papaspiliopoulos and Roberts}{Papaspiliopoulos
  and Roberts}{2008}]{papaspiliopoulos2008retrospective}
Papaspiliopoulos, O. and G.~O. Roberts (2008).
\newblock Retrospective {Markov chain Monte Carlo} methods for {Dirichlet}
  process hierarchical models.
\newblock {\em Biometrika\/}~{\em 95\/}(1), 169--186.

\bibitem[\protect\citeauthoryear{Petralia, Rao, and Dunson}{Petralia
  et~al.}{2012}]{PeRaDu12}
Petralia, F., V.~Rao, and D.~Dunson (2012).
\newblock Repulsive mixtures.
\newblock In F.~Pereira, C.~Burges, L.~Bottou, and K.~Weinberger (Eds.), {\em
  Advances in Neural Information Processing Systems}, Volume~25. Curran
  Associates, Inc.

\bibitem[\protect\citeauthoryear{Pitman}{Pitman}{1996}]{pitman1996}
Pitman, J. (1996).
\newblock Some developments of the {B}lackwell-{M}ac{Q}ueen urn scheme.
\newblock In {\em Statistics, probability and game theory}, Volume~30 of {\em
  IMS Lecture Notes Monogr. Ser.}, pp.\  245--267. Inst. Math. Statist.,
  Hayward, CA.

\bibitem[\protect\citeauthoryear{Quinlan, Quintana, and Page}{Quinlan
  et~al.}{2021}]{quinlan2017parsimonious}
Quinlan, J.~J., F.~A. Quintana, and G.~L. Page (2021).
\newblock Parsimonious hierarchical modeling using repulsive distributions.
\newblock {\em Test\/}~{\em 30}, 445--461.

\bibitem[\protect\citeauthoryear{Regazzini, Lijoi, and Pr{\"u}nster}{Regazzini
  et~al.}{2003}]{ReLiPr03}
Regazzini, E., A.~Lijoi, and I.~Pr{\"u}nster (2003).
\newblock {Distributional results for means of normalized random measures with
  independent increments}.
\newblock {\em The Annals of Statistics\/}~{\em 31\/}(2), 560 -- 585.

\bibitem[\protect\citeauthoryear{Rigon, Herring, and Dunson}{Rigon
  et~al.}{2023}]{rigon23}
Rigon, T., A.~H. Herring, and D.~B. Dunson (2023, 01).
\newblock {A generalized Bayes framework for probabilistic clustering}.
\newblock {\em Biometrika\/}~{\em 110\/}(3), 559--578.

\bibitem[\protect\citeauthoryear{Rodr{\'\i}guez and Walker}{Rodr{\'\i}guez and
  Walker}{2014}]{rodriguez2014univariate}
Rodr{\'\i}guez, C.~E. and S.~G. Walker (2014).
\newblock Univariate bayesian nonparametric mixture modeling with unimodal
  kernels.
\newblock {\em Statistics and Computing\/}~{\em 24}, 35--49.

\bibitem[\protect\citeauthoryear{Roeder}{Roeder}{1990}]{roeder1990density}
Roeder, K. (1990).
\newblock Density estimation with confidence sets exemplified by superclusters
  and voids in the galaxies.
\newblock {\em Journal of the American Statistical Association\/}~{\em
  85\/}(411), 617--624.

\bibitem[\protect\citeauthoryear{Shirai and Takahashi}{Shirai and
  Takahashi}{2003}]{shirai2003random}
Shirai, T. and Y.~Takahashi (2003).
\newblock {Random point fields associated with certain Fredholm determinants I:
  fermion, Poisson and boson point processes}.
\newblock {\em Journal of Functional Analysis\/}~{\em 205\/}(2), 414--463.

\bibitem[\protect\citeauthoryear{Soshnikov}{Soshnikov}{2000}]{soshnikov2000determinantal}
Soshnikov, A. (2000).
\newblock Determinantal random point fields.
\newblock {\em Russian Mathematical Surveys\/}~{\em 55\/}(5), 923.

\bibitem[\protect\citeauthoryear{Steele}{Steele}{1994}]{steele1994cam}
Steele, J.~M. (1994).
\newblock Le cam's inequality and poisson approximations.
\newblock {\em The American Mathematical Monthly\/}~{\em 101\/}(1), 48--54.

\bibitem[\protect\citeauthoryear{Sun, Zhang, and Rao}{Sun
  et~al.}{2022}]{sun2022bayesian}
Sun, H., B.~Zhang, and V.~Rao (2022).
\newblock {Bayesian Repulsive Mixture Modeling with Matern Point Processes}.
\newblock {\em arXiv preprint arXiv:2210.04140\/}.

\bibitem[\protect\citeauthoryear{Sun, Zhao, and Zhu}{Sun
  et~al.}{2015}]{nystrom}
Sun, S., J.~Zhao, and J.~Zhu (2015).
\newblock A review of {N}ystr{\"o}m methods for large-scale machine learning.
\newblock {\em Information Fusion\/}~{\em 26}, 36--48.

\bibitem[\protect\citeauthoryear{Teh, Jordan, Beal, and Blei}{Teh
  et~al.}{2006}]{Teh06}
Teh, Y.~W., M.~I. Jordan, M.~J. Beal, and D.~M. Blei (2006).
\newblock {Hierarchical Dirichlet Processes}.
\newblock {\em Journal of the American Statistical Association\/}~{\em
  101\/}(476), 1566--1581.

\bibitem[\protect\citeauthoryear{Wade}{Wade}{2023}]{wade23}
Wade, S. (2023).
\newblock Bayesian cluster analysis.
\newblock {\em Philos. Trans. Roy. Soc. A\/}~{\em 381\/}(2247), Paper No.
  20220149, 20.

\bibitem[\protect\citeauthoryear{Wang, Degleris, Williams, and Linderman}{Wang
  et~al.}{2024}]{wang2022spatiotemporal}
Wang, Y., A.~Degleris, A.~H. Williams, and S.~W. Linderman (2024).
\newblock Spatiotemporal clustering with {Neyman-Scott} processes via
  connections to {Bayesian} nonparametric mixture models.
\newblock {\em Journal of the American Statistical Association\/}~{\em
  119\/}(547), 2382--2395.

\bibitem[\protect\citeauthoryear{Xie and Xu}{Xie and
  Xu}{2019}]{xie2019bayesian}
Xie, F. and Y.~Xu (2019).
\newblock Bayesian repulsive {G}aussian mixture model.
\newblock {\em Journal of the American Statistical Association\/}, 187--203.

\bibitem[\protect\citeauthoryear{Xu, M{\"u}ller, and Telesca}{Xu
  et~al.}{2016}]{xu2016bayesian}
Xu, Y., P.~M{\"u}ller, and D.~Telesca (2016).
\newblock Bayesian inference for latent biologic structure with determinantal
  point processes {(DPP)}.
\newblock {\em Biometrics\/}~{\em 72\/}(3), 955--964.

\bibitem[\protect\citeauthoryear{Zhou, Favaro, and Walker}{Zhou
  et~al.}{2017}]{zhou_fof}
Zhou, M., S.~Favaro, and S.~G. Walker (2017).
\newblock Frequency of frequencies distributions and size-dependent
  exchangeable random partitions.
\newblock {\em Journal of the American Statistical Association\/}~{\em
  112\/}(520), 1623--1635.

\end{thebibliography}

\newpage
\appendix

\begin{center}
   \LARGE \textbf{ Supplementary materials to:\\
    ``Bayesian Mixtures Models with Repulsive and Attractive Atoms''}
\end{center}



The material in the appendix is organized as follows. 
Appendix~\ref{app:palm} reports details on Palm calculus, while Appendix~\ref{app:lemmata} contains preparatory lemmas related to independently marked point processes. In Appendix~\ref{app:prior} we study the properties of $\tilde{p} \sim \text{nRM}(\plaw_\Phi; H)$, which are useful for prior elicitation. 
Appendix~\ref{app:proofs}  includes proofs of the main results, i.e., posterior characterization (Theorem~\ref{teo:post}), marginal characterization 
(Theorem~\ref{teo:marg}) and predictive distribution 
(Theorem~\ref{teo:pred_conditional}).x
Appendices~\ref{app:poisson}-\ref{app:sncp} 
give details on the specific priors for $\Phi$ we consider, that is the Poisson, Gibbs, determinantal and shot-noise Cox point processes. Finally, Appendix ~\ref{app:numerical_examples} provides further numerical illustrations.

\section{Background on Palm calculus}\label{app:palm}

The basic tool used in our computations is a disintegration of the Campbell measure of a point process $\Phi$ with respect to its mean measure $M_{\Phi}$, usually called the Palm kernel or the family of Palm distributions of $\Phi$.
Below, we recall the main results needed later in this paper. For further details about Palm distributions and Palm calculus, see, e.g., the papers \cite{Kallenberg84, CaMoWa17} 
or the monographs \cite{Kallenberg17} (Chapter 6), \cite{DaVeJo2} (Chapter 13). Here, we adapt the notation from the recent monograph \cite{BaBlaKa} (Chapter 3).

Let $\M(\X)$ the space of bounded measures on $\X$ and denote by $\calM(\X)$ its Borel $\sigma$-algebra.
For a point process $\Phi$ on $\X$, let us denote the mean measure by $M_\Phi$, i.e.,  $M_\Phi(B) := \E[\Phi(B)]$ for all $B \in \calX$.
We define the Campbell measure $\calC_\Phi$ on $\X \times \M(\X)$ as
\[
    \calC_\Phi(B \times L) = \E \left[\Phi(B) \indicator_L(\Phi)\right], \quad B \in \calX, L \in \calM(\X).
\]
Then, as a consequence of the Radon-Nikodym theorem, for any fixed $L$, there exists a $M_\Phi$-a.e. unique disintegration probability kernel $\{\plaw_\Phi^x(\cdot)\}_{x \in \X}$ of $\calC_\Phi$ with respect to $M_\Phi$, i.e.
\[
    \calC_\Phi(B \times L) = \int_B \plaw_\Phi^x(L) M_\Phi(\dd x) \quad B \in \calX, L \in \calM(\X) .
\]
Note that, for any $x \in \X$, $\plaw_\Phi^x$ is the distribution of a random measure (specifically, a point process) on $\X$. See \cite[][Theorem 31.1]{Kallenberg2021}. 
Therefore, $\plaw_\Phi^x$ can be identified with the distribution of a point process $\Phi_x$ such that $\plaw_\Phi^x(L) = \prob(\Phi_x \in L)$. 
In particular, Proposition 3.1.12 in \cite{BaBlaKa} shows that the point process $\Phi_x$ has almost surely
an atom at $x$. This property allows for the interpretation of $\plaw_\Phi^x$ as
the probability distribution of the point process $\Phi$ given that (conditionally to) it
has an atom at $x$.
$\Phi_x$ is called  the Palm version of $\Phi$ at $x$.

\begin{theorem}[Campbell-Little-Mecke formula. Theorem 3.1.9 in \cite{BaBlaKa}]\label{teo:clm}
Let $\Phi$ be a point process on $\X$ such that $M_\Phi$ is $\sigma$-finite. Denote with $\plaw_\Phi(\cdot)$ its law. Let 
$\{\plaw_\Phi^x(\cdot)\}_{x \in \X}$ be a family of Palm distributions of $\Phi$. 
Then, for all measurable $g: \X \times \M(\X) \rightarrow \R_+$, one has
\begin{equation}
 \E \left[ \int_\X g(x, \Phi) \Phi(\dd x) \right] = \int_{\X \times \M(\X)} g(x, \nu) \nu(\dd x) \plaw_\Phi(\dd \nu)  = \int_{\M(\X) \times \X} g(x, \nu) \plaw_\Phi^x(\dd \nu) M_\Phi(\dd x).
\label{eq:clm}
\end{equation}
\end{theorem}

In many applications to spatial statistics, formula \eqref{eq:clm}, referred to as CLM formula, is stated in terms of the \emph{reduced} Palm kernel $\plaw_{\Phi^!}^{x}$, i.e., 
\[
\E \left[ \int_\X g(x, \Phi - \delta_x) \Phi(\dd x) \right] = \int_{\M(\X) \times \X} g(x, \nu) \plaw_{\Phi^!}^x(\dd \nu) M_\Phi(\dd x)
\]
where $\Phi - \delta_x$ is obtained by removing the point $x$ from $\Phi$ and $\plaw_{\Phi^!}^x$ is the distribution of the point process 
\[
    \Phi^!_x := \Phi_x - \delta_x.
\]
Hence, given a reduced Palm kernel, we can construct the non-reduced one by considering the distribution of $\Phi^!_x + \delta_x$.

Given a point process $\Phi$ define the $k$-th Campbell measure as the unique measure \citep[Lemma 3.3.1 in][]{BaBlaKa} characterized by
\[
    \calC^k_\Phi(B \times L) = \E \left[ \int_B \indicator_L(\Phi) \Phi^k(\dd \bm x) \right], \quad B \in \calX^{\otimes k}, L \in \calM(\X),
\]
where $\dd \bm x = \dd x_1 \cdots \dd x_k$ and $\Phi^k(\dd \bm x) = \prod_{i=1}^k \Phi(\dd x_i)$. 
Let $M_{\Phi^k}$ be the mean measure of $\Phi^k$, i.e., $M_{\Phi^k}(B) = \calC^k(B \times \M(\X))$, then the $k$--th Palm distribution $\{\plaw_\Phi^{\bm{x}}\}_{\bm{x} \in \X^k}$ is defined as the disintegration kernel of $\calC_\Phi^k$ with respect to $M_{\Phi^k}$, that is
\[
    \calC^k_\Phi(B \times L) = \int_B \plaw_\Phi^{\bm{x}}(L) M_{\Phi^k}(\dd \bm x), \quad B \in \calX^{\otimes k}, L \in \calM(\X) .
\]
The following is a multivariate extension of Theorem~\ref{teo:clm} that will be useful for later computations.
\begin{theorem}[Higher order CLM formula.]\label{teo:clm_mult}
Let $\Phi$ a point process on $\X$ such that $M_{\Phi^k}$ is $\sigma$-finite.  Let 
$\{\plaw_\Phi^{\bm x}(\cdot)\}_{\bm{x} \in \X^k}$ a family of $k$--th Palm distributions of $\Phi$. 
Then, for all measurable functions $g: \X^k \times \M(\X) \rightarrow \R_+$
\begin{equation}
  \E \left[ \int_{\X^k} g(\bm x, \Phi) \Phi^k(\dd \bm x) \right] =  \int_{\M(\X) \times \X^k} g(\bm x, \nu) \plaw_\Phi^{\bm x}(\dd \nu) M_{\Phi^k}(\dd \bm x) .
\label{eq:clm_multi}
\end{equation}
Equivalently, in terms of the \emph{reduced} Palm kernel, we have
\begin{equation}
  \E \left[ \int_{\X^k} g\left(\bm x, \Phi - \sum_{j=1}^k \delta_{x_j} \right) \Phi^{(k)}(\dd \bm x) \right] =  \int_{\M(\X) \times \X^k} g(\bm x, \nu) \plaw_\Phi^{!\bm x}(\dd \nu) M_{\Phi^{(k)}}(\dd \bm x) .
\label{eq:clm_multi_red}
\end{equation}
Moreover, Equations \eqref{eq:clm_multi} and \eqref{eq:clm_multi_red} uniquely characterize the Palm and reduced Palm kernel up to null sets, respectively.
\end{theorem}
Formulation \eqref{eq:clm_multi_red} is often convenient since it involves the \emph{factorial} moment measure. The following distributional identity relates the two formulations
\[
    \Phi^!_{\bm x} \stackrel{d}{=} \Phi_{\bm x} - \sum_{j = 1}^k \delta_{x_j},
\]
that holds for $M_{\Phi^{(k)}}$-almost all $\bm x \in \X^k$.

\section{Preparatory Lemmas}\label{app:lemmata}

We now state some results relating to independently marked point processes.
\begin{lemma}\label{prop:n_mean}
Let $\Psi$ be an independently marked point process on $\X \times \Ss$ obtained by marking a point process $\Phi$ on $\X$ with i.i.d marks $ (S_j)_{j \geq 1}$ from a probability measure $H$, that does not depend on the value of the associated atom $X_j$.
Then, the mean measure $M_{\Psi}$ is
\[
    M_{\Psi}(\dd x \, \dd s) = H(\dd s) M_{\Phi}(\dd x) .
\]
Analogously, if $M_{\Psi^{(n)}}$ is the $n$-th factorial mean measure of $\Psi$ , then it equals
\[
    M_{\Psi^{(n)}}(\dd \bm x  \, \dd \bm s) = H^n(\dd \bm s) M_{\Phi^{(n)}} (\dd \bm x).
\]
\end{lemma}
\begin{proof}
Let $C = A \times B$, where $A \in \calX$ $B \in \calS$. We have that:
    \begin{align*}
        \E[\Psi(C)] &= \E \left[ \sum_{j \geq 1} \indicator_{A \times B}(X_j, S_j) \right] =  \E\left[\sum_{j \geq 1 }\indicator_A(X_j)  \indicator_B(S_j)\right] \\
        &= \sum_{j \geq 1 } \E \left[\indicator_A(X_j) \indicator_B(S_j)\right] = \sum_{j \geq 1 } H(B) \E \left[\indicator_A(X_j)\right] \\
        &= H(B) \E \left[ \sum_{j \geq 1} \indicator_A(X_j) \right] = H(B) M_\Phi(B). 
    \end{align*}    
The proof for the $n$-th factorial moment measure is achieved following the same steps with $A \in \calX^{\otimes n}$ and $B \in \calS^{\otimes n}$.
\end{proof}

\begin{lemma}\label{prop:marked_palm}
Let $\Psi$ be as in \Cref{prop:n_mean}, then the Palm distribution $\{\plaw_\Psi^{x, s}\}_{(x, s) \in \X \times \Ss}$ is the distribution of the point process $\delta_{(x, s)} + \Psi^!_{x, s}$, where $\Psi^!_{x, s}$ is an independently marked point process obtained by marking $\Phi^!_x \sim \plaw_{\Phi^!}^x$ with i.i.d marks from $H$.
Similarly, let $(\bm{x}, \bm{s}) = (x_1, \ldots, x_n, s_1, \ldots, s_n)$, the Palm distribution $\{\plaw_\Psi^{\bm x, \bm s}\}_{(\bm{x}, \bm{s}) \in \X^{n} \times \Ss^{n}}$ is the distribution of the point process $\sum_{i=1}^n \delta_{(x_i, s_i)} +  \Psi^!_{\bm x, \bm s}$, where $\Psi^!_{\bm x, \bm s}$ is an independently marked point process obtained by marking $\Phi^!_{\bm x} \sim \plaw_\Phi^{\bm x}$ with i.i.d. marks from $H$.
\end{lemma}
\begin{proof}
By the CLM formula for the reduced Palm kernel, we know that $\Psi^!_{x, s}$ satisfies
\begin{equation*}\label{eq:clm_marked_red}
    \E \left[ \int_{\X \times \Ss} g(x, s, \Psi - \delta_{(x, s)}) \Psi(\dd x \, \dd s) \right] = \int_{\X \times \Ss} \E \left[ g(x, s, \Psi^!_{x, s}) \right] M_\Psi(\dd x \, \dd s) .
\end{equation*}
Consider now the point process $\Psi^\prime_{x, s}$ obtained by marking $\Phi^!_x$ with i.i.d marks from $H$, if
\begin{equation}\label{eq:ans}
     \int_{\X \times \Ss} \E \left[ g(x, s, \Psi^\prime_{x, s}) \right] M_\Psi(\dd x \, \dd s) = \int_{\X \times \Ss} \E \left[ g(x, s, \Psi^!_{x, s}) \right] M_\Psi(\dd x \, \dd s),
\end{equation}
for any $g$, we can conclude that $\Psi^\prime_{x, s}$ and $\Psi^!_{x, s}$ are equal in distribution.
To prove \eqref{eq:ans}, we will show that
\[
  \int_{\X \times \Ss} \E \left[ g(x, s, \Psi^\prime_{x, s}) \right] M_\Psi(\dd x \, \dd s) = \E \left[ \int_{\X \times \Ss} g(x, s, \Psi - \delta_{x, s}) \Psi(\dd x \, \dd s) \right] .
\]
In the following, we write $E_V[f(x, z)]$ to indicate that the expectation is taken with respect to the generic random variable $V$.
Write $\Psi^\prime = (\Phi^!, \bm m)$ where $\bm m$ is the collection of marks. With a slight abuse of notation, we write $g(x, s, \Psi^\prime_{x, s}) = g(x, s, \Phi^!_x, \bm m)$. Then
\begin{align*}
    \int_{\X \times \Ss} & \E_{\Psi} \left[ g(x, s, \Psi^\prime_{x, s}) \right] M_\Psi(\dd x \, \dd s) = \int_{\X \times \Ss} \E_{\Phi^!_x, \bm m} \left[ g(x, s, \Phi^!_x, \bm m) \right] M_\Psi(\dd x \, \dd s) \\
    &= \int_{\X \times \Ss} \E_{\Phi^!_x} \left[ \E_{\bm m} \left[ g(x, s, \Phi^!_x, \bm m) \mid \Phi^!_x \right]\right] M_\Psi(\dd x, \dd s) \\
    &=  \int_{\X \times \Ss} \E_{\Phi^!_x} \left[ \int_{\Ss^{n^!}} g(x, s, \Phi^!_x, \bm m) \prod_{i: x_i \in \Phi^!_x} H(\dd m_i) \right] M_{\Phi}(\dd x) H(\dd s)
\end{align*}
where $n^!$ denotes the cardinality of $\Phi^!_x$. Denoting the cardinality of $\Phi$ with $n$, $n^! = n -1$, by Fubini's theorem, we can interchange the outermost integral over $\Ss$ with $\E_{\Phi^!_x}$. Then, we apply the CLM formula (in reversed order) with respect to $\Phi^!_x$, thus obtaining
\begin{multline*}
    \int_{\X \times \Ss} \E_{\Psi} \left[ g(x, s, \Psi^\prime_{x, s}) \right] M_\Psi(\dd x, \dd s) \\= \E_\Phi \left[ \int_{\X} \int_{\Ss^{n - 1 + 1}} g(x, s, \Phi_x - \delta_x, \bm m) \prod_{i: x_i \in \Phi_x - \delta_x} H(\dd m_i) H(\dd s) \Phi(\dd x) \right] .
\end{multline*}
We now observe that $(\Phi - \delta_x, \bm m) = \Psi - \delta_{(x, s)}$ and $\Psi(\dd x \, \dd s) = H(\dd s) \Phi(\dd x)$, as a consequence the right-hand side of the equation above equals
\begin{align*}
  \E & \left[ \int_{\X \times \Ss} \int_{\Ss^{n-1}} g(x, s, \Psi - \delta_{(x, s)}) \prod_{i: x_i \in \Phi_x - \delta_x} H(\dd m_i) \Psi(\dd x \, \dd s) \right]   \\
  &= \E \left[ \int_{\Ss^{n-1}}  \int_{\X \times \Ss} g(x, s, \Psi - \delta_{(x, s)}) \prod_{i: x_i \in \Phi_x - \delta_x} H(\dd m_i) \Psi(\dd x \, \dd s) \right] \\
  &= \E \left[ \int_{\X \times \Ss} g(x, s, \Psi - \delta_{(x, s)}) \Psi(\dd x \, \dd s) \right] 
\end{align*}
which proves \eqref{eq:ans} and the results follows. Similar considerations hold true to prove the second part of the lemma, which involves the distribution of the $n$-th Palm measure.  
\end{proof}

\begin{lemma}\label{lemma:laplace_marked}
    Let $\tilde \mu(A) = \int_{A \times \R_+} s \Psi(\dd x \,\dd s)$ where $\Psi = \sum_{j \geq 1} \delta_{( X_j , S_j)}$ is a marked point process obtained by marking $\Phi = \sum_{j \geq 1 } \delta_{X_j}$ with i.i.d. marks $S_j$ from a distribution $H$ on $\R_+$. Then for any measurable function $f : \X \to \R_+$:
    \[
        \E\left[ \e^{- \int_X f(x) \mutilde(\dd x)} \right] = \E\left[ \exp\left(\int_\X \log \psi(f(x)) \Phi(\dd x) \right) \right]
    \]
    where $\psi(f(x)): = \int_{\R_+} \e^{- s f(x)} H(\dd s)$ is the Laplace transform of $H$ evaluated at $f(x)$.
\end{lemma}
\begin{proof}
By exploiting the tower property of expected values, the Laplace functional equals:
    \begin{align*}
    \E   \left[ \exp \left\{ - \int_\X f(x) \mutilde(\dd x)\right\}\right] & = \E \left[ \exp  \left\{ - \sum_{j \geq 1 }
    S_j  f (X_j) \right\} \right]\\
    &  = \E \left[  \E \left[ \exp \left\{ - \sum_{j \geq 1}
    S_j  f (X_j) 
   \right\}  \Big|   \Phi  \right]  \right]\\
    &  = \E \left[  \prod_{j \geq 1 }\E \left[ \exp \left\{ - 
    S_j  f (X_j )
   \right\}  \Big|   \Phi \right]  \right]\\
     &  = \E \left[  \prod_{j \geq 1 }
     \int_{\R_+}  \e^{- s f (X_j) } H (\dd s)\right] \\
     &= \E \left[  \exp \left\{ \sum_{j \geq 1} \log \left(
     \int_{\R_+}  \e^{- s f (X_j) } H (\dd s) \right) \right\}\right] \\
     &= \E\left[ \exp\left(\int_\X \log \psi(f(x)) \Phi(\dd x) \right) \right]
\end{align*}
and the result follows.
\end{proof}

The following Lemma provides a convenient way of representing the  Palm kernel of the superposition of independent point processes. It will be useful in Appendix \ref{app:sncp}.

\begin{lemma}[Theorem 1 in \citealp{BerCamPalm}]\label{lem:superp_palm}
    Let $\Phi_1$ and $\Phi_2$ be two simple and independent point processes. Then the  Palm kernel $(\Phi_1 + \Phi_2)_{x}$ can be expressed as a mixture, i.e.
    \[
        (\Phi_1 + \Phi_2)_{x} \deq \begin{cases}
            \Phi_{1 x} + \Phi_2 & \text{with probability } \frac{M_{\Phi_1}(\dd x)}{M_{\Phi_1}(\dd x) + M_{\Phi_2}(\dd x)} \\
            \Phi_{1} + \Phi_{2 x} & \text{with probability } \frac{M_{\Phi_2}(\dd x)}{M_{\Phi_1}(\dd x) + M_{\Phi_2}(\dd x)}
        \end{cases}
    \]
\end{lemma}

\begin{remark}
    Since $\Phi_1$ and $\Phi_2$ are simple, also $\Phi$ is. Then, exploiting the relation $\Phi_x \deq \Phi_x^! + \delta_x$, we obtain an equivalent statement of Lemma \ref{lem:superp_palm} in terms of the reduced Palm kernel, namely:
    \[
        (\Phi_1 + \Phi_2)^!_{x} \deq \begin{cases}
            \Phi^!_{1 x} + \Phi_2 & \text{with probability } \frac{M_{\Phi_1}(\dd x)}{M_{\Phi_1}(\dd x) + M_{\Phi_2}(\dd x)} \\
            \Phi_{1} + \Phi^!_{2 x} & \text{with probability } \frac{M_{\Phi_2}(\dd x)}{M_{\Phi_1}(\dd x) + M_{\Phi_2}(\dd x)}
        \end{cases}
    \]
\end{remark}
Lemma \ref{lem:superp_palm} can be extended to deal with more than two independent processes by induction.
\begin{lemma}[Corollary 1 in \cite{BerCamPalm}]\label{lem:superp_palm2}
    Let $\Phi_1, \ldots, \Phi_k$ be independent simple point processes, then
    \[
        \left(\sum_{j=1}^k \Phi_j\right)_x \deq \Phi_{j x} + \sum_{\ell \neq j} \Phi_\ell \quad \text{with probability proportional to } M_{\Phi_j}(\dd x).
    \]
    Moreover, the same results hold true for the corresponding reduced Palm kernels.
\end{lemma}

\section{Prior analysis}\label{app:prior}

We study here the properties of $\tilde{p} \sim {\rm nRM} (\plaw_\Phi; H)$, useful for prior elicitation. In particular, here 
 we derive expressions for the expectations $\E[\mutilde(A)]$, $\E[\tilde{p}(A)]$, and the covariance $\mbox{Cov}(\mutilde(A), \mutilde(B))$.

\begin{proposition}\label{prop:prior_moms}
  Let $\tilde \mu \sim {\rm RM}(\plaw_{\Phi}; H)$ and $\E [S] := \int_{\R_+} s H(\dd s)$ be the expected value of a random variable $ S\sim H$.  We have:
  \begin{itemize}
      \item[(i)] for any measurable set $A$, $\E[\mutilde(A)] = M_{\Phi}(A) \E[S]$, and 
      \[    
        \E[\ptilde(A)] = \int_{\R_+} \kappa(u, 1)  \int_A \E\left[\e^{- u \mutilde_x^! (\X)} \right]  M_{\Phi}(\dd x) \dd u
      \]
      where $\kappa(u, n):= \E[\e^{-uS} S^n]$ and $\mutilde_x^!$ is as in \eqref{eq:palm_mu};
      \item[(ii)] for arbitrary measurable sets $A,B \in \mathcal{X}$
      \[
        \Cov(\mutilde(A),\mutilde(B)) = \E[S^2] M_{\Phi}(A \cap B) + \E[S]^2 \left(M_{\Phi^{(2)}}(A \times B) - M_{\Phi}(A) M_{\Phi}(B)\right);
      \]
      
      \item[(iii)] the mixed moments of $\tilde p$ are given by
      \begin{multline*}
       \E[\ptilde(A) \ptilde(B)] =  \int_{\R_+} u \Bigg\{ \kappa(u, 1)^2 \int_{A \times B} \E\left[\e^{-u \mutilde^!_{\bm x}  (\X)} \right] M_{\Phi^{(2)}}(\dd \bm x)  \\ +  \kappa(u, 2) \int_{A \cap B}\E\left[\e^{-u \mutilde^!_{x} (\X)} \right] M_{\Phi}(\dd x) \Bigg\} \dd u
      \end{multline*}
      where $\bm x = (x_1, x_2)$ and $A,B \in \mathcal{X}$.
  \end{itemize}
\end{proposition}
We now specialise the previous proposition to examples of interest.
\begin{proof}
    We prove the three parts of the proposition separately.

\begin{enumerate}
    \item[(i)] For any Borel set $A$, one has:
    \begin{multline*}
        \E[\mutilde(A)] = \E \left[ \int_{A \times \R_+} s \Psi(\dd x \, \dd s) \right] \\ = \int_{A \times \R_+} \int_{\mathbb M(A \times \R_+)} s \plaw_{\Psi}^{x, s} (\dd \nu) M_{\Psi}(\dd x \, \dd s) = \int_{A \times \R_+} s H(\dd s) M_\Phi(\dd x)
    \end{multline*}
    where the second equality follows from Theorem~\ref{teo:clm} with $g(x, s, \Psi) = s$, while the third follows from Lemma~\ref{prop:n_mean}.

    Recalling the definition of $\ptilde$ and exploiting the identity $x^{-1} = \int_0^\infty \e^{-u x} \dd u$ we have
    \begin{multline*}
      \E[\tilde p (A)]  =  \E \left[  \frac{\tilde \mu (A)}{\tilde \mu (\X)} \indicator\left[\tilde \mu (\X) >0\right]\right] \\=  
      \E\left[ \indicator\left[\tilde \mu (\X ) >0\right]\int_{\R_+} \dd u \int_{A \times \R_+} \exp\left( - \int_{\X \times \R_+} u t \Psi(\dd z \, \dd t) \right) s \Psi(\dd x \, \dd s)\right].
  \end{multline*}
  We can further exchange the outermost expectation with the integral with respect to $\dd u$ by the Fubini theorem. Then, we apply the Campbell-Little-Mecke formula to get
  \begin{equation} \label{eq:EP}
    E[\tilde p (A)]= \int_{\R_+} \int_{A \times \R_+} \e^{-us} s \times\E\left[
     \indicator\left[s + \tilde \mu^{!}_x (\X) >0\right] \e^{- \int_{\X \times \R_+} u v \Psi^!_{x, s} (\dd z \, \dd v)} \right]  M_{\Phi}(\dd x) H(\dd s)   \dd u
  \end{equation}
  where $\Psi^!_{x, s}$ is the reduced Palm kernel of $\Psi$ and we recall that
  \[
  \tilde \mu^{!}_x (A) : = \int_{A \times \R_+}  t  \Psi^!_{x, s} (\dd x \, \dd t).
  \]
  We observe that the indicator function appearing in \eqref{eq:EP} equals $1$ almost surely,  since $H(0)=0$, being $H$ the probability distribution of the jumps. 
  
  The same argument applies in all our proofs; thus, the fact that  $\prob(\tilde \mu (\X ) =0 ) >0$  is immaterial.

\item[(ii)] We can evaluate the covariance as follows
    \[
        \Cov(\mutilde(A), \mutilde(B)) = \E[\mutilde(A) \mutilde(B)] - \E[\mutilde(A)] \E[\mutilde(B)].
    \]
    Focusing on the first term, we get
    \begin{align*}
        \E[\mutilde(A) \mutilde(B)] = \E \left[ \int_{A \times \R_+} \int_{B \times \R_+} s t \Psi(\dd x \, \dd s) \Psi(\dd z \, \dd t)\right].
    \end{align*}
    Let $\bm s = (s, t)$ and $\bm x = (x, z)$, so that
    \[
        \E[\mutilde(A) \mutilde(B)] = \E \left[ \int_{(\X \times \R_+)^2} \indicator_{A \times B}(\bm x) s t \Psi^2(\dd \bm x \, \dd \bm s) \right].
    \]    
    Now, let $\Theta = \X \times \R_+$ and define $\mathcal D_{\Theta^2} = \{(\theta_1, \theta_2) \in \Theta^2 \text{ s.t. } \theta_1 = \theta_2\}$. An application of Theorem~\ref{teo:clm_mult} with $g(\bm x, \bm s, \Psi) = \indicator_{A \times B}(\bm x) s t$ yields 
    \begin{align*}
        \E[\mutilde(A) \mutilde(B)]  &= \int_{\Theta^2} \indicator_{A \times B}(\bm x) s t M_{\Psi^2}(\dd \bm x \, \dd \bm s) \\
        &= \int_{\mathcal D_{\Theta^2}} \indicator_{A \cap B}(x) s^2  M_{\Psi}(\dd s \, \dd x) + \int_{\Theta^2 \setminus \mathcal D_{\Theta^2}} \indicator_{A \times B}(\bm x) s t  M_{\Psi^{(2)}}(\dd \bm{s} \, \dd \bm x)
    \end{align*}
   \MBtext{where the last equality is a consequence of the following relation between the second order moment measure of $\Psi$ and the corresponding factorial moment measures
    \[
    M_{\Psi^2}(\dd \bm x \, \dd \bm s) = M_{\Psi}(\dd x, \dd s) \delta_{(x, s)}(\dd z, \dd t) + M_{\Psi^{(2)}}(\dd \bm x \, \dd \bm s) \indicator[x \neq z, s \neq t],
    \]
    as can be proved by a simple application of  \citep[Lemma 14.E.4]{BaBlaKa}. Now, we can exploit 
     \Cref{prop:n_mean} to show that
     \[
   \E[\mutilde(A) \mutilde(B)] = \int_{\mathcal D_{\Theta^2}} \indicator_{A \cap B}(x) s^2 H(\dd s) M_{\Phi}(\dd x) + \int_{\Theta^2 \setminus \mathcal D_{\Theta^2}} \indicator_{A \times B}(\bm x) s t H(\dd s) H(\dd t) M_{\Phi^{(2)}}(\dd \bm x) .
     \]
    The conclusion follows from standard algebra.
    }

\item[(iii)]
To compute $\E[\ptilde(A) \ptilde(B)]$ we can proceed as above and write
\begin{align*}
    \E[\ptilde(A) \ptilde(B)] &= \E \left[\frac{1}{T^2} \mutilde(A) \mutilde(B) \right] \\
    & = \int_{\R_+} \E \left[ \int_{\Theta^2}u \e^{- \int u t \Psi(\dd z \, \dd t)} s_1 s_2  \indicator_{A \times B}(x_1, x_2) \Psi^2( \dd \bm x \, \dd \bm s)  \right] \dd u .
\end{align*}
Again, from the relation between moment and factorial moment measures, we write
\begin{align*}
    \E[\ptilde(A) \ptilde(B)] &= \int_{\R_+} u \Bigg\{ \E\left[\int_{\Theta^2 \setminus \mathcal D_{\Theta^2}} \e^{-u \int t \Psi(\dd z \, \dd t)} s_1 s_2 \indicator_{A \times B}(x_1, x_2) \Psi^{(2)}(\dd \bm x \, \dd \bm s)    \right] \\
    & \qquad \qquad  + \E \left[ \int_{\mathcal D_{\Theta^2}} \e^{-u  \int t \Psi(\dd z \, \dd t)} s_1^2 \indicator_{A \cap B}(x) \Psi(\dd x \, \dd s) \right] \Bigg \} \dd u \\
    & = \int_{\R_+} u \Bigg\{ \int_{\Theta^2 \setminus \mathcal D_{\Theta^2}} \E[\e^{-u \mutilde^!_{\bm x}(\X)}] \e^{-u(s_1 + s_2)} s_1 s_2 \indicator_{A \times B}(x_1, x_2) H(\dd s_1) H(\dd s_2) M_{\Phi^{(2)}}(\dd \bm x)  \\
    & \qquad \qquad  +  \int_{\mathcal D_{\Theta^2}} \E[\e^{-u \mutilde^!_{x} (\X)}] \e^{-us} s^2 \indicator_{A \cap B}(x) H(\dd s) M_{\Phi}(\dd x)\Bigg\} 
\end{align*}
and the proof follows.
\end{enumerate}
\end{proof}

\begin{example}[Determinantal point process]
If $\Phi$ is a DPP, the first moment measure is equal to $M_\Phi(A) = \int_{A} K(x, x) \omega(\dd x)$.
Moreover, the second factorial moment measure of $\Phi$ is given by
\[
    M_{\Phi^{(2)}}(A \times B) = \int_{A \times B} \det \{K(z_1, z_2)\}_{z_1, z_2 \in \{x_1, x_2\}} \omega(\dd x_1) \omega(\dd x_2).
\]
When $K(x, y) = \rho \e^{- \|x - y\|^2 / \alpha}$ for $\rho,\alpha>0$, $\X$ is a compact subset of $\R^q$ and $\omega$ is the Lebesgue measure,   we recover the Gaussian-DPP. In this case $M_\Phi(A) = \rho \times |A|$, where $|\cdot|$ is the Lebesgue measure of a set, and 
\[
     M_{\Phi^{(2)}}(A \times B) = \rho^2\int_{A \times B} 1 - \e^{- 2 \|x - y\|^2 / \alpha} \dd x  \dd y.
\]
\end{example}

\begin{example}[Shot-Noise Cox Process]
    It is easy to see that $M_\Phi(A) = \gamma \int_A \int_\X k_\alpha(x - v) \omega(\dd v) \dd x$.
    Moreover, by putting
\begin{equation*}\label{eq:eta_cox}
    \eta(x_1, \ldots, x_{l}) = \int \prod_{i=1}^l k_{\alpha}(x_i - v) \omega(\dd v)
\end{equation*}
and using \Cref{lemma:moment_cox} in Appendix \ref{app:sncp}, we specialize expressions in \Cref{prop:prior_moms} as follows
\begin{align*}
    \E[\mutilde(A)] &= \E[S] \gamma \int_A \eta(x) \dd x \\
    \Cov(\mutilde(A), \mutilde(B)) &= \E[S^2] \gamma \int_{A \cap B} \eta(x) \dd x \  \\ 
    & \qquad + \E[S]^2 \gamma^2 \left\{ \int_{A \times B} \eta(x, y) \dd x \, \dd y   - \int_{A} \eta(x) \dd x  \int_{ B} \eta(x) \dd x \right\}.
\end{align*} 

\end{example}

\section{Proofs of the main results}
\label{app:proofs}

\subsection{Proof of Theorem~\ref{teo:post}}

We consider the Laplace functional of $\mutilde$ given $\bm Y = \bm y$ and $U_n = u$:
\[
	\E \left[ \e^{-\int_\X f(z) \mutilde(\dd z)} \midd \bm Y = \bm y, U_n = u \right] = 
	\frac{\E \left[ \e^{-\int_\X f(z) \mutilde(\dd z)} \prob(\bm Y \in \dd \bm y, U_n \in \dd u \mid \mutilde)  \right]}{\E \left[ \prob(\bm Y \in \dd \bm y, U_n \in \dd u \mid \mutilde)  \right] }
\]
where we observe that the denominator is obtained as a special case of the numerator by letting $f(x) = 0$.
\MBtext{The equality formally follows by the \textit{abstract Bayes' theorem}, see, for instance,  Proposition B.41 in \cite{bjork2009arbitrage}. In particular, in Bj{\"o}rk's notation, $Q$ is the posterior of $\mutilde$, $P$ is the prior of $\mutilde$, $L$ is $\prob(\bm Y \in \dd \bm y, U_n \in \dd u \mid \mutilde)$ and $\mathcal G = \{\Omega, \emptyset\}$.
Hence,}
we now focus on the numerator:
\begin{align*}
	& \E \left[ \e^{-\int_\X f(z) \mutilde(\dd z)} \prob(\bm Y \in \dd \bm y, U_n \in \dd u \mid \mutilde) \right] = \E \left[ \e^{-\int_\X f(z) \mutilde(\dd z)} \frac{u^{n-1}}{\Gamma(n)} \prod_{j=1}^k \mutilde(\dd y^*_j)^{n_j} \e^{-Tu} \right]\\
	&= \frac{u^{n-1}}{\Gamma(n)} \E \left[ \e^{-\int_\X (f(z) + u) \mutilde(\dd z)} \prod_{j=1}^k \mutilde(\dd y^*_j)^{n_j}  \right] \\
	&= \frac{u^{n-1}}{\Gamma(n)} \E \left[ \int_{(\X \times \R_+)^k} \e^{-\int_\X (f(z) + u) \mutilde(\dd z)} \prod_{j=1}^k s_j^{n_j} \delta_{x_j}(\dd y^*_j) \Psi(\dd x_j \, \dd s_j)  \right].
\end{align*}
Note that, for the sake of notational simplicity, we have omitted the indicator function $\indicator[\mutilde(\mathbb X) > 0]$ since this is effectively immaterial as discussed below \eqref{eq:EP}.
We set
\[
    g(\bm x, \bm s, \Psi) = \e^{-\int_\X (f(z) + u) \mutilde(\dd z)} \prod_{j=1}^k s_j^{n_j} \delta_{x_j}(\dd y^*_j),
\]
and by an application of Theorem~\ref{teo:clm_mult},  we obtain that
\begin{equation}\label{eq:piaceafederico}
    \E \left[ \e^{-\int_\X f(z) \mutilde(\dd z)} \prob(\bm Y \in \dd \bm y, U_n \in \dd u \mid \mutilde) \right]  = \frac{u^{n-1}}{\Gamma(n)} \int_{(\X \times \R_+)^k} \E_{\Psi \sim \plaw_{\Psi}^{\bm x, \bm s}} \left[g(\bm x, \bm s, \Psi)\right] M_{\Phi^{(k)}}(\dd \bm x) H^{k}(\dd \bm s)
\end{equation}
where we have used $M_{\Psi^k}(\dd \bm x \, \dd \bm s) =  M_{\Psi^{(k)}}(\dd \bm x \, \dd \bm s)$ (i.e., we pass from the moment measure to the factorial one since the $x_j$'s can be considered pairwise distinct thanks to the term $\prod_{j=1}^k \delta_{x_j}(\dd y^*_j)$), and that $M_{\Psi^{(k)}}(\dd \bm x\, \dd \bm s) = M_{\Phi^{(k)}}(\dd \bm x) H^{k}(\dd \bm s)$ as in Lemma \ref{prop:n_mean}.

By the properties of the Palm distribution, $\Psi \sim \plaw_{\Psi}^{\bm x, \bm s}$ has the same law of $\sum_{j=1}^k \delta_{(x_j, s_j)} + \Psi^!_{\bm{x}, \bm{s}}$. Moreover, following Proposition \ref{prop:marked_palm}, $\Psi^!_{\bm{x}, \bm{s}} = \sum_{j \geq 1 } \delta_{(\tilde X_j, \tilde S_j)}$ where $\Phi^!_{\bm{x}} := \sum_{j \geq 1 } \delta_{\tilde X_j} \sim \plaw_{\Phi^{!}}^{\bm{x}}$ and the $\tilde S_j$'s are i.i.d. from $H$. Therefore, as in the main text, we use the short-hand notation $\Psi^!_{\bm{x}} \equiv \Psi^!_{\bm{x}, \bm{s}}$
Hence
\[
    \E_{\Psi \sim \plaw_{\Psi}^{\bm x, \bm s}} \left[g(\bm x, \bm s, \Psi)\right] = 
    \E \left[g\left(\bm x, \bm s, \sum_{j=1}^k \delta_{(x_j, s_j)} + \Psi^!_{\bm{x}}\right) \right] 
\]
where the expected value on the right-hand side is taken with respect to $\Psi^!_{\bm{x}}$.
Let $\mutilde_{\bm x}^!(A) := \int_{A \times \R_+} s \Psi^!_{\bm{x}}(\dd x \, \dd s)$, we have:
\begin{equation*}
    g\left(\bm x, \bm s, \sum_{j=1}^k \delta_{(x_j, s_j)} + \Psi^!_{\bm{x}}\right) = 
        \exp \left( - \int_\X (f(z) + u) \left( \sum_{j=1}^k s_j \delta_{x_j}(\dd z) + \mutilde^!_{\bm x}(\dd z) \right) \right) \prod_{j=1}^k s_j^{n_j} \delta_{x_j}(\dd y^*_j) .
\end{equation*}
Hence, \eqref{eq:piaceafederico} boils down to
\begin{align*}
    & \E \left[ \e^{-\int_\X f(z) \mutilde(\dd z)} \prob(\bm Y \in \dd \bm y, U_n \in \dd u \mid \mutilde) \right] \\
    & = \frac{u^{n-1}}{\Gamma(n)} \int_{(\X \times \R_+)^k} \E \left[ \exp \left( - \int_\X (f(z) + u) \left( \sum_{j=1}^k s_j \delta_{x_j}(\dd z) + \mutilde_{\bm x}^!(\dd z) \right) \right) \prod_{j=1}^k s_j^{n_j} \delta_{x_j}(\dd y^*_j) \right] \\
    & \qquad \qquad \qquad \qquad \times M_{\Phi^{(k)}}(\dd \bm x) H^{k}(\dd \bm s) \\
    &= \frac{u^{n-1}}{\Gamma(n)} \int_{(\X \times \R_+)^k} \E \left[ \e^{- \int_\X (f(z) + u) \mutilde_{\bm x}^!(\dd z)} \right] \exp \left( - \int_\X \left( f(z) + u \right) \left( \sum_{j=1}^k s_j \delta_{x_j}(\dd z) \right) \right) \\
    & \qquad \qquad \qquad \qquad \times \prod_{j=1}^k s_j^{n_j} \delta_{x_j}(\dd y^*_j) M_{\Phi^{(k)}}(\dd \bm x) H^{k}(\dd \bm s) \\
    &= \frac{u^{n-1}}{\Gamma(n)} \E \left[ \e^{- \int_\X (f(z) + u) \mutilde_{\bm y^*}^!(\dd z)} \right] M_{\Phi^{(k)}}(\dd \bm y^*) \prod_{j=1}^k \int_{\R_+} \e^{-(f(y^*_j) + u) s_j} s_j^{n_j} H(\dd s_j).
\end{align*}
Setting $f=0$ yields the denominator:
\begin{equation}
    \E \left[\prob(\bm Y \in \dd \bm y, U_n \in \dd u \mid \mutilde) \right] = \frac{u^{n-1}}{\Gamma(n)} \E \left[ \e^{- \int_\X u \mutilde_{\bm y^*}^!(\dd z)} \right] M_{\Phi^{(k)}}(\dd \bm y^*) \prod_{j=1}^k \int_{\R_+} \e^{-u s_j} s_j^{n_j} H(\dd s_j).
\label{eq:den} 
\end{equation}
Letting $\kappa(u, n_j) = \int_{\R_+} \e^{-u s_j} s_j^{n_j} H(\dd s_j)$, we obtain
\begin{align}
        & \E \left[ \exp \int_\X - f(z) \mutilde(\dd z) \midd \bm Y = \bm y, U_n = u \right] =
    	\frac{\E \left[ \e^{- \int_\X (f(z) + u) \mutilde_{\bm y^*}^!(\dd z)} \right]}{\E \left[ \e^{- \int_\X u \mutilde_{\bm y^*}^!(\dd z)} \right]} \prod_{j=1}^k \int_{\R_+} \frac{\e^{-f(y^*_j)s_j} \e^{-u s_j} s_j^{n_j}}{\kappa(u, n_j)} H(\dd s_j)\nonumber\\
     & \qquad\qquad = 
     \E \left[ \e^{- \int_\X f (z) \mutilde_{\bm y^*}^!(\dd z)} \cdot \frac{\e^{- u \mutilde_{\bm y^*}^!(\X )}}{\E \left[ \e^{-u \mutilde_{\bm y^*}^!(\X)}\right]}  \right] \prod_{j=1}^k \int_{\R_+} \frac{\e^{-f(y^*_j)s_j} \e^{-u s_j} s_j^{n_j}}{\kappa(u, n_j)} H(\dd s_j).   \label{eq:posterior_final_proof}
\end{align}
The product term over $j =1, \ldots, k$ in the previous expression equals the Laplace transform of $\sum_{j=1}^k S^*_j \delta_{Y^*_j}$ where the $S^*_j$'s are independent positive random variables with density proportional to $\e^{-u s_j} s_j^{n_j}  H(\dd s_j)$. 
Moreover, the first expected value in \eqref{eq:posterior_final_proof}  corresponds to the Laplace transform of a random measure $\mutilde^\prime$, whose probability distribution is absolutely continuous with respect to the distribution of $\mutilde_{\bm y^*}^!$ with associated density
\[
f_{\mutilde^\prime} (\mu) := \frac{\e^{-u \mu (\X)}}{\E \left[ \e^{-u \mutilde_{\bm y^*}^! (\X)}\right]}.
\]
Finally, the conditional distribution of $U_n$ easily follows from \eqref{eq:den} by conditioning on $\bm Y$.

\subsection{Proof of Theorem \ref{teo:marg}}

The proof follows by integrating \eqref{eq:den} with respect to $u$.

\subsection{Proof of Theorem \ref{teo:pred_conditional}}

Consider a Borel set $A \in \calX$, the predictive distribution equals
\begin{equation}
    \label{eq:predittiva_p1}
    \begin{split}
    \prob (Y_{n+1} \in A \mid \bm{Y} = \bm{y} )  & = \E \left[ \prob \left( Y_{n+1} \in A \mid \bm{Y} = \bm{y}  , \tilde p    \right) \mid \bm{Y} = \bm{y} \right]\\
    & = \E \left[ \tilde p (A) \mid \bm{Y} = \bm{y} \right]  = \int_0^\infty \E \left[ \tilde p (A) \mid \bm{Y} = \bm{y} , U_n=u\right]  f_{U_n | \bm{Y}} (u)  \dd u
    \end{split}
\end{equation}
where we used standard properties of conditional expectations. We now concentrate on the evaluation of the expected value 
$\E \left[ \tilde p (A) \mid \bm{Y} = \bm{y} , U_n=u\right] $ in \eqref{eq:predittiva_p1}. Thanks to the posterior representation of Theorem~\ref{teo:post}, we obtain
\begin{align} 
    &\E \left[ \tilde p (A) \mid \bm{Y} = \bm{y} , U_n=u\right] = \E \left[  \frac{\mutilde^\prime (A) + \sum_{j=1}^k S_j^* \delta_{y_j^*} (A)}{\mutilde^\prime (\X) + \sum_{j=1}^k S_j^*}  \right] \nonumber \\
    &\qquad = \E \left[  \frac{\mutilde^\prime (A)}{\mutilde^\prime (\X) + \sum_{j=1}^k S_j^*}  \right] +
    \E \left[  \frac{\sum_{j=1}^k S_j^* \delta_{y_j^*} (A)}{\mutilde^\prime (\X) + \sum_{j=1}^k S_j^*}  \right] \nonumber  \\
    & \qquad = \int_0^\infty \E \left[ \e^{-v \sum_{j=1}^k S_j^*}\right] \E \left[ \e^{-v \mutilde^\prime (\X) } \mutilde^\prime (A)\right] \dd v + \sum_{\ell=1}^k
    \int_0^\infty \E \left[ \e^{-v \sum_{j=1}^k S_j^*}  S_\ell^* \delta_{y_\ell^*} (A)\right] \E \left[ \e^{-v \mutilde^\prime (\X) } \right] \dd v \nonumber  \\
    & \qquad =: I_1^{(n)} (u; \bm{y})+ I_2^{(n)} (u; \bm{y}),\label{eq:prob_p_predittiva}
\end{align}
where we have observed that the $S_j^*$'s and $\mutilde^\prime$ are independent (see Theorem~\ref{teo:post}).
We now focus on the evaluation of the two terms in \eqref{eq:prob_p_predittiva} separately. As for the first integral, we have:
\begin{equation*}
    \begin{split}
        I_1^{(n)} (u; \bm{y}) &= \int_0^\infty \E \left[ \e^{-v \sum_{j=1}^k S_j^*}\right] \E \left[ \e^{-v \mutilde^\prime (\X) } \mutilde^\prime (A)\right] \dd v = \int_0^\infty \prod_{j=1}^k\E \left[ \e^{-v S_j^*}\right] \E \left[ \e^{-v \mutilde^\prime (\X) } \mutilde^\prime (A)\right] \dd v\\
        & = \int_0^\infty \prod_{j=1}^k \frac{\kappa (u+v , n_j)}{\kappa (u, n_j)} \E \left[ \e^{-v \mutilde^\prime (\X)} \mutilde^\prime (A)\right] \dd v
    \end{split}
\end{equation*}
where we have used the density of the $S_j^*$'s (see Theorem~\ref{teo:post}). Now, we exploit the fact that the distribution of $\mutilde^\prime$ is absolutely continuous with respect to the distribution of $\mutilde_{\bm y^*}^!$ with available Radon-Nikodym derivative, and we get
\begin{equation} \label{eq:proof_pred2}
    \begin{split}
        I_1^{(n)} (u; \bm{y}) &= \int_0^\infty \prod_{j=1}^k \frac{\kappa (u+v , n_j)}{\kappa (u, n_j)}
        \E \left[   \e^{-v \mutilde_{\bm y^*}^! (\X)}  \mutilde_{\bm y^*}^! (A)  \frac{\e^{-u \mutilde_{\bm y^*}^! (\X)} }{\E \left[ \e^{-u \mutilde_{\bm y^*}^! (\X)} \right]}\right]\dd v\\
        & =\frac{1}{\E \left[ \e^{-u \mutilde_{\bm y^*}^! (\X)}\right]} \times \int_0^\infty \prod_{j=1}^k \frac{\kappa (u+v , n_j)}{\kappa (u, n_j)}\E \left[  \int_{A \times \R_+} \e^{- (u+v)  \mutilde_{\bm y^*}^! (\X)}  s  \Psi^!_{\bm y^*} (\dd y \, \dd s)\right] \dd v.
    \end{split}
\end{equation}
The expected value inside the integral in \eqref{eq:proof_pred2} can be evaluated resorting to the Campbell-Little-Mecke formula, and we obtain:
\begin{equation} \label{eq:I1_Mphi}
\begin{split}
    & I_1^{(n)} (u; \bm{y}) = \int_0^\infty \prod_{j=1}^k \frac{\kappa (u+v , n_j)}{\kappa (u, n_j)} \int_A \int_{\R_+} s \e^{-(u+v)s}
    \frac{\E \left[   \e^{- (u+v)  \mutilde_{(\bm y^*, y)}^! (\X)}  \right]}{\E \left[ \e^{-u \mutilde_{\bm y^*}^! (\X)}\right]} H (\dd s ) M_{\Phi^{!}_{\bm y^*}} (\dd y)\dd v
    \end{split}
\end{equation}
where we used the Palm algebra, i.e., $(\mutilde_{\bm y^*}^!)^!_{y} = \mutilde_{(\bm y^*, y)}^!$. Since the factorial moment measure $M_{\Phi^{(k)}}$ is absolutely continuous with respect to $P_0^k$ with Radon-Nikodym derivative $m_{\Phi^k}$, then also the moment measure $M_{\Phi^{!}_{\bm y^*}} $ is absolutely continuous with respect to $P_0$. 
A proof of this simple fact can be obtained by applying \cite[Proposition 3.3.9]{BaBlaKa}, which states that
the following equality holds
\[
M_{\Phi^{(k+1)}} (A \times B) = \int_A M_{\Phi^!_{\bm{x}}}  (B) M_{\Phi^{(k)}} (\dd \bm{x})
\]
for any $A \in \calX^k$ and $B \in \calX$. By the fact that the factorial moment measures are absolutely continuous, the previous equality boils down to
\[
\int_A \int_B m_{\Phi^{k+1}} (\bm x, y) P_0 (\dd y) P_0^{k } (\dd \bm x) =
\int_A M_{\Phi^!_{\bm{x}}}  (B) m_{\Phi^k } (\bm x)  P_0^{k } (\dd \bm x) , 
\]
thus, since the previous equality holds for any $A \in \calX^k$, the density of the moment measure $M_{\Phi^!_{\bm{x}}}$ with respect to $P_0$ equals $m_{\Phi^{k+1}} (\bm x, y)/ m_{\Phi^k} (\bm x)$, for $P_0^k$-almost all $\bm x$. As a consequence, 
$I_1^{(n)} (u; \bm{y})$ in \eqref{eq:I1_Mphi} can be written as
\begin{equation} \label{eq:expression_I1}
\begin{split}
    I_1^{(n)} (u; \bm{y}) &= \int_0^\infty \kappa (u+v,1) \prod_{j=1}^k \frac{\kappa (u+v , n_j)}{\kappa (u, n_j)} \int_A 
    \frac{\E \left[   \e^{- (u+v)  \mutilde_{(\bm y^*, y)}^! (\X)}  \right]}{{\E \left[ \e^{-u \mutilde_{\bm y^*}^! (\X)}\right]}}  \frac{m_{\Phi^{k+1}} (\bm y^*, y)}{m_{\Phi^{k}} (\bm y^*) } P_0(\dd y)\dd v.
    \end{split}
\end{equation}
As for the second integral in \eqref{eq:prob_p_predittiva}, one can easily exploit the distribution of the jumps $S_j^*$'s
 and $\mutilde^\prime$, available in Theorem~\ref{teo:post}, to show that
\begin{equation}
    \label{eq:expression_I2}
    \begin{split}
    I_2^{(n)} (u; \bm{y}) & =\sum_{\ell=1}^k
    \int_0^\infty \E \left[ \e^{-v \sum_{j=1}^k S_j^*}  S_\ell^* \delta_{y_\ell^*} (A)\right] \E \left[ \e^{-v \mutilde^\prime (\X) } \right] \dd v \\
    &  = \sum_{\ell=1}^k  \int_0^\infty \prod_{j=1}^k \frac{\kappa (u+v, n_j)}{\kappa (u, n_j)}  \cdot
    \frac{\kappa (u+v, n_\ell +1)}{\kappa (u+v, n_\ell)}  \frac{\E \left[ \e^{- (u+v) \mutilde_{\bm y^*}^! (\X) } \right]}{
    \E \left[ \e^{- u \mutilde_{\bm y^*}^! (\X) } \right]} \dd v \times \delta_{y_\ell^*} (A).
    \end{split}
\end{equation}
We now remind that the predictive distribution can be obtained by integrating \eqref{eq:prob_p_predittiva} with respect to the posterior distribution of $U_n$ as follows:
\begin{equation}
    \label{eq:pred_with_integral}
    \begin{split}
     \prob (Y_{n+1} \in A \mid \bm{Y} = \bm{y} )  &=  \int_0^\infty \E \left[ \tilde p (A) \mid \bm{Y} = \bm{y} , U_n=u\right]  f_{U_n | \bm{Y}} (u)  \dd u\\
    & =     \int_0^\infty   I_1^{(n)} (u; \bm{y}) f_{U_n | \bm{Y}} (u)  \dd u + 
       \int_0^\infty   I_2^{(n)} (u; \bm{y}) f_{U_n | \bm{Y}} (u)  \dd u 
    \end{split}
\end{equation}
where $ I_1^{(n)} (u; \bm{y})$ and $ I_2^{(n)} (u; \bm{y})$ coincide with the two expressions in \eqref{eq:expression_I1} and \eqref{eq:expression_I2}, respectively. The first integral in \eqref{eq:pred_with_integral} equals
\begin{equation*}
    \begin{split}
    & \int_0^\infty   I_1^{(n)} (u; \bm{y}) f_{U_n | \bm{Y}} (u)  \dd u  \\
    & = \int_0^\infty  \int_0^\infty \kappa (u+v,1) \prod_{j=1}^k \frac{\kappa (u+v , n_j)}{\kappa (u, n_j)} \int_A 
    \frac{\E \left[   \e^{- (u+v)  \mutilde_{(\bm y^*, y)}^! (\X)}  \right]}{{\E \left[ \e^{-u \mutilde_{\bm y^*}^! (\X)}\right]}}  \frac{m_{\Phi^{k+1}} (\bm y^*, y)}{m_{\Phi^{k}} (\bm y^*) } P_0(\dd y)  f_{U_n | \bm{Y}} (u) \dd v \dd u .
    \end{split}
\end{equation*}
\MBtext{Now the change of variables $(w,z)= (u+v, u)$ implies that
\begin{align*}
    & \int_0^\infty   I_1^{(n)} (u; \bm{y}) f_{U_n | \bm{Y}} (u)  \dd u \\
    & = \int_A \int_0^\infty \int_0^w  \kappa (w,1) \prod_{j=1}^k \frac{\kappa (w , n_j)}{\kappa (z, n_j)} 
    \frac{\E \left[   \e^{- w  \mutilde_{(\bm y^*, y)}^! (\X)}  \right]}{{\E \left[ \e^{-z \mutilde_{\bm y^*}^! (\X)}\right]}}  \frac{m_{\Phi^{k+1}} (\bm y^*, y)}{m_{\Phi^{k}} (\bm y^*) }   f_{U_n | \bm{Y}} (z) \dd z \dd w P_0(\dd y)\\
     & = \int_A \int_0^\infty  \kappa (w,1) \prod_{j=1}^k \kappa (w , n_j)
    \E \left[   \e^{- w  \mutilde_{(\bm y^*, y)}^! (\X)}  \right]  \frac{m_{\Phi^{k+1}} (\bm y^*, y)}{m_{\Phi^{k}} (\bm y^*) } \\
  & \qquad \times \int_0^w  \left(\prod_{j=1}^k \kappa (z, n_j)\E \left[ \e^{-z \mutilde_{\bm y^*}^! (\X)}\right]\right)^{-1}f_{U_n | \bm{Y}} (z) \dd z \,  \dd w P_0(\dd y) .
\end{align*}
By observing that the density $f_{U_n | \bm{Y}} (z) $ is proportional to
\[
f_{U_n | \bm{Y}} (z) \propto z^{n-1} \prod_{j=1}^k \kappa (z, n_j)\E \left[ \e^{-z \mutilde_{\bm y^*}^! (\X)}\right]
\]
and writing the normalizing constant explicitly, we get that
\begin{align*}
     \int_0^\infty   I_1^{(n)} (u; \bm{y}) f_{U_n | \bm{Y}} (u)  \dd u &= \int_A \int_0^\infty  \kappa (w,1) \prod_{j=1}^k \kappa (w , n_j)
    \E \left[   \e^{- w  \mutilde_{(\bm y^*, y)}^! (\X)}  \right]  \frac{m_{\Phi^{k+1}} (\bm y^*, y)}{m_{\Phi^{k}} (\bm y^*) } \\
  &\qquad \times  \frac{\int_0^w z^{n-1} \dd z}{\int_0^\infty  u^{n-1} \prod_{j=1}^k \kappa (u, n_j)\E \left[ \e^{-u \mutilde_{\bm y^*}^! (\X)}\right] \dd u } \dd w P_0(\dd y)\\
  &= \int_A \int_0^\infty \frac{w}{n} \kappa (w,1) 
    \frac{\E \left[   \e^{- w  \mutilde_{(\bm y^*, y)}^! (\X)}  \right]}{\E \left[ \e^{-w \mutilde_{\bm y^*}^! (\X)}\right]} \frac{m_{\Phi^{k+1}} (\bm y^*, y)}{m_{\Phi^{k}} (\bm y^*) } \\
  &\qquad \times  \frac{w^{n-1} \prod_{j=1}^k \kappa (w , n_j)\E \left[ \e^{-w \mutilde_{\bm y^*}^! (\X)}\right]}{ \int_0^\infty  u^{n-1} \prod_{j=1}^k \kappa (u, n_j)\E \left[ \e^{-u \mutilde_{\bm y^*}^! (\X)}\right] \dd u } \dd w P_0(\dd y)
\end{align*}
where in the last equality we have solved the integral in the variable $z$ and we have multiplied and divided by the quantity
$\E \left[ \e^{-w \mutilde_{\bm y^*}^! (\X)}\right]$. We now observe that the expression appearing in the last line of the previous equation is exactly
\[
 f_{U_n | \bm{Y}} (w) = \frac{w^{n-1} \prod_{j=1}^k \kappa (w , n_j)\E \left[ \e^{-w \mutilde_{\bm y^*}^! (\X)}\right]}{ \int_0^\infty  u^{n-1} \prod_{j=1}^k \kappa (u, n_j)\E \left[ \e^{-u \mutilde_{\bm y^*}^! (\X)}\right] \dd u } .
\]
Thus, an application of the Fubini-Tonelli theorem yields:
\begin{equation} \label{eq:intI1}
    \begin{split}
    & \int_0^\infty   I_1^{(n)} (u; \bm{y}) f_{U_n | \bm{Y}} (u)  \dd u  \\
    & \qquad\qquad = \int_0^\infty  \frac{w}{n} \kappa (w,1)\int_A \frac{\E \left[   \e^{- w  \mutilde_{(\bm y^*, y)}^! (\X)}  \right]}{{\E \left[ \e^{-w \mutilde_{\bm y^*}^! (\X)}\right]}}  \frac{m_{\Phi^{k+1}} (\bm y^*, y)}{m_{\Phi^{k}} (\bm y^*) } P_0(\dd y)  f_{U_n | \bm{Y}} (w) \dd w . 
    \end{split}
\end{equation}}
Analogous considerations show that the second integral in \eqref{eq:pred_with_integral} equals
\begin{equation}
    \label{eq:intI2}
    \begin{split}
        &\int_0^\infty   I_2^{(n)} (u; \bm{y}) f_{U_n | \bm{Y}} (u)  \dd u = \sum_{j=1}^k \int_0^\infty  \frac{w}{n} \frac{\kappa (w, n_j +1)}{\kappa (w, n_j)}  f_{U_n | \bm{Y}} (w) \dd w  \delta_{y_j^*} (A).
    \end{split}
\end{equation}
By substituting Equations \eqref{eq:intI1}--\eqref{eq:intI2} in the expression of the predictive distribution \eqref{eq:pred_with_integral}, we finally obtain
\begin{equation*}
    \begin{split}
     &    \prob (Y_{n+1} \in A \mid \bm{Y} = \bm{y} )  =  \int_0^\infty  \frac{w}{n} \left[  \sum_{j=1}^k  \frac{\kappa (w, n_j +1)}{\kappa (w, n_j)}    \delta_{y_j^*} (A) \right.\\
     & \qquad\qquad\qquad \left. +   \kappa (w,1)\int_A \frac{\E \left[   \e^{- w  \mutilde_{(\bm y^*, y)}^! (\X)}  \right]}{{\E \left[ \e^{-w \mutilde_{\bm y^*}^! (\X)}\right]}}  \frac{m_{\Phi^{k+1}} (\bm y^*, y)}{m_{\Phi^{k}} (\bm y^*) } P_0(\dd y)   \right]  f_{U_n | \bm{Y}} (w) \dd w 
    \end{split}
\end{equation*}
and the thesis now follows by disintegration.

\subsection{Proof of Proposition \ref{prop:joint_nclus}}

    First, observe that $\mutilde_{\bm y^*}^!(A) = \sum_{j \geq 1} S_j \delta_{X_j}(A) = \int_{A \times \R_+} s \Psi^!_{\bm y^*}(\dd x \, \dd s)$, where $\Psi^!_{\bm y^*}$ is obtained by marking $\Phi^!_{\bm y^*}$ (the reduced Palm version of $\Phi$) with i.i.d. marks from $H$. So that, denoting with $n^!$ the number of points in $\Phi^!_{\bm y^*}$ we have
    \begin{align*}
        \E\left[ \e^{-\int_X u \mutilde_{\bm y^*}^! (\dd x)} \right] & = \E \left[ \E \left[ \exp \left(- \sum_{j \geq 1} u S_j \delta_{X_j}(\mathbb X) \right) \mid \Phi^!_{\bm{y^*}} \right]\right] \\
        &= \E \left[ \prod_{j=1}^{n^!} \int_{\R_+} \e^{- u s} H(\dd s) \right] = \E[\psi(u)^{n^!}] .
    \end{align*}
Let $q_{r} = P(n^! = r)$, $r=0, 1, \ldots$ the probability mass function of the number of points in $\Phi^!_{\bm y^*}$, so that $\E[\psi(u)^{n^!}] = \sum_{r \geq 0} \psi(u)^r q_r$, then we can write the marginal as
    \begin{align*}
        \prob(\bm Y^* \in \dd y^*, \tilde \pi = \pi) & = \int_{\R_+} \frac{u^{n-1}}{\Gamma(n)} \E \left[ \e^{- \int_\X u \mutilde_{\bm y^*}^!(\dd z)} \right] \prod_{j=1}^k \kappa(u, n_j) \dd u \,M_{\Phi^{(k)}}(\dd \bm y^*) \\
        &= \int_{\R_+} \frac{u^{n-1}}{\Gamma(n)} \sum_{r \geq 0} \psi(u)^r q_r \prod_{j=1}^k \kappa(u, n_j) \dd u \,M_{\Phi^{(k)}}(\dd \bm y^*) \\
        &= \sum_{r \geq 0} q_r \int_{\R_+} \frac{u^{n-1}}{\Gamma(n)}  \psi(u)^r  \prod_{j=1}^k \kappa(u, n_j) \dd u \, M_{\Phi^{(k)}}(\dd \bm y^*),
    \end{align*}
    where the third equality follows by an application of  Fubini's theorem.\\ 
    Then, by the definition of $V(n_1, \ldots, n_k; r)$, we have
    \begin{align*}
        \prob(K_n = k, \bm Y^* \in \dd \bm y^*) = \frac{1}{k!} \sum_{r=0}^\infty \left(\sum_{(n_1,  \ldots, n_k) \in \mathbb S_n^k} \binom{n}{n_1 \cdots n_k}  V(n_1, \ldots, n_k; r) \right) q_r \  M_{\Phi^{(k)}}(\dd \bm y^*),
    \end{align*}
    where $\mathbb S_n^k$ denotes the set of $k$-compositions, i.e. $\mathbb S_n^k = \{(n_1, \ldots, n_k) \ : n_i \geq 1, n_1 + \cdots + n_k = n\} \subset \mathbb N^k$.

\subsection{Proof of Corollary \ref{cor:uniques_gamma}}

We have that $\psi(u) = (u + 1)^{- a}$ and $\kappa(u, n) = (a)_n (u+1)^{-(n+a)}$, where $(a)_n := \Gamma(a + n) / \Gamma(a)$ denotes the rising factorial or Pochhammer symbol.
First, observe that
\begin{align*}
     V(n_1, \ldots, n_k; r)  & = \int_{\R_+} \frac{u^{n-1}}{\Gamma(n)}  \psi(u)^r  \prod_{j=1}^k \kappa(u, n_j) \dd u = \frac{1}{\Gamma(n)} \prod_{j=1}^k (a)_{n_j} \int \frac{u^{n -1}}{(u + 1)^{n + a k + a r}} \dd u \\
     &= \frac{1}{\Gamma(n)} \prod_{j=1}^k (a)_{n_j} \frac{\Gamma\left((k + r) a\right) \Gamma(n)}{\Gamma\left((k + r)a + n\right)}.
\end{align*}
Thanks to the previous expression and by Proposition \ref{prop:joint_nclus}, one has:
\[
    \prob(\bm Y^* \in \dd y^*, \tilde \pi = \pi) = \prod_{j=1}^k (a)_{n_j}   \left(\sum_{r \geq 0}  q_r \frac{\Gamma(\left(k + r) a\right)}{\Gamma\left((k + r)a + n\right)} \right) \   M_{\Phi^{(k)}}(\dd \bm y^*).
\]
Thus, the probability of interest can be evaluated by marginalising out the frequencies of the observations:
\begin{align*}
    \prob(K_n = k, \bm Y^* \in \dd \bm y^*) & = \left(\sum_{r \geq 0}  q_r \frac{\Gamma\left((k + r) a\right)}{\Gamma\left((k + r)a + n\right)} \right) M_{\Phi^{(k)}}(\dd \bm y^*) \\
    & \qquad \times \frac{1}{k!} \sum_{(n_1,  \ldots, n_k) \in \mathbb S_n^k}  \binom{n}{n_1 \cdots n_k} \prod_{j=1}^k (a)_{n_j} \\
    &= \frac{1}{\Gamma(n)} (-1)^n \mathcal{C}(n, k; -a) \left(\sum_{r \geq 0}  q_r \frac{\Gamma\left((k + r) a\right)}{\Gamma\left((k + r)a + n\right)} \right) M_{\Phi^{(k)}}(\dd \bm y^*)
\end{align*}
where the last equality follows from \cite[Theorem 8.16]{charalambides_enumerativeC}.

\medskip

\section{The case of Poisson point processes}
\label{app:poisson}

A Poisson point process $\Phi$ with mean measure $\omega$ (or intensity $\omega$) is a random counting measure such that for  $A_1, \ldots, A_n \in \calX$ disjoint sets and for any $n \geq 1$, the random elements $\Phi(A_1), \ldots, \Phi(A_n)$ are independent Poisson random variables with parameters  $\omega(A_1), \ldots, \omega(A_n)$.
Poisson point processes provide the main building block for completely random measures. 

Focusing on the prior moment results in Proposition \ref{prop:prior_moms} we can see that if $\Phi$ is a Poisson point process with intensity $\omega(\dd x)$, $M_\Phi(A) = \omega(A)$ and $M_{\Phi^2}(A \times B) = \omega(A) \omega(B) + \omega(A \cap B)$. Hence
\[
    \E[\mutilde(A)] = \E[S] \omega(A), \qquad \Cov(\mutilde(A),\mutilde(B)) = \E[S^2] \omega(A \cap B) .
\]
In particular, if $A \cap B = \emptyset$, $\Cov(\mutilde(A),\mutilde(B))$ is zero, which was expected since the random variables are independent.

Next, we specialise Theorem~\ref{teo:post} to the case of Poisson processes.
\begin{theorem}\label{teo:post_Poi}
    Let $\Phi$ be a Poisson point process with intensity $\omega(\dd x)$. Then, the distribution of the random measure $\mutilde^\prime(\cdot)$ in \Cref{teo:post} is equal to the distribution of
    \[
    \int_{ \cdot \times \R_+ } s \Psi'(\dd x \, \dd s ) = \sum_{j \geq 1} S^\prime_j \delta_{X^\prime_j}(\cdot), 
    \]
    where $\Psi':=\sum_{j \geq 1} \delta_{(X_j' , S_j')}$ is a marked point process whose unmarked point process $\Phi' := \sum_{j \geq 1} \delta_{X_j '}$ is a Poisson process with intensity given by $\psi(u) \omega(\dd x)$ and the marks $S_j'$ are i.i.d. with distribution \eqref{eq:h_tilt}.
\end{theorem}

\begin{proof}
    The random measure $\mutilde^\prime$ in Theorem~\ref{teo:post} has Laplace functional \eqref{eq:Laplace_mu_post}.
    By Lemma \ref{lemma:laplace_marked}, we have that the numerator in \eqref{eq:Laplace_mu_post} can be expressed as
    \[
        \E \left[ \exp \left\{ - \int_\X (f(z) + u) \mutilde_{\bm y^*}^!(\dd z)\right\}\right] 
        = \E\left[ \exp\left(\int_\X \log \int_{\R_+} \e^{- s(f(z) + u)} H(\dd s) \Phi^!_{\bm y^*}(\dd x) \right) \right],
    \]
    where $\Phi^!_{\bm y^*}$ is the reduced Palm version of $\Phi$ at $\bm y^*$. Thanks to the properties of the Poisson process, we have that  $\Phi^!_{\bm y^*}$ equals $\Phi$ in distribution. Hence, the previous expected value can be evaluated by resorting to the L\'evy-Khintchine representation (see, e.g., \cite{kingman1992poisson}):
    \[
      \E \left[ \exp \left\{ - \int_\X (f(z) + u) \mutilde_{\bm y^*}^!(\dd z)\right\}\right] 
        = \exp \left( - \int_\X \int_{\R_+} 1 - \e^{- s(f(z) + u)} H(\dd s) \omega(\dd x)\right).
    \]
  The previous expression is useful to evaluate the denominator in \eqref{eq:Laplace_mu_post}, which can be obtained with the choice $f=0$. We now combine the numerator and denominator to evaluate the Laplace functional of $\tilde{\mu}^\prime$:
    \begin{align*}
        \E\left[\exp \int_\X - f(z) \mutilde^\prime(\dd z) \right] & = \exp \left( - \int_{\X}\int_{\R_+} -\e^{- s(f(z) + u)} + \e^{- su} H(\dd s) \omega(\dd x) \right) \\
        &  = \exp \left( - \int_{\X}\int_{\R_+} \left( 1 -\e^{- sf(z)} \right) \e^{- su} H(\dd s) \omega(\dd x) \right).
    \end{align*}
    By multiplying and dividing by $\psi(u) := \int_{\R_+} \e^{-su} H(\dd s)$, we recognize the Laplace transform of the random measure 
    \[ 
        \mutilde^\prime = \sum_{j \geq 1} S^\prime_j \delta_{X^\prime_j}
    \]
    where the $S'_j$'s are i.i.d. random variables with density $\psi(u)^{-1} \e^{-su} H(\dd s)$ and $\sum_{j \geq  1} \delta_{X^\prime_j}$ is a Poisson random measure with intensity $\psi(u) \omega(\dd x)$.
\end{proof}

Let us now concentrate on the marginal and predictive characterizations in Theorems~\ref{teo:marg} and \ref{teo:pred_conditional}. Since $M_{\Phi^k}(\dd \bm y^*) = \prod_{i=1}^k \omega(\dd y^*_i)$, it follows that \[
    \e^{\omega(\X)(\psi(u) - 1)} \prod_{j=1}^k \kappa(u, n_j) \prod_{j=1}^k \omega(\dd y^*_j) .
\]
Moreover, it is clear that
\[
    \frac{m_{\Phi^{k+1}}(\bm y^*, y)}{m_{\Phi^{k}}(\bm y^*)} P_0(\dd y) = \omega(\dd y).
\]
By the properties of the Poisson point process, we have that $\mu^!_{(\bm y^*, y)}$ and $\mu^!_{\bm y^*}$ have the same distribution so that the ratio of expectations in Theorem~\ref{teo:pred_conditional} equals one.

\medskip

\section{The case of Gibbs point processes}
\label{app:gibbs}

In this appendix, we specialize our general results in the manuscript to the case of $\Phi$ being a Gibbs point process.
Following \cite{BaBlaKa}, we say that a Gibbs point process is a point process $\Phi$ that is absolutely continuous with respect to a given Poisson process $\tilde N$. More specifically, $\Phi$ on  $(\X, \mathcal{X})$
is \textit{absolutely continuous} with respect to $\tilde N$ if for any $L \in \mathcal{M} (\X)$ such that $\prob (\tilde N\in L)=0$, one has $\prob (\Phi\in L)=0 $.
The Radon-Nikodym theorem yields the existence of the density $f_\Phi: \X \rightarrow \R_+$ (i.e., the Radon-Nikodym derivative) such that 
\begin{equation*}
\label{eq:pp_dens}
    \plaw_\Phi(\dd \nu) = f_\Phi(\nu) \plaw_{\tilde N}(\dd \nu).
\end{equation*}

In practice, we assign a Gibbs point process through its density $f_\Phi$. However, this task might be complex because of the constraint $\E[f_\Phi(\tilde N)] = 1$.
In particular, to the best of our knowledge, the expected value cannot be computed in closed form, even for very simple densities.
Therefore, unnormalized densities $g_\Phi$ such that $\E[g_\Phi (\tilde N)] < +\infty$ are typically considered and then normalized as  
$f_\Phi(\nu) = g_\Phi / Z$, where $Z := \E[g_\Phi(\tilde N)]$ are intractable normalizing constants.
Among the other Gibbs points processes, we mention the Strauss process \citep[see][Section 2.1]{beraha21}.
 See also \cite{MoWaBook03} for different examples of Gibbs point processes and associated unnormalised densities.

We assume that $\Phi$ is hereditary, meaning that for any $\nu = \sum_{j = 1}^m \delta_{X_j}\in \M(\X)$ and $ X^* \in \X \setminus \{X_1, \ldots X_m \},$ $g_\Phi(\nu + \delta_{X^*}) > 0$ implies $g_\Phi(\nu) > 0$.
Then, we can introduce the Papangelou conditional intensity, defined as
\begin{equation*}
\label{eq:papangelou}
    \lambda_\Phi(\nu; \bm X^*) = \frac{f_\Phi\left( \nu + \sum_{j=1}^k \delta_{X^*_j} \right)}{f_\Phi\left(\sum_{j=1}^k \delta_{X^*_j} \right)} = \frac{g_\Phi\left( \nu + \sum_{j=1}^k \delta_{X^*_j} \right)}{g_\Phi\left(\sum_{j=1}^k \delta_{X^*_j} \right)}
\end{equation*}
where $\bm X^* = \{X^*_1, \ldots, X^*_k\}$.
The Papangelou conditional intensity can be understood as the conditional ``density'' of $\Phi$ having atoms as in $\nu$, given that the rest of $\Phi$ is $\bm X^*$.
We say that a Gibbs point process $\Phi$ has a repulsive behaviour if
$\lambda_\Phi(\nu; \bm X^*)$ is a nonincreasing function of $\bm X^*$, that is, $\lambda_\Phi(\nu; \bm X^*) \geq \lambda_\Phi(\nu; \bm X^* \cup X^\prime)$ for any $X^\prime \in \X \setminus (\bm X^* \cup X^\prime)$.

Before specialising our treatment to the case of Gibbs processes, we derive here the expressions of the factorial moment measure and reduced Palm kernel.

\begin{proposition}\label{prop:mean_gibbs}
  The $k$-th factorial moment measure of a Gibbs point process with density $f_\Phi$ with respect to a Poisson point process $\tilde N$ with intensity $\omega$ is
  \[
    M_{\Phi^{(k)}}(\dd x_1 \cdots \dd x_k) =\E_{\tilde N}\left[f_\Phi\left(\tilde N + \sum_{j=1}^k \delta_{x_j} \right)  \right] \indicator[x_1 \neq x_2 \cdots \neq x_k] \omega(\dd x_1) \cdots \omega(\dd x_k).
  \]
\end{proposition}
\begin{proof}
Let $\dd \bm x = \dd x_1 \times \cdots \times \dd x_k$, where the $x_j$'s are pairwise distinct. Then
\begin{align*}
    M_{\Phi^{(k)}}(\dd \bm x) &= \int_{\M(\X)} \prod_{j=1}^k \nu(\dd x_j) \plaw_\Phi(\dd \nu) = \int_{\M(\X)} \prod_{j=1}^k \nu(\dd x_j) f_\Phi(\nu) \plaw_N(\dd \nu) \\
    &= \E_{\tilde N}\left[\int_{\X^k} \prod_{j=1}^k \delta_{x_j}(y_j) f_\Phi(\tilde N) \tilde N^k(\dd y_1, \ldots, \dd y_k) \right] \\
    &= \int_{\X^k} \E_{\tilde N}\left[f_\Phi\left(\tilde N + \sum_{j=1}^k \delta_{x_j} \right)  \right] \omega(\dd x_1) \cdots \omega(\dd x_k),
\end{align*}
where the last equation follows by applying \Cref{teo:clm_mult} to the Poisson process $\tilde N$, for which $\tilde N^!_{\bm{x}} \sim \tilde N$ and the fact that the $k$th factorial moment measure of  $\tilde N$ equals $  \omega^k(\dd \bm y)$.
\end{proof}

\begin{proposition}  \label{prop:Gibbs_density_reduced_Palm}
  The $k$-th reduced Palm distribution of a Gibbs point process with density $f_\Phi$ with respect to a Poisson point process $\tilde N$ is the distribution of another Gibbs point process $\Phi^!_{\bm x}$ with density
  \[
    f_{\Phi^!_{\bm x}}(\nu) = \frac{f_\Phi(\nu + \sum_{j=1}^k \delta_{x_j})}{\E_{\tilde N}\left[ f_\Phi(\tilde N + \sum_{j=1}^k \delta_{x_j}) \right]}
  \]
  with respect to the law of $\tilde N$.
\end{proposition}
See \cite{BaBlaKa}, Section~3.2.3, for a proof.

\begin{theorem}
\label{teo:post_Gibbs}

If $\Phi$ is a Gibbs point process with density $f_{\Phi}$ with respect to the Poisson process $\tilde N$, the random measure $\mutilde^\prime$ in \Cref{teo:post} is equal to the distribution of the random measure
    \[
    \tilde \eta(A) := \int_{ A \times \R_+ } s \Psi'(\dd x \, \dd s ) = \sum_{j \geq 1} S^\prime_j \delta_{X^\prime_j}(A), \qquad A \in \mathcal X
    \]
where $\Psi':=\sum_{j \geq 1} \delta_{(X_j' , S_j')}$ is a marked point process whose unmarked point process $\Phi' := \sum_{j \geq 1} \delta_{X_j '}$ is of Gibbs type with density, with respect to the Poisson process $\tilde N$, given by
\[
    f_{\Phi'} (\nu ) :=  \frac{ \psi(u)^{\nu(\X)} f_{\Phi} (\nu +  \sum_{j=1}^k \delta_{y_j^*} ) }{\E_{\tilde N} \left[ \psi(u)^{\tilde N(\X)}  f_{\Phi} (\tilde N +  \sum_{j=1}^k \delta_{y_j^*} ) \right]},
\]
and the marks $S_j'$ are i.i.d. with distribution \eqref{eq:h_tilt}.
\end{theorem}

\begin{proof}
The random measure $\tilde{\mu}^\prime $ in Theorem~\ref{teo:post} has Laplace functional \eqref{eq:Laplace_mu_post}. In order to characterise its distribution, we first evaluate the numerator in \eqref{eq:Laplace_mu_post}.
From Lemma \ref{lemma:laplace_marked}, we have
\[
\E   \left[ \exp \left\{ - \int_\X (f(z) + u) \mutilde^!_{\bm y^*}(\dd z)\right\}\right] = \E \left[  \prod_{j \geq 1 }
     \int_{\R_+}  \e^{- s (f (X_j)+u ) } H (\dd s)\right]
\]
where the $X_j$'s are the support points of $\mutilde^!_{\bm y^*}$.
We now exploit Proposition \ref{prop:Gibbs_density_reduced_Palm} to evaluate the last expected value. Indeed, by virtue of this proposition, $\Phi_{\textbf{y}^*}^{!}$ is again a Gibbs point process, with density with respect to the Poisson process $\tilde N$, given by
\[
f_{\Phi_{\textbf{y}^*}^{!}}(\nu) = \frac{f_\Phi(\nu + \sum_{j=1}^k \delta_{y_j^*})}{\E_{\tilde N}\left[ f_\Phi(\tilde N + \sum_{j=1}^k \delta_{y_j^*}) \right]}, \quad \nu \in \M(\X).
\]
As a consequence, we have:
\begin{align*}
     & \E   \left[ \exp \left\{ - \int_\X (f(z) + u) \mutilde^!_{\bm y^*}(\dd z)\right\}\right] =   \E \left[ \exp \left\{  
    \sum_{j \geq 1} \log \Big( \int_{\R_+}  \e^{- s (f (X_j)+u ) } H (\dd s) \Big) \right\}\right]\\
    & \qquad = \E \left[ \exp \left\{  
    \int_\X \log \Big( \int_{\R_+}  \e^{- s (f (x)+u ) } H (\dd s) \Big) \Phi_{\textbf{y}^*}^{!} (\dd x) \right\}\right]\\
    &  \qquad = \frac{ \E_{\tilde N} \left[ \exp \left\{  
    \int_\X \log \Big( \int_{\R_+}  \e^{- s (f (x)+u ) } H (\dd s) \Big) \tilde N (\dd x) \right\}
     f_\Phi(\tilde N + \sum_{j=1}^k \delta_{y_j^*})\right]
     }{
     \E_{\tilde N}\left[ f_\Phi(\tilde N + \sum_{j=1}^k \delta_{y_j^*})\right]}.
\end{align*}
The previous expression corresponds to the numerator in \eqref{eq:Laplace_mu_post}, while the denominator follows by considering the function $f=0$ in the expression above. Hence 
\begin{equation}
    \label{eq:Laplce_gibbs_mu}
    \begin{split}
     &\E\left[\exp \int_\X - f(z) \tilde{\mu}^\prime(\dd z) \right] \\
     & \qquad =
     \frac{\E_{\tilde N} \left[ \exp \left\{  
    \int_\X \log \Big( \int_{\R_+}  \e^{- s (f (x)+u ) } H (\dd s) \Big) \tilde N (\dd x) \right\}
     f_\Phi(\tilde N + \sum_{j=1}^k \delta_{y_j^*})\right]}{\E_{\tilde N} \left[ \exp \left\{  
    \int_\X \log \Big( \int_{\R_+} \e^{- s u} H (\dd s) \Big) \tilde N (\dd x) \right\}
     f_\Phi(\tilde N + \sum_{j=1}^k \delta_{y_j^*})\right]}.
    \end{split}
\end{equation}
The last expression in \eqref{eq:Laplce_gibbs_mu} may be rewritten as follows
\[
\E_{\tilde N} \left[  \exp \left\{ \int_\X  \log \Big( \int_{\R_+}  \e^{-s f(x)}  H' (\dd s) \Big)  \tilde N (\dd x) \right\} f_{\Phi'} (\tilde N) \right] 
\]
where $f_{\Phi'}$ is a density with respect to a Poisson process $\tilde N$ defined as 
\[
f_{\Phi'} (\nu ) :=  \frac{\exp\left\{ \int_\X  \log \Big( \int_{\R_+}  \e^{-su } H (\dd s ) \Big)  \nu (\dd x ) \right\} f_{\Phi} (\nu +  \sum_{j=1}^k \delta_{y_j^*} ) }{\E_{\tilde N} \left[ \exp\left\{ \int_\X  \log \Big( \int_{\R_+}  \e^{-su } H (\dd s ) \Big) \tilde N (\dd x) \right\}  f_{\Phi} (\tilde N +  \sum_{j=1}^k \delta_{y_j^*} ) \right]}
\]
and $H'$ is a new measure on the positive real line $\R_+$ defined as 
\[
H'(\dd s) :=  \frac{\e^{-su } H (\dd s)}{\int_{\R^+}  \e^{-su}  H (\dd s)}
\]
which is an exponential tilting of $H$. As a consequence, we can conclude that ${\mutilde}^\prime$ is a random measure that can be represented as follows
\[
{\mutilde}^\prime(A) \stackrel{d}{=} \int_{A \times \R_+}  s \Psi' (\dd x \, \dd s), \; \quad
\Psi' := \sum_{j \geq 1} \delta_{({X}_j', {S}_j')}\qquad A\in \calX
\]
and $\Psi'$ is obtained by marking $\Phi^\prime$ with i.i.d. marks having distribution $H'$.
\end{proof}

\begin{theorem}\label{teo:marg_Gibbs}

If $\Phi$ is a Gibbs point process with density $f_{\Phi}$ with respect to the Poisson process $\tilde N$,
the marginal distribution in Theorem~\ref{teo:marg} conditionally to $U_n = u$ is proportional to
\[
       \prod_{j=1}^k \kappa(u, n_j) \omega(\dd y^*_j) \E\left[\psi(u)^{\tilde N(\X )} f_{\Phi}(\tilde N + \sum_{j=1}^k \delta_{y^*_j}) \right].
\]
    
\end{theorem}

\begin{proof}
From \Cref{teo:marg} and Propositions \ref{prop:mean_gibbs} and \ref{prop:Gibbs_density_reduced_Palm} we have that
\begin{align*}
    \prob(\bm Y \in \dd \bm y \mid U_n = u) &\propto \prod_{j=1}^k \kappa(u, n_j) \E\left[\e^{-\int_\X u \mutilde_{\bm y^*}^!(\dd x)}\right] M_{\Phi^{(k)}}(\dd \bm y^*) \\
    &= \prod_{j=1}^k \kappa(u, n_j) \omega(\dd y^*_j) \E_{\tilde N}\left[\exp\left(\int_\X \log \psi(u) \tilde N(\dd x)\right) f_{\Phi^!_{\bm{y^*}}}(\tilde N) \right] \E\left[f_{\Phi}(\tilde N + \sum_{j=1}^k \delta_{y^*_j}) \right] \\
    &= \prod_{j=1}^k \kappa(u, n_j) \omega(\dd y^*_j) \E_{\tilde N}\left[\exp\left(\int_\X \log \psi(u) \tilde N(\dd x)\right) f_{\Phi}(\tilde N + \sum_{j=1}^k \delta_{y^*_j}) \right]
\end{align*}
and the proof follows.
\end{proof}

Focusing now on the predictive distribution in Theorem~\ref{teo:pred_conditional}, it is clear that $M_{\Phi^{(k)}} \ll P_0^k$, where $P_0$ is obtained by normalising the intensity measure $\omega$.
In most applications, $\omega$ is the Lebesgue measure on a compact set.
The predictive distribution in Theorem~\ref{teo:pred_conditional} boils down to
\begin{multline*}
\label{eq:pred_gibbs}
   \prob (Y_{n+1} \in A \mid  \bm{Y} = \bm y, U_n=u) \propto \sum_{j=1}^k \frac{\kappa(u, n_j + 1)}{\kappa(u, n_j)} \delta_{y_j^*} (A)\\
     + \int_A \kappa(u, 1) \frac{\E\left[\psi^{\tilde N(\X)}(u) g_{\Phi}(\tilde N + \sum_{j=1}^k \delta_{y^*_j} + \delta_y)\right]}{\E\left[\psi^{\tilde N(\X)}(u) g_{\Phi}(\tilde N + \sum_{j=1}^k \delta_{y^*_j})\right]} \omega(\dd y).
\end{multline*}
The last term involves two expected values with respect to the Poisson process $\tilde N$, which can be easily approximated via Monte Carlo integration.

Finally, we remark that when Gibbs point processes are used as mixing measures in Bayesian mixture models, the posterior inference is unaffected by the (intractable) normalising constant if the hyper-parameters appearing in the density are fixed. In the opposite case, updating those hyper-parameters is a ``doubly intractable'' problem. See, e.g., \cite{beraha21} for a solution.

\section{Details about the DPP Examples}
\label{app:ddp}

Here we provide proof of the more general definition of a DPP we have introduced in Section~\ref{sec:back_on_PP} (see Example~\ref{ex:dpp_def}). The expression generalizes the one given by \cite{Lav15}, which deals with the case of $\X$ compact set in $\R^q$ and $\omega$ the Lebesgue measure. Then we explicitly detail the expression of $\mu'$ in Theorem~\ref{teo:post}. See also Example \ref{ex:dpp1}.

\begin{theorem}[DPP densities]\label{teo:dpp_dens}
Let $\omega$ be a finite measure on $\X$.  Consider a complex-valued covariance function $K: \X \times \X \to \mathbb C$, such that $\int_\X K(x, x) \omega(\dd x) < +\infty$, with spectral representation
\[
    K(x, y) = \sum_{h \geq 1} \lambda_h \varphi_h(x) \overline{\varphi_h(y)}, \qquad x, y \in \X
\]
where $(\varphi_h)_{h \geq 1}$ form an orthonormal basis for the space $L^2(\X; \omega)$ of complex-valued functions, $0 \leq \lambda_h < 1$ with $\sum_{h \geq 1} \lambda_h < +\infty$.
Then, the DPP $\Phi$ with kernel $K$ is absolutely continuous with respect to the Poisson process on $\X$ with intensity measure $\omega$. Its density is given by 
\begin{equation}\label{eq:dpp_dens}
    f_{\Phi}(\nu)  = \e^{\omega(\X) - D} \det \left\{C(x, y) \right\}_{x, y \in \nu},
\end{equation}
where $D:=-\sum_{h \geq 1} \log(1 - \lambda_h)$ and $C(x, y) := \sum_{h \geq 1} \frac{\lambda_h}{1 - \lambda_h} \varphi_h(x) \overline{\varphi_h(y)}$ for $x, y \in \X$.
\end{theorem}
\begin{proof}
    A direct, but rather cumbersome, proof can be achieved following the same lines of \cite{shirai2003random}.
    Here, we take a shorter but indirect approach.

    First, let us introduce the Janossy measure $J_k(\dd x_1 \cdots \dd x_k)$. For a finite point process $\Phi$, such a measure is the probability that $\Phi$ consists of exactly $k$ points located in infinitesimal neighborhoods of $\dd x_j$, $j=1, \ldots, k$. See Chapter 7 of \cite{DaVeJo1}.
    For a Poisson point process with intensity measure $\omega$,
    \begin{equation}\label{eq:jan_poi}
        J^{\omega}_k(\dd x_1 \cdots \dd x_k) = \e^{-\omega(\X)} \prod_{j=1}^k \omega(\dd x_j).
    \end{equation}

    Given two point processes, say $\Phi$ and $\Phi^\prime$, with associated Janossy measures $J_k$ and $J_k^\prime$, then $\Phi \ll \Phi^\prime$ if and only if $J_k \ll J_k^\prime$.
    The Janossy measure of a general DPP $\Phi$ is given by \citep[see, e.g.,][]{georgii2005conditional}
    \begin{equation}\label{eq:jan_dpp}
        J_k(\dd x_1 \cdots \dd x_k) = \det(\mathcal I - \mathcal K) \det \left\{L(x_i, x_j) \right\}_{i, j=1}^k \prod_{j=1}^k \omega(\dd x_j),
    \end{equation}
    where $\mathcal I$ is the identity operator, $\mathcal K$ is the integral operator $\mathcal K: L^2(\X; \omega) \rightarrow L^2(\X; \omega)$ defined by $x \mapsto \mathcal K f(x) = \int_\X K(x, y) f(y) \omega(\dd y)$, and $L$ is the kernel associated with the operator $\mathcal L := (\mathcal I - \mathcal K)^{-1} \mathcal K$.  The determinant  $\det(\mathcal I - \mathcal K)$ is to be intended as a Fredholm determinant \citep[see, e.g., Section 3.4 in][for a basic treatment on Fredholm determinants]{anderson2010introduction}.
     Comparing \eqref{eq:jan_dpp} and \eqref{eq:jan_poi}, it is clear that the DPP with kernel $K$ is absolutely continuous w.r.t. $\mathrm{PP}(\omega)$.
    The proof is now concluded by noting that the collection of the ratios $J_k / J_k^\omega$ defines the density of $\Phi$ with respect to the Poisson point process with intensity $\omega$ \citep[][Chapter 7]{DaVeJo1}, which is exactly
    \[
        f_{\Phi}(\nu) = \e^{\omega(\X)}\det(\mathcal I - \mathcal K) \det \left\{L(x, y) \right\}_{x, y \in \nu},
    \]
    which matches \eqref{eq:dpp_dens} as shown below.

    From the Mercer decomposition of $K$, we have that the integral operator $\mathcal K$ has eigenvalues $(\lambda_h)_{h \geq 1}$ and eigenfunctions $(\varphi_h)_{h \geq 1}$. Simple algebra shows that $\mathcal I - \mathcal K$ has the same eigenfunctions and the eigenvalues are $(1 - \lambda_h)_{h \geq 1}$. The properties of the Fredholm determinant thus entail
    \[
        \det(\mathcal I - \mathcal K) = \prod_{h \geq 1} 1 - \lambda_h.
    \]
    Finally, basic algebra shows that if $\mathcal K \varphi = \lambda \phi$, then 
    \[
        \mathcal L\varphi := (\mathcal I - \mathcal K)^{-1} \mathcal K \varphi = \frac{\lambda}{1 -\lambda} \varphi
    \]
    which shows $ \det \left\{L(x, y) \right\}_{x, y \in \nu} = \left\{C(x, y) \right\}_{x, y \in \nu}$ for $C$ defined as in the statement of the theorem (observe that the series $\lambda_h / (1 - \lambda_h)$ clearly converges since the series of the $\lambda_h$'s does). 
\end{proof}

\begin{theorem}\label{teo:post_dpp}
    Assume that $\Phi$ is a DPP with kernel $K$. Moreover, assume that all of its eigenvalues $\lambda_j$ in \eqref{eq:k_mercer} are strictly smaller than one. Then, the random measure $\mutilde^\prime$ in \Cref{teo:post} is equal to the distribution of the random measure defined by
    \[
   \tilde \eta(A) :=  \int_{ A \times \R_+ } s \Psi'(\dd x \, \dd s ) = \sum_{j \geq 1} S^\prime_j \delta_{X^\prime_j}(A), \qquad A \in \calX
    \]
where $\Psi':=\sum_{j \geq 1} \delta_{(X_j' , S_j')}$ is a marked point process whose unmarked point process $\Phi' := \sum_{j \geq 1} \delta_{X_j '}$ is a DPP with density with respect to $\mathrm{PP}(\omega)$ given by $f_\phi(\nu) \propto \det[ C^\prime(x_i, x_j)]_{(x_i, x_j) \in \nu}$, where
\[
    C^\prime(x, y) = \psi(u) \left[C(x, y) - \sum_{i, j = 1}^k \left(C_{\bm y^*}^{-1} \right)_{i, j} C(x, y^*_i) C(y, y^*_j)\right],
\] 
and the marks $S_j'$ are i.i.d. with distribution \eqref{eq:h_tilt}.
\end{theorem}
\begin{proof}
Under our assumptions, the DPP $\Phi$ has a density with respect to the Poisson process. Then, it is possible to apply Theorem~\ref{teo:post_Gibbs}, so that the point process $\Phi^\prime$ has unnormalized density
\[
    q_{\Phi^\prime}(\nu) = \psi(u)^{n_\nu} f_{\Phi} \left( \nu +  \sum_{j=1}^k \delta_{y_j^*} \right),
\]
where $n_\nu = \nu(\X)$. If $\nu := \sum_{j=1}^{n_\nu} \delta_{x_j}$, then the density $f_\Phi$ equals the determinant of the matrix
\[
\left[
    \begin{array}{c c c | c c c}
    C(y^*_1, y^*_1) & \cdots & C(y^*_1, y^*_k) & C(y^*_1, x_1) & \cdots & C(y^*_1, x_{n_\nu}) \\
    \vdots & & \vdots & \vdots & & \vdots \\
    C(y^*_k, y^*_1) & \cdots & C(y^*_k, y^*_k) & C(y^*_k, x_1) & \cdots & C(y^*_k, x_{n_\nu}) \\\hline
    C(x_1, y^*_1) & \cdots & C(x_1, y^*_k) & C(x_1, x_1) & \cdots & C(x_1, x_{n_\nu}) \\
    \vdots & & \vdots & \vdots & & \vdots \\
    C(x_{n_\nu}, y^*_1) & \cdots & C(x_{n_\nu}, y^*_k) & C(x_{n_\nu}, x_1) & \cdots & C(x_{n_\nu}, x_{n_\nu})
    \end{array}
\right]
\]
Let $C_{\bm y^*}$ denote the upper left block, $C_{x y}$ the bottom left one and $C_{x x}$ the bottom right one. Then, thanks to Schur's determinant identity and ignoring the terms that do not depend on $\nu$, we have
\[
    q_{\Phi^\prime}(\nu) = \psi(u)^{n_\nu} \det(C_{x x} - C_{x y} C_{\bm y^*}^{-1} C_{xy}^T).
\]
If we define
\[
    C^\prime(x, y) = \psi(u) \left[C(x, y) - \sum_{i, j = 1}^k \left(C_{\bm y^*}^{-1} \right)_{i, j} C(x, y^*_i) C(y, y^*_j)\right],
\] 
then it is easy to see that $q_{\Phi^\prime}(\nu) = \det[C^\prime(x_i, x_j)]_{(x_i, x_j) \in \nu}$. Therefore, we can conclude that $\Phi^\prime$ is a DPP with parameters reported in the statement.
\end{proof}

\section{Details about the shot-noise Cox process example}\label{app:sncp}

In this appendix we give proofs of the main theoretical results for the shot-noise Cox process, and also report auxiliary properties.

\subsection{Auxiliary Results}

\begin{lemma}\label{lemma:moment_cox}
Let $\Phi$ be a shot-noise Cox process with kernel $k_\alpha$ and base intensity $\omega$.
Define $\eta(x_1, \ldots, x_{l}): = \int \prod_{i=1}^l k_{\alpha}(x_i - v) \omega(\dd v)$.
Then
\[
    M_{\Phi^{(k)}}(\dd x_1 \cdots \dd x_k) = \gamma^k \sum_{j=1}^k \sum_{G_1, \ldots, G_j \in (*)} \prod_{\ell=1}^j \eta(\bm x_{G_\ell})  \dd x_1 \cdots \dd x_k,
\]
 where $(*)$ denotes all the partitions of $k$ elements in $j$ groups and $\bm x_{G_\ell} = (x_i : i \in G_\ell)$.
\end{lemma}
\begin{remark}
 Alternatively, we can introduce a set of \emph{allocation variables} $\bm t = (t_1, \ldots, t_k)$ such that $t_i = \ell$ if and only if $i \in G_\ell$ and write 
 \[
    M_{\Phi^{(k)}}(\dd x_1 \cdots \dd x_k) = \gamma^k \sum_{\bm t \in (\bullet)} \prod_{\ell=1}^{|\bm t|} \eta(\bm x_{\ell})  \dd x_1 \cdots \dd x_k,
 \]
 where $(\bullet)$ stands for the set of all possible indicator variables describing the partition of $k$ objects, $|\bm t|$ is the number of unique values in $\bm t$, and $\bm x_\ell = (x_i : t_i = \ell)$.
\end{remark}
\begin{proof}
        By Campbell's theorem:
\begin{align*}
   M_{\Phi^{(k)}}(\dd \bm x) &= \E\left[ \prod_{i=1}^k \E\left[ \Phi(\dd x_i)\mid \Lambda \right] \right] = \gamma^k \E\left[\int_{\X^k} \prod_{i=1}^k k_\alpha(x_i - v_i) \Lambda(\dd v_1) \cdots \Lambda(\dd v_k)\right] \dd x_1 \cdots \dd x_k \\
    &= \gamma^k \int_{\X^k} \prod_{i=1}^k k_\alpha(x_i - v_i) M_{\Lambda^k}(\dd v_1 \cdots \dd v_k) \dd x_1 \cdots \dd x_k
\end{align*}
where $M_{\Lambda^k}$ is the $k$-th moment measure of the Poisson point process $\Lambda$, which can be expressed as
\[
   M_{\Lambda^k}(\dd v_1 \cdots \dd v_k) = \sum_{j=1}^k \sum_{G_1, \ldots G_j \in (*)} \prod_{l = 1}^j \left[ \omega(\dd v_{G_{l1}}) \prod_{m \in G_l} \delta_{v_{G_{l1}}}(v_{G_{lm}}) \right]
\]
where $(*)$ denotes all the partitions of $k$ elements in $j$ groups. 
Then
\[
 M_{\Phi^{(k)}}(\dd \bm x) = \gamma^k \sum_{j=1}^k \sum_{G_1, \ldots G_j \in (*)} \int_{\X^k} \prod_{i=1}^k k_\alpha(x_i - v_i) \prod_{l = 1}^j \left[ \omega(\dd v_{G_{l1}}) \prod_{m \in G_l} \delta_{v_{G_{l1}}}(v_{G_{lm}}) \right] \dd x_1 \cdots \dd x_k
\]
observe that the integral over $\X^k$ has a nice interpretation of the product of marginals of models of the kind
$\bm x_{G_l} := (x_i: i \in G_l) \mid v_{G_{l1}} \iid k_{\alpha}(\cdot \mid v_{G_{l1}})$, $v_{G_{l1}} \sim \omega$. 
Denoting by $\eta(\bm x_{G_l})$ the marginal distribution for the model, the integral can be expressed as
$\prod_{l=1}^j \eta(\bm x_{G_l})$ leading to
\[
    M_{\Phi^{(k)}}(\dd \bm x) = \gamma^k \sum_{j=1}^k \sum_{G_1, \ldots G_j \in (*)} \prod_{l=1}^j \eta(\bm x_{G_l})  \dd x_1 \cdots \dd x_k.
\]
\end{proof}

\begin{lemma}\label{lem:palm_cox}
Let $\Phi$ be a shot-noise Cox process. Then the reduced Palm distribution of $\Phi$ at $\bm x = (x_1, \ldots, x_k)$ 
can be described as a mixture as follows. Introduce random indicator variables $T^!_1, \ldots, T^!_k$ representing a partition of $[k]= \{ 1, \ldots  , k\}$, and define $\bm x_\ell = (x_j \colon T^!_j = \ell)$. Then 
\[
    \Phi^!_{\bm x} \mid \bm T^!  \deq \Phi + \sum_{\ell=1}^{|\bm T^!|} \Phi_{\bm x_\ell},
\]
where
\begin{itemize}
    \item[(i)] $\Phi_{\bm x_\ell} \mid \zeta_{\bm x_\ell} \sim \mathrm{PP}\left(\gamma k_\alpha(z - \zeta_{\bm x_\ell}) \dd z\right)$ and $\zeta_{\bm x_\ell} \sim f_{\ell}(v)\dd v \propto \prod_{j \colon T^!_j = \ell} k_\alpha(x_j - v) \omega(\dd v)$;
    \item[(ii)] the processes $\Phi$ and $\Phi_{\bm x_\ell}$, $\ell = 1, \ldots, |\bm T^!|$ are (jointly) independent conditionally to $\bm T^!$;
    \item[(iii)] $\prob(\bm T^! = \bm t^!) \propto \prod_{\ell=1}^{ | \bm t^!|} \eta((x_j \colon t^!_j = \ell))$, where $\eta$ is defined in Lemma \ref{lemma:moment_cox}.
\end{itemize}
\end{lemma}
\begin{proof} 
We prove the result by induction, relying on Lemmas \ref{lem:superp_palm} and \ref{lem:superp_palm2}.
When $k=1$, \cite{Mo03Cox} proved that
\[
    \Phi^!_{x} \deq \Phi + \Phi_{\zeta_x}
\]
where $\Phi_{\zeta_x} \mid \zeta_x \sim \mathrm{PP}(\gamma k_\alpha(z - \zeta_x) \dd z)$ and $\zeta_x \sim f_{\zeta_x}(v) \dd v \propto k_\alpha(x - v) \omega(\dd v)$, which coincides with the statement of the theorem.

Let now $k=2$ and $\bm x = (x, y)$. By the Palm algebra \citep[Proposition 3.3.9 in][]{BaBlaKa}, we have, $\Phi^!_{(x, y)} = (\Phi^!_x)^!_y \deq (\Phi + \Phi_{\zeta_x})^!_y$. Then, an application of Lemma \ref{lem:superp_palm} yields
\begin{equation}\label{eq:palm_sncp2points}
    \Phi^!_{(x, y)} \deq \begin{cases}
        \Phi + \Phi_{\zeta_y} + \Phi_{\zeta_x} & \quad \text{with probability proportional to } M_{\Phi}(\dd y)\\
        \Phi + \Phi_{\zeta_{x, y}} & \quad \text{with probability proportional to } M_{\Phi_{\zeta_x}}(\dd y)
    \end{cases}
\end{equation}
where we observe that $M_{\Phi}(\dd y) = \gamma \eta(y) \dd y$ and $M_{\Phi_{\zeta_x}}(\dd y) = \gamma \eta(x, y)/\eta (x) \dd y$.
Furthermore, $\Phi_{\zeta_{x, y}} := (\Phi_{\zeta_x})^!_y$ has the following representation
\begin{equation}\label{eq:tmp}
    \begin{aligned}
    \Phi_{\zeta_{x, y}} \mid \zeta_{x, y} & \sim \mathrm{PP}(\gamma k_\alpha(z - \zeta_{x, y}) \dd v) \\
    \zeta_{x, y} & \sim f_{\zeta_{x, y}}(v) \dd v \propto k_\alpha(y - v) k_\alpha(x - v) \omega(\dd v).
\end{aligned}
\end{equation}
We will prove \eqref{eq:tmp} below. Having established $\Phi^!_{(x, y)}$ as in \eqref{eq:palm_sncp2points}, we reason by induction to obtain the statement of the theorem. 

Finally, to show \eqref{eq:tmp}, we exploit Proposition 3.2.1 in \cite{BaBlaKa}, which entails that $\Phi_{\zeta_{x, y}}$ is uniquely characterized by
\[
    \partder{t} L_{\Phi_{\zeta_x}}(f + t g)|_{t= 0} = -\int_{\X} g(y) \e^{- f(y)} L_{\Phi_{\zeta_{x, y}}}(f) M_{\Phi_{\zeta_x}}(\dd y),
\]
for all measurable $f, g: \X \rightarrow \R_+$ such that $g$ is $M_{\Phi_{\zeta_x}}$ integrable.

Then, writing $\Phi(f)$ in place of $\int_\X f(x) \Phi(\dd x)$, and applying Lebesgue's theorem, we get
\begin{align*}
    \partder{t} & L_{\Phi_{\zeta_x}}(f + t g) \Big|_{t= 0} =  \partder{t} \E\left[ \E\left[ \e^{- \Phi_{\zeta_x}(f + tg)}\mid \zeta_x\right] \right] \Big|_{t= 0}  = \E\left[ \E\left[ - \Phi_{\zeta_x}(g) \e^{- \Phi_{\zeta_x}(f)} \mid \zeta_x\right] \right].
\end{align*}
An application of the L\'evy-Khintchine representation yields
\begin{align*}
    \partder{t} & L_{\Phi_{\zeta_x}}(f + t g) \Big|_{t= 0} \\
    & = - \int_{\X \times \X} g(y) \e^{-f(y)} \exp\left\{- \int_\X (1 - \e^{- f(z)}) \gamma k_\alpha(z - v) \dd z \right\} \gamma k_\alpha(y - v) \frac{k_\alpha(x - v)}{\eta(x)} \omega(\dd v) \dd y
\end{align*}
and the proof of \eqref{eq:tmp} follows by recognizing the law of $\zeta_{x, y}$ in the integrand and the disintegration theorem \citep[Theorem 3.4 in][]{Kallenberg2021}.

Now observe that \eqref{eq:palm_sncp2points} can be equivalently expressed as the statement of the Lemma, augmenting the underlying probability space introducing $\bm T^! = (T^!_1, T^!_2)$ such that
\[
    \prob(\bm T^! = (1, 1)) \propto \eta(x, y), \qquad \prob(\bm T^! = (1, 2)) \propto \eta(x) \eta(y).
\]
So that
\[
    \Phi^!_{(x, y)} \mid \bm T^! \deq \Phi + \sum_{\ell=1}^{|\bm T^!|} \Phi_{\bm x_\ell},
\]
which verifies the lemma for $k=2$. The general case with $k \geq 3$ follows by induction.
\end{proof}

    \begin{lemma}\label{lemma:laplace_cox}
        Let $\Psi^\prime = \sum_{j \geq 1} \delta_{(X^\prime_j, S^\prime_j)}$ where $S^\prime \iid H^\prime$ and $\Phi^\prime = \sum_{j \geq 1} \delta_{X^\prime_j}$ is a shot-noise Cox point process with base intensity $\rho^\prime$. Let
        \[
            \mu^\prime(A) = \int_{A \times \R_+} s \Psi^\prime(\dd x \, \dd s),
        \]
        then for any $f \geq 0$
        \[
            \E \left[ \e^{- \int_\X f(x) \mu^\prime(\dd x)}\right] = \exp \left\{ - \int_\X 1 -  \exp \left( - \int_{\X} \gamma k_{\alpha}(x - y) \int_{\R_+} 1 - \e^{- s f(x)} H^\prime(\dd s) \dd x \right) \rho^\prime(\dd y)  \right\} .
        \]
    \end{lemma}
    \begin{proof}
        By Lemma \ref{lemma:laplace_marked} we have that
        \begin{align*}
            \E \left[ \e^{- \int_\X f(x) \mu^\prime(\dd x)}\right] = \E\left[\exp\left\{ \int_\X \log \left( \int_{\R_+} \e^{- s f(x)} H^\prime(\dd s) \right) \Phi^\prime(\dd x) \right\} \right].
        \end{align*}
        Then, by the tower property of the expected value and exploiting the L\'evy-Kintchine representation and Fubini's theorem, we have
        \begin{align*}
            \E \left[ \e^{- \int_\X f(x) \mu^\prime(\dd x)}\right] & = \E \left[ \E\left[\exp\left\{ \int_\X \log \left( \int_{\R_+} \e^{- s f(x)} H^\prime(\dd s) \right) \Phi^\prime(\dd x) \right\} \mid \Lambda \right] \right] \\
            &= \E \left[ \exp \left\{ - \int_{\X}\left( 1 - \int_{\R_+} \e^{- s f(x)} H^\prime(\dd s) \right)\gamma \int_\X k_{\alpha}(x - y) \Lambda(\dd y) \dd x \right\} \right] \\
            &= \E \left[ \exp \left\{ - \int_{\X} \int_{\X} \gamma k_{\alpha}(x - y) \int_{\R_+} \left(1 - \e^{- s f(x)} H^\prime(\dd s) \right)   \dd x \Lambda(\dd y)  \right\} \right] \\
            &= \exp \left\{ - \int_\X 1 -  \exp \left( - \int_{\X} \gamma k_{\alpha}(x - y) \int_{\R_+} 1 - \e^{- s f(x)} H^\prime(\dd s) \dd x \right) \rho^\prime(\dd y)  \right\}
        \end{align*}
        and the result follows.
    \end{proof}

\subsection{Main results}
\begin{theorem}\label{teo:post_cox}
Assume that $\Phi$ is a shot-noise Cox process with base intensity $\omega(\dd x)$. Then,  the random measure $\mutilde^\prime$ in \eqref{eq:mu_post} satisfies the following distributional equality
\[
    \mutilde^\prime \mid (\bm T^\prime = \bm t^\prime) \stackrel{d}{=}\mutilde_0 + \sum_{\ell = 1}^{|\bm t^\prime|} \mutilde_{{\bm y^*_\ell}},
\]
where:
\begin{itemize}
    \item[(i)] $\mutilde_0 = \sum_{j \geq 1} \tilde S_j \delta_{\tilde X_j}$, such that $\tilde S_j \iid H^\prime$ in \eqref{eq:h_tilt}, and $\Phi_0 = \sum_{j \geq 1} \delta_{\tilde X_j}$ is a shot-noise Cox process such that
    \[
        \Phi_0 \mid \Lambda_0 \sim \mathrm{PP}\left(\gamma \psi(u)  \int_{\X} k_\alpha(x - v) \Lambda_0(\dd v) \dd x \right), \quad  \Lambda_0 \sim \mathrm{PP}\left(  \e^{- \gamma (1-\psi(u))} \omega(\dd y) \right);
    \]
    \item[(ii)] $\bm T^\prime = (T^\prime_1, \ldots, T^\prime_k)$ are latent variables that identify a partition of $[k]$, $k$ the number of distinct values in $\bm y$, with distribution $\prob(\bm T^\prime = \bm t^\prime) \propto \e^{\gamma |\bm t^\prime|(\psi(u) - 1)} \prod_{\ell=1}^{|\bm t^\prime|} \eta((y^*_j \colon t^\prime_j = \ell))$, where $\eta$ is defined in Lemma \ref{lemma:moment_cox};
    
    \item[(iii)] $\bm y^*_\ell = (y^*_j \colon t^\prime_j = \ell)$ and $\mutilde_{{\bm y^*_\ell}} = \sum_{j \geq 1} \tilde S_{\ell, j} \delta_{\tilde X_{\ell, j}}$, such that $\tilde S_{\ell, j} \iid H^\prime$ in \eqref{eq:h_tilt} and 
    \[
        \Phi_{{\bm y^*_\ell}} := \sum_{j \geq 1} \delta_{\tilde X_{\ell, j}} \mid \zeta_{\bm y^*_\ell} \sim \mathrm{PP}(\gamma \psi(u) k_\alpha(z - \zeta_{\bm y^*_\ell}) \dd z), \quad \zeta_{\bm y^*_\ell} \sim f_{\zeta_{\bm y^*_\ell}}(v) \dd v \propto \prod_{j\colon t_j = \ell} k_\alpha(y^*_j - v) \omega(\dd v).
    \]
\end{itemize}
Conditionally to $\bm T^\prime$, all the point processes defined above are independent.
\end{theorem}
\begin{proof}
We specialise Theorem~\ref{teo:post} in the shot-noise Cox process case. Consider the distribution of $\tilde\mu^\prime$ in \eqref{eq:mu_post}, which depends on the reduced Palm measure $\mutilde^!_{\bm y^*}$.
Recall that, from \Cref{prop:marked_palm}, $\mutilde^!_{\bm y^*}$ is obtained by independently marking the process $\Phi^!_{\bm y^*}$ defined in \Cref{lem:palm_cox} with i.i.d. jumps from $H$.
Then, it is clear that we can operate the same disintegration by introducing latent variables $\bm T^! = (T^!_1, \ldots, T^!_k)$ and write
\begin{equation}\label{eq:mutilde_mix}
    \mutilde^!_{\bm y^*} \mid \bm T^! = \mutilde + \sum_{\ell = 1}^{|\bm T^!|} \mutilde_{\bm y^*_{\ell}}
\end{equation}
where the $\mutilde_{\bm y^*_{\ell}}$'s are obtained by marking the processes $\Phi_{\bm y^*_\ell}$ defined in \Cref{lem:palm_cox}.

Consider now the Laplace functional of $\tilde \mu^\prime$ in \eqref{eq:Laplace_mu_post}, simple algebra leads to
\begin{align*}
    \E\left[\e^{-\int_\X f(x) \mutilde^\prime(\dd x)}\right] &= \E\left[\frac{\E\left[\e^{-\int_\X (f(x) + u) \mutilde^!_{\bm y^*}(\dd x)} \mid \bm T^!\right]}{\E\left[\E\left[\e^{-u \mutilde^!_{\bm y^*}(\X)} \mid \bm T^!\right]\right]}\right] \\
    &= \E\left[\frac{\E\left[\e^{-\int_\X (f(x) + u) \mutilde^!_{\bm y^*}(\dd x)} \mid \bm T^!\right]}{\E\left[\e^{-u \mutilde^!_{\bm y^*}(\X)} \mid \bm T^!\right]} \cdot \frac{\E\left[\e^{-u \mutilde^!_{\bm y^*}(\X)} \mid \bm T^!\right]}{\E\left[\E\left[\e^{-u \mutilde^!_{\bm y^*}(\X)} \mid \bm T^!\right]\right]} \right] \\
    &= \sum_{\bm t^\prime \in (\bullet)} \frac{\E\left[\e^{-\int_\X (f(x) + u) \mutilde^!_{\bm y^*}(\dd x)} \mid \bm T^! = \bm t^\prime \right]}{\E\left[\e^{-u \mutilde^!_{\bm y^*}(\X)} \mid \bm T^! = \bm t^\prime\right]} \cdot \frac{\E\left[\e^{-u \mutilde^!_{\bm y^*}(\X)} \mid \bm T^! = \bm t^\prime\right]}{\E\left[\E\left[\e^{-u \mutilde^!_{\bm y^*}(\X)} \mid \bm T^! = \bm t^\prime\right]\right]} \prob(\bm T^! = \bm t^\prime)
\end{align*}
where $(\bullet)$ denotes all the possible partitions of $k$ objects. Now observe that the term
\begin{align*}
    \frac{\E\left[\e^{-u \mutilde^!_{\bm y^*}(\X)} \mid \bm T^!= \bm t^\prime\right]}{\E\left[\E\left[\e^{-u \mutilde^!_{\bm y^*}(\X)} \mid \bm T^! = \bm t^\prime\right]\right]} \prob(\bm T^! =\bm t^\prime)  \propto \e^{\gamma |\bm t^\prime| (\psi(u) - 1)} \prod_{\ell=1}^{|\bm t^\prime|} \eta(\bm y^*_\ell)
\end{align*}
is simply the distribution of the variables $\bm T^\prime$ described in the statement.

Therefore, by Theorem 3.4 in \cite{Kallenberg2021} we can introduce latent variables $\bm T^\prime$ and consider the conditional expectation
\[
    \E\left[\e^{-\int_\X f(x) \mutilde^\prime(\dd x)} \mid \bm T^\prime =\bm t^\prime \right] = \frac{\E\left[\e^{-\int_\X (f(x) + u) \mutilde^!_{\bm y^*}(\dd x)} \mid \bm T^! = \bm t^\prime\right]}{\E\left[\e^{-u \mutilde^!_{\bm y^*}(\X)} \mid \bm T^! = \bm t^\prime\right]}.
\]
Further, by \eqref{eq:mutilde_mix} and the independence of the terms on the right hand side of \eqref{eq:mutilde_mix}, we have
\begin{equation} \label{eq:mu_prime_shot_posterior}
    \E\left[\e^{-\int_\X f(x) \mutilde^\prime(\dd x)} \mid \bm T^\prime\right] =  \frac{\E\left[\e^{-\int_\X (f(x)+u) \mutilde(\dd x)}\right]}{\E\left[\e^{-u \mutilde(\X)}\right]} \prod_{\ell=1}^{|\bm T^\prime|} \frac{\E\left[\e^{-\int_\X (f(x)+u) \mutilde^!_{\bm y^*_\ell}(\dd x)}\right]}{\E\left[\e^{-u \mutilde^!_{\bm y^*_\ell}(\X)}\right]}.
\end{equation}
Focusing on the first term on the right hand side, and in particular the numerator, by \Cref{lemma:laplace_marked} we have 
\begin{align*}
    & \E\left[\e^{-\int_\X (f(x)+u) \mutilde(\dd x)}\right] = \E\left[ \e^{\int_\X \log \left(\int_{\R_+} \e^{- (f(x) + u)s } H(\dd s) \right) \Phi(\dd x)}  \right] \\
   & \qquad \qquad = \E_\Lambda\left[ \E\left[ \e^{\int_\X \log \left(\int_{\R_+} \e^{- (f(x) + u) } H(\dd s) \right) \Phi(\dd x)}  \mid \Lambda \right] \right] \\
    & \qquad \qquad = \E_\Lambda\left[ \exp \left\{ - \int_\X  \left[ 1 - \left(\int_{\R_+} \e^{- (f(x) + u)s } H(\dd s) \right)\right] \gamma \int_{\X} k_{\alpha}(x - y) \Lambda(\dd y) \dd x \right\}  \right] \\
    & \qquad \qquad = \E_\Lambda\left[ \exp \left\{ - \int_\X \int_\X \left[ 1 - \left(\int_{\R_+} \e^{- (f(x) + u)s } H(\dd s) \right) \right] \gamma k_{\alpha}(x - y) \dd x \Lambda(\dd y) \right\} \right] \\
     & \qquad \qquad = \exp \left\{ - \int_\X 1 - \exp \left\{ - \int_\X \left[ 1 - \left(\int_{\R_+} \e^{- (f(x) + u)s } H(\dd s) \right)\right] \gamma k_{\alpha}(x - y) \dd x \right\} \omega(\dd y) \right\}
\end{align*}
where the third equality follows from the L\'evy-Khintchine representation for $\Phi \mid \Lambda$ (see, e.g., \cite{kingman1992poisson}, Chapter 8), the fourth one by Fubini's theorem and the last one from the L\'evy-Khintchine representation for $\Lambda$.
When $f \equiv 0$, the previous expression reduces to
\begin{equation*}
    \E\left[ \e^{\int_\X \log \left(\int_{\R_+} \e^{- us } H(\dd s) \right) \Phi(\dd x)}  \right] = \exp \left\{ - \int_\X 1 - \exp \left\{ - \gamma \int_{\R_+} (1 - \e^{- us}) H(\dd s) \right\} \omega(\dd y) \right\}.
\end{equation*}
Therefore,
\begin{align*}
    & \frac{\E\left[ \e^{\int_\X \log \left(\int_{\R_+} \e^{- (f(x) + u)s } H(\dd s) \right) \Phi(\dd x)}  \right]}{\E\left[ \e^{\int_\X \log \left(\int_{\R_+} \e^{- us } H(\dd s) \right) \Phi(\dd x)}  \right]} \\
    & \qquad \qquad = \exp \Bigg\{ - \int_\X \exp \left[ - \gamma \int_{\R_+} (1 - \e^{- us}) H(\dd s) \right] - \\
     & \qquad \qquad  \qquad \exp \left[- \int_\X \gamma k_{\alpha}(x - y)  \int_{\R_+} ( 1 - \e^{- (f(x) + u)s }) H(\dd s)  \dd x \right] \omega(\dd y) \Bigg\} \\
    & \qquad \qquad = \exp\Bigg\{ - \int_\X \left[ 1 - \exp\left( - \int_\X \gamma k_\alpha(x - y) \int_{\R_+} \left(1-\e^{- s f(x)}\right) \e^{-us} H(\dd s) \dd x\right)  \right] \\
    & \qquad \qquad  \qquad \qquad  \times  \e^{- \gamma \int_{\R_+} (1 - \e^{-us}) H(\dd s)} \omega(\dd y) \Bigg\} .
\end{align*}
By a close inspection of the last expression, we recognize the Laplace functional of a measure $\mutilde(\dd x) = \int_{\R_+} s \tilde \Psi(\dd x \ \dd s)$, where $\tilde \Psi = \sum_{j \geq 1} \delta_{(\tilde X_j, \tilde S_j)}$ is such that
\begin{itemize}
    \item[(i)] $\tilde S_j \iid H^\prime(\dd s)$, as $j\geq 1$, with $H^\prime$ specified as
\[
    H^\prime(\dd s) := \frac{\e^{-us} H(\dd s)}{\psi(u)},
\]
and $\psi(u) = \int_{\R_+} \e^{-us} H(\dd s)$;
\item [(ii)] the $\tilde X_j$'s are the points of  a shot-noise Cox process  $\tilde \Phi = \sum_{j \geq 1} \delta_{\tilde X_j}$ defined as
\[
    \tilde \Phi \mid \tilde \Lambda \sim \mathrm{PP}\left(\gamma \psi(u)  \int_{\X} k_\alpha(x - v) \tilde \Lambda(\dd v) \dd x \right), \qquad \tilde \Lambda \sim \mathrm{PP}\left(  \e^{- \gamma \int_{\R_+} 1 - \e^{-us} H(\dd s)} \omega(\dd y) \right).
\]
\end{itemize}
The distributional equivalence follows by a  simple application of  Lemma \ref{lemma:laplace_cox}.

Regarding the ratio involving the generic term $\mutilde^!_{\bm y^*_\ell}$ in Equation \eqref{eq:mu_prime_shot_posterior}, arguing as above, it is easy to show that
\begin{equation*}
    \frac{\E\left[ \e^{\int_\X \log \left(\int_{\R_+} \e^{- (f(x) + u)s } H(\dd s) \right) \Phi_{\bm y^*_\ell}(\dd x)}  \right]}{\E\left[ \e^{\int_\X \log \left(\int_{\R_+} \e^{- us } H(\dd s) \right) \Phi_{\bm y^*_\ell}(\dd x)}  \right]} = 
    \E_{\zeta_{\bm y^*_\ell}} \left[\exp\left( - \int_{\X\times\R_+} (1 - \e^{-f(x) s}) H^\prime(\dd s)  \psi(u) \omega_{\zeta_{\bm y^*_\ell}}(\dd x) \right) \right] 
\end{equation*}
where we recognise on the right-hand side the Laplace transform of the random measure $\mutilde_{\bm y^*_\ell} = \sum_{j \geq 1} S^\prime_j \delta_{X^\prime_j}$ where $S^\prime_j \iid H^\prime$ (defined above) and the point process $\tilde \Phi_{\bm y^*_\ell} = \sum_{j \geq 1} \delta_{X^\prime_j}$ is a Poisson point process with random intensity 
\[ 
    \omega_{\zeta_j}(\dd x) = \gamma  \psi(u) k_\alpha(\zeta_{\bm y^*_\ell} - x) \dd x
\] 
where
\[
    \zeta_{\bm y^*_\ell} \sim f_\zeta(v)\dd v \propto \prod_{j \colon t_j^\prime = \ell} k_\alpha(y^*_j - v) \omega(\dd v).
\]
\end{proof}

\begin{proposition}\label{cor:marg_cox}
    Consider the model \eqref{eq:model} and assume that $\plaw_{\Phi}$ is the law of a shot-noise Cox process, then conditionally on $U_n$, the marginal distribution of a sample $\bm Y$  and the latent variables $\bm T^\prime = (T^\prime_1, \ldots, T^\prime_k)$, introduced in \Cref{teo:post_cox}, equals
    \[
        \prob(\bm Y \in \dd \bm y , \bm T^\prime= \bm t^\prime \mid U_n=u) \propto \gamma^k \prod_{j=1}^k \kappa(u, n_j)  \prod_{\ell =1}^{| \bm t^\prime|} \eta(y^*_i \colon t_i^\prime = \ell ) \e^{\gamma |\bm t^\prime| (\psi(u) - 1)} \dd \bm y^* .
    \]
\end{proposition}
\begin{proof}
    The marginal distribution of $\bm Y$ given $U_n=u$ can be recovered from \Cref{teo:marg}, and it equals
    \[
        \prob(\bm Y \in \dd \bm y \mid U_n=u) \propto \E\left[\e^{-u \mu^!_{\bm y^*}(\X)}\right] \prod_{j=1}^k \kappa(u, n_j) M_{\Phi^{(k)}}(\dd \bm y^*).
    \]
    By \Cref{lem:palm_cox}, we have
    \begin{align*}
        & \prob(\bm Y \in \dd \bm y \mid U_n=u) \propto \prod_{j=1}^k \kappa(u, n_j) \sum_{\bm t \in (\bullet)} \frac{ \prod_{\ell =1}^{|\bm t|} \eta(y^*_i \colon t_i = \ell )}{\sum_{\bm{\tilde{t}}\in (\bullet)} \prod_{\ell =1}^{|\bm{\tilde{t}}|} \eta(y^*_i \colon \tilde t_i = \ell ) } \e^{\gamma |\bm t| (\psi(u) - 1)} M_{\Phi^{(k)}} (\dd \bm y^*)\\
        & \qquad   \propto \prod_{j=1}^k \kappa(u, n_j) \sum_{\bm t \in (\bullet)} \frac{ \prod_{\ell =1}^{|\bm t|} \eta(y^*_i \colon t_i = \ell )}{\sum_{\bm{\tilde{t}} \in (\bullet)} \prod_{\ell =1}^{|\bm{\tilde{t}}|} \eta(y^*_i \colon \tilde t_i = \ell ) } \e^{\gamma |\bm t| (\psi(u) - 1)} \gamma^k \sum_{\bm{\tilde{t}}\in (*)} \prod_{\ell =1}^{|\bm{\tilde{t}}|} \eta(y^*_i \colon \tilde t_i = \ell ) \dd \bm y^* \\
        & \qquad = \gamma^k \prod_{j=1}^k \kappa(u, n_j) \sum_{\bm t \in (\bullet)} \prod_{\ell =1}^{|\bm t|} \eta(y^*_i \colon t_i = \ell ) \e^{\gamma |\bm t| (\psi(u) - 1)} \dd \bm y^* .
    \end{align*}
    Therefore, we can introduce auxiliary latent variables $\bm T^\prime$ distributed as in \Cref{teo:post_cox} by a suitable augmentation of the underlying probability space, and the result follows. 
\end{proof}

\subsection{The predictive distribution} \label{app:sncp_pred}

We discuss now the predictive distribution, by specializing Theorem~\ref{teo:pred_conditional} to the shot-noise Cox process case. To this end, let $\bm T^k$ the latent indicator variables describing the partition of $[k]$ introduced in \Cref{lem:palm_cox}, where we have now made explicit the dependence on $k$ (the length of the vector $\bm y^*$) and suppressed the $!$ superscript for ease of notation. Define $\bm T^{k+1}$ analogously. Let $\bm{\tilde y} = (\bm y^*, y)$.
Then, the ratio appearing in \eqref{eq:prediction_rule} equals
\begin{align*}
    \frac{\E\left[\e^{-u \mutilde^!_{(\bm y^*, y)}(\X)}\right] m_{\Phi^{k+1}}(\bm y^*, y)}{\E\left[\e^{-u \mutilde^!_{\bm y^*}(\X)}\right] m_{\Phi^{k}}(\bm y^*)} P_0(\dd y) =  \frac{\sum_{\bm t^{k+1}} \e^{\gamma |\bm t^{k+1}|(\psi(u)-1)} \prod_{\ell=1}^{|\bm t^{k+1}|}\eta(\bm{\tilde y}_\ell) }{\sum_{\bm t^{k}} \e^{\gamma |\bm t^{k}|(\psi(u)-1)} \prod_{m=1}^{|\bm t^{k}|}\eta(\bm y^*_m) } \dd y 
\end{align*}
where the summation at the numerator ranges over all possible partitions of $[k+1]$ and the one at the denumerator over all possible partitions of $[k]$.
Furthermore
\begin{align*}
    & \frac{\E\left[\e^{-u \mutilde^!_{(\bm y^*, y)}(\X)}\right] m_{\Phi^{k+1}}(\bm y^*, y)}{\E\left[\e^{-u \mutilde^!_{\bm y^*}(\X)}\right] m_{\Phi^{k}}(\bm y^*)} P_0(\dd y) \\
    & \qquad = \frac{\sum_{\bm t^{k}} \left\{\sum_{t_{k+1}=1}^{|\bm t_k|} \e^{\gamma |\bm t^{k}|(\psi(u)-1)} \prod_{\ell \neq t_{k+1}} \eta(\bm{y}^*_\ell) \eta(\bm y^*_{t_{k+1}}, y) + \e^{\gamma (|\bm t^{k}| + 1) (\psi(u)-1)} \prod_{\ell=1}^{|\bm t^{k}|}\eta(\bm y^*_\ell) \eta(y) \right\}}{\sum_{\bm t^{k}} \e^{\gamma |\bm t^{k}|(\psi(u)-1)} \prod_{m=1}^{|\bm t^{k}|}\eta(\bm y^*_m) } \dd y \\
    & \qquad = \frac{\sum_{\bm t^{k}} \left\{\sum_{t_{k+1}=1}^{|\bm t_k|} \e^{\gamma |\bm t^{k}|(\psi(u)-1)} \prod_{\ell=1}^{|\bm t_k|} \eta(\bm{y}^*_\ell) \frac{\eta(\bm y^*_{t_{k+1}}, y)}{\eta(\bm y^*_{t_{k+1}})} + \e^{\gamma (|\bm t^{k}| + 1) (\psi(u)-1)} \prod_{\ell=1}^{|\bm t^{k}|}\eta(\bm y^*_\ell) \eta(y) \right\}}{\sum_{\bm t^{k}} \e^{\gamma |\bm t^{k}|(\psi(u)-1)} \prod_{m=1}^{|\bm t^{k}|}\eta(\bm y^*_m) } \dd y. 
\end{align*}
It is now straightforward to recognize the expectation with respect to random variables $\bm T^{\prime}$ (describing a partition of $[k]$) distributed as in Theorem~\ref{teo:post_cox}, namely
\[
    \frac{\E\left[\e^{-u \mutilde^!_{(\bm y^*, y)}(\X)}\right] m_{\Phi^{k+1}}(\bm y^*, y)}{\E\left[\e^{-u \mutilde^!_{\bm y^*}(\X)}\right] m_{\Phi^{k}}(\bm y^*)} P_0(\dd y)  = \E_{\bm T^\prime}\left[ \sum_{j=1}^{|\bm T^\prime|} \frac{\eta(\bm y^*_{j}, y)}{\eta(\bm y^*_{j})} + \e^{\gamma (\psi(u) - 1)} \eta(y)\right] \dd y.
\]
Consequently, the predictive distribution \eqref{eq:prediction_rule} of Theorem~\ref{teo:pred_conditional} for the shot-noise Cox process case boils down to
\begin{equation*}
    \begin{split}
        \prob(Y_{n+1} \in A \mid \bm Y = \bm y, \bm T^\prime = \bm t^\prime, U_n=u)  \propto \sum_{j=1}^k &\frac{\kappa(u, n_j + 1)}{\kappa(u, n_j)}   \delta_{y_j^*} (A) \\
    &+\int_A \kappa(u, 1) \left[ \sum_{j=1}^{|\bm t^\prime|} \frac{\eta(\bm y^*_{j}, y)}{\eta(\bm y^*_{j})} + \e^{\gamma (\psi(u) - 1)} \eta(y)  \right]  \dd y
    \end{split}
\end{equation*}
Augmenting the underlying probability space and including latent variables $\bm T^\prime = \bm t^\prime$, we can now deduce a restaurant representation of the predictive distribution, whereby the restaurant tables are divided into ``thematic rooms'' and tables within the same room eat similar dishes. For concreteness, one can think of a huge restaurant whereby different rooms serve dishes from different parts of the world.
Customer 1 enters the first room and eats dish $y^*_1 \sim \eta(y^*_1)$ in the first table; we use $t^\prime_1 = 1$ to indicate the room in which the customer is sitting. As before, all customers sitting at the same table are eating the same dish.
After $n$ customers have entered, suppose they have occupied $k$ distinct tables, serving dishes $y^*_1, \ldots, y^*_k$, in $|\bm t_k|$ different rooms. Then, customer $n+1$ can carry out three distinct choices, as described below.
\begin{itemize}
    \item[(i)] Customer $n+1$ sits at one of the previously occupied tables, say table $j$, with probability proportional to $\kappa(u, n_j + 1) / \kappa(u, n_j)$. Note that this probability depends only on the ``popularity'' of the table, and not the room in which it is placed. 

    \item[(ii)]  Customer $n+1$ sits at a new table in room $\ell \in \{1, \ldots, |\bm t^\prime_k| \}$ with probability proportional to $\int_\X \kappa(u, 1) \eta(\bm y^*_\ell, y) / \eta(\bm y^*_\ell) \dd y$, eating a dish $y$ generated from a probability density proportional to $\kappa(u, 1) \eta(\bm y^*_\ell, y) / \eta(\bm y^*_\ell)$. In this case, we set $t^\prime_{k+1} = \ell$ and $y^*_{k+1} = y$. 

    \item[(iii)]  Customer $n+1$ enters an empty room (sitting, a fortiori, in a new table) with probability proportional to $\int_\X \kappa(u, 1) \e^{\gamma (\psi(u)- 1))} \eta(y) \dd y $. The customer eats a new dish $y$ generated from a  probability density function proportional to $\kappa(u, 1) \e^{\gamma (\psi(u)- 1))} \eta(y) \dd y $. In this case, we set $t^\prime_{k+1} = |\bm t^\prime| + 1$ and $y^*_{k+1} = y$. 
\end{itemize}

\subsection{Details about the MCMC algorithm}\label{app:scnp_algo}

Let $\Lambda = \sum_{h \geq 1} \delta_{\zeta_h}$ be the directing Poisson process, and recall the latent variables $(T_1, \ldots, T_k)$ introduced in Example \ref{ex:sncp_def} and the double-summation formulation for $\Phi$. 
In the following, let $g = |\bm T|$ be the number of unique values in $(T_1, \ldots, T_k)$.
Then, the Gibbs sampling algorithm can be summarized in the following steps.

\begin{enumerate}[nolistsep]
    \item Sample $U_n \sim \Gamma(n, S_\bullet)$, where $S_\bullet = \sum_{j \geq 1} S_j$. 
    
    \item Let $X^{(a)}_1, \ldots, X^{(a)}_k$ denote the active atoms. Consider now the posterior distribution in \Cref{teo:post_cox} and let $X^{(na)}_{\ell, j}$ denote the $j$-th atom of the measure $\mutilde_{\bm y^*_\ell}$ for $\bm y^*_\ell := (X^{(a)}_j \colon T_j = \ell)$. Let $X^{(na)}_{0, j}$ denote the $j$-th support point of the measure $\mutilde_0$ in \Cref{teo:post_cox}. Denote the concatenation of all such atoms as $X_1, \ldots, X_K$ and analogously for $W_1, \ldots, W_k$ and $S_1, \ldots, S_K$.
    Sample the cluster allocations from a categorical distribution with weights:
    \[
        \prob(C_i = h \mid \text{rest}) \propto S_h f(y_i \mid X_h, W_h)
    \]
    Relabel the allocation variables, the atoms, and the unnormalized weights so that the first $k$ of them are the active ones.
    
    \item Sample the active part
    \begin{enumerate}
        \item Sample $S^{(a)}_h \sim f_{S_h}(s) \propto s^{n_h} \e^{-U_n s} H(\dd s)$, $h=1, \ldots, k$, where $n_h = \sum_{i=1}^n \indicator[C_i=h]$.
        \item Sample $X^{(a)}_h \sim f_X(x) \propto \prod_{i: C_i = h} f(Z_i \mid x, W^{(a)}_h) k_\alpha(x - \zeta_{T_h})$, $h=1, \ldots, k$.
        \item Sample $W^{(a)}_h \sim f_W(w) \propto \prod_{i: C_i = h} f(Z_i \mid X^{(a)}_h, w) f_W(w)$, $h=1, \ldots, k$
    \end{enumerate}
    \item Sample the non-active part using the characterization in \Cref{teo:post_cox}:
    \begin{enumerate}
        \item For $\ell=1, \ldots, |\bm T|$:
        \begin{enumerate} 
            \item Sample $\zeta_\ell \sim f_{\zeta_h}(v) \dd v \propto \prod_{j: T_j = \ell} k_\alpha(X^{(a)}_j - v) \omega(\dd v)$
            \item Sample $X^{(na)}_{\ell, 1}, \ldots, X^{(na)}_{\ell, n_\ell}$ from a Poisson process with intensity $\gamma \psi(u) k_\alpha(z - \zeta_\ell)\dd z$ as in Theorem~\ref{teo:post_cox}.
        \end{enumerate}
        \item Sample $X^{(na)}_{0, 1}, \ldots, X^{(na)}_{0, n_0} \sim \mathrm{SNCP}(\gamma \psi(u), k_\alpha, \e^{-\gamma(1 - \psi(u)}\omega)$.
        \item For all the $X^{(na)}$'s simulated above, sample the associated $S^{(na)} \sim f_S(s) \propto  \e^{-us} H(\dd s)$, $h = k_a + 1, \ldots$
        \item For all the $X^{(na)}$'s simulated above, sample the associated $W^{(na)} \sim f_W$
    \end{enumerate}
    
    \item Sample the latent Poisson process
    \begin{enumerate}
        \item For each atom $X_1, \ldots, X_K$, sample the associated latent variable, say $T_h$, from a categorical distribution over all the atoms of $\Lambda$ with weights $ \prob(T_h = \ell \mid \text{rest}) \propto k_{\alpha}(X_h - \zeta_\ell)$. 
        Relabel the $T_h$'s and the atoms of $\Lambda$ so that the unique values in the $T_h$'s are the first ones. 
    \end{enumerate}
\end{enumerate}

\section{Further numerical illustrations}\label{app:numerical_examples}

\subsection{The DPP prior}

In this appendix, we consider the general class of power exponential DPPs \citep{Lav15}, for which the kernel $K$ has spectral representation on $S = [-1/2, 1/2]^d$
\begin{equation}\label{eq:pes_mercer}
    K(x, y) = \sum_{j \in \mathbb Z^d} \lambda_j \cos(2 \pi \langle j, x - y \rangle),
\end{equation}
for eigenvalues
\begin{equation}\label{eq:pes_dpp}
    \lambda_j = \rho \frac{\alpha^d \Gamma(d/2 + 1)}{\pi^{d/2} \Gamma(d/\nu + 1)}  \exp(- \|\alpha j\|^\nu).
\end{equation}
Such a DPP has a density with respect to the unit-rate Poisson process on $S$ if $\alpha^d < \alpha^d_{\max}:= \rho^{-1} \pi^{d/2} \Gamma(d/\nu + 1) / \Gamma(d/2 + 1)$.
Moreover, we recover the prior specification for Gaussian DPP discussed in Section~\ref{sec:numeric} setting $\nu=2$ and $\alpha = 0.5 \alpha_{\max}$.

Parameter $\rho$ is the intensity of the process and is equal to the expected number of points a priori. The repulsiveness of the process is controlled by the parameters $(\rho, \nu, \alpha)$ in a rather complex manner, as discussed in \cite{Lav15}. In particular, having fixed $(\nu, \alpha)$, increasing $\rho$ reduces the repulsiveness of the process (but observe that $\rho$ has a lower bound since $\alpha < \alpha_{\max}$). Similarly, for fixed $(\rho, \alpha)$ increasing $\nu$ increases the repulsiveness of the process and so does increasing $\alpha$ for fixed $(\rho, \nu)$. 

Focusing on the case $d=1$, to ensure the validity of the DPP density, we reparametrize \eqref{eq:pes_dpp} by introducing $s \in (0, 1)$ and letting $\alpha = s \alpha_{\max}$, leading to
\[
    \lambda_j = s \exp(- \|\alpha j\|^\nu).
\]
Then, $s$ takes the interpretation of the repulsiveness of the DPP (in the natural scale between zero and one) \emph{for fixed $\rho$ and $\nu$}. See also Section 2.3 of \cite{ghilotti2023bayesian} for a discussion on the interpretation of repulsiveness hyperparameters in DPPs (under a similar, though not same, parameterization).

\subsubsection{Approximation of the DPP density}\label{app:dpp_density}

The implementation of the conditional algorithm requires the numerical evaluation of the density of the DPP with respect to a suitable Poisson process. 
For our discussion, it is sufficient to consider the case of a DPP defined on $S = [-1/2, 1/2]$. Indeed, to define the DPP on a general interval $[l, u]$, it is sufficient to operate a simple rescaling as shown in \cite{Lav15}. In our examples, we set $l = \bar y - 1.5 \delta$ and $u = \bar y + 1.5 \delta$  where $\bar y$ is the empirical mean of the data and $\delta = \max\{y_1, \ldots, y_n\} - \min\{y_1, \ldots, y_n\}$ is the length of the empirical range.

From \eqref{eq:pes_mercer} and Theorem~\ref{teo:dpp_dens}, the DPP has density with respect to the unit-rate Poisson process on $S$ given by
\[
    f_{\Phi}(\nu) = \e^{-D} \det \{C(x, y)\}_{x, y \in \nu}
\]
where $C(x, y) = \sum_{j \in \mathbb Z} \frac{\lambda_j}{1 -\lambda_j} \cos(2 \pi \langle j, x - y \rangle)$, which cannot be numerically evaluated due to the infinite summation.
However, as shown numerically in \cite{ghilotti2023bayesian}, the $\lambda_j$'s typically decrease exponentially fast so that replacing $\mathbb Z$ with $\{-N, \ldots -1, 0, 1, \ldots, N\}$ in the summation produces accurate numerical evaluations of the DPP density even for very small values of $N$. In our examples, we fix $N=25$ and note that if $\rho = 2, \nu=2, s = 0.5$, $\lambda_j > 0$ only for $|j| \leq 5$.

\subsubsection{Sensitivity analysis}\label{app:dpp_sensitivity}

\begin{figure}[t]
    \centering
    \includegraphics[width=\linewidth]{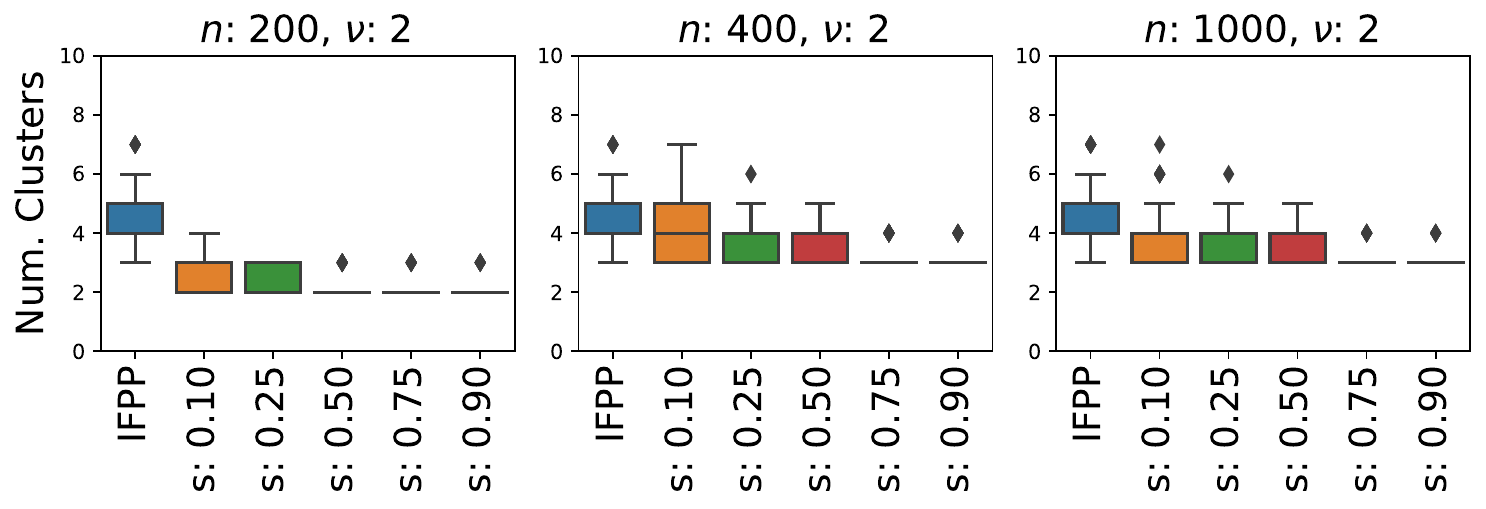}
    \includegraphics[width=\linewidth]{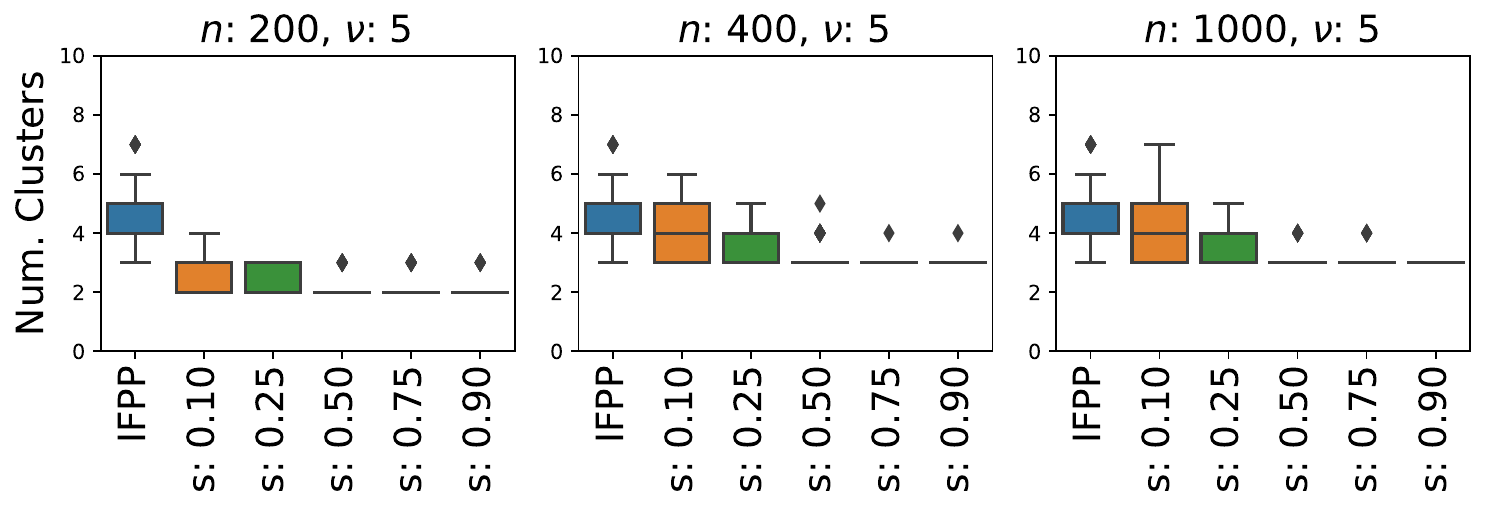}
    \caption{Posterior mean of the number of clusters $K_n$ for the simulation in Appendix \ref{app:dpp_sensitivity}. Boxplots refer to 100 independent replicates.}
    \label{fig:dpp_nclus}
\end{figure}

\begin{figure}[t]
    \centering
    \includegraphics[width=\linewidth]{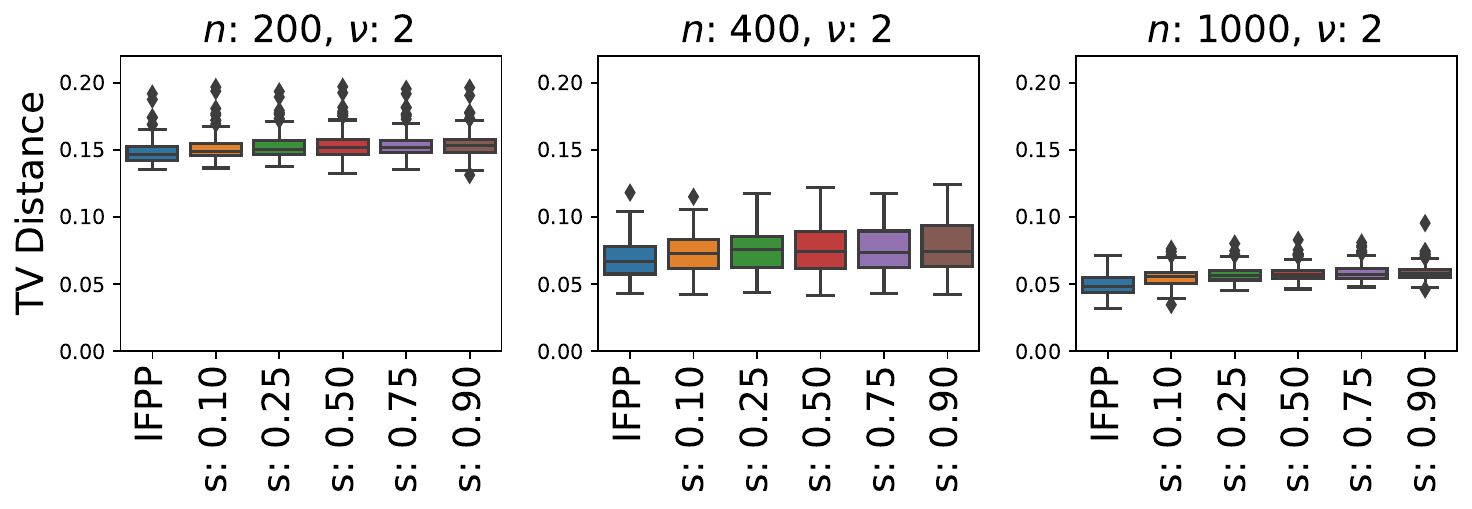}
    \includegraphics[width=\linewidth]{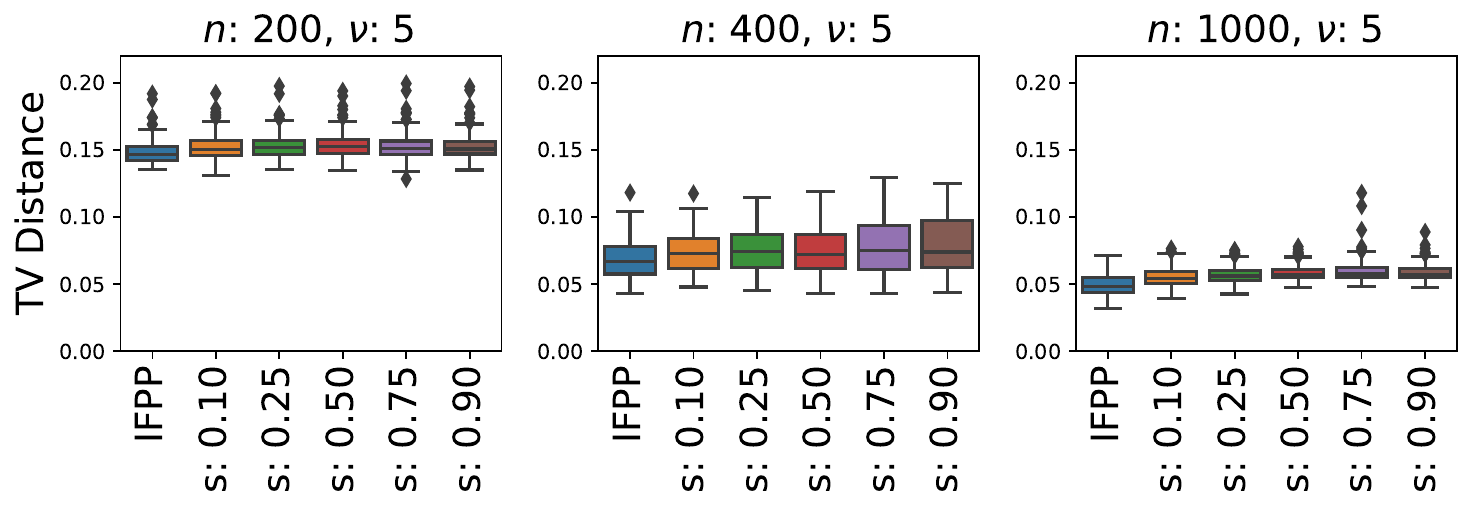}
    \caption{Total variation between the true and estimated densities for the simulation in Appendix \ref{app:dpp_sensitivity}. Boxplots refer to 100 independent replicates.}
    \label{fig:dpp_tv}
\end{figure}

We consider here a simple but illustrative numerical example, whereby we generate $n \in \{200, 500, 1000\}$ datapoints from a mixture of two Student's $t$ distributions as in Section~\ref{sec:simu1}.
Our modelling choices and prior hyperparameters are the same of Section~\ref{sec:numeric} except for the specification of the DPP prior. In particular, our goal here is to showcase the effect of different hyperparameters on posterior inference, specifically concerning cluster detection and density estimation. As discussed in Section~\ref{sec:tradeoff}, with a repulsive prior, we cannot jointly and accurately estimate the number of clusters and the true data-generating density due to model misspecification. 
Below, we empirically confirm that selecting hyperparameters that force more repulsiveness in the DPP prior leads to better cluster estimates at the cost of less accurate density estimates, and vice versa.
Moreover, we also show that the repulsiveness of the prior needs to be adapted to the sample size.

Figure~\ref{fig:dpp_nclus} shows the posterior mean of the number of clusters over 100 independent replicates for the DPP prior for different choices of $\nu$ and $s$ and $\rho=2$. We also add to the comparison the IFPP prior considered in Section~\ref{sec:numeric}.
Looking at all panels, it is clear that increasing $s$ leads to estimating fewer clusters a posteriori. 
However, moving across either the top or bottom row, for a fixed value of $s$ (e.g., $s=0.1$), we notice that as the sample size increases, so does the number of estimated clusters. Therefore, the larger the sample size, the more repulsive the prior needs to be.
Comparing the top and the bottom row of Figure~\ref{fig:dpp_nclus}, we appreciate the effect of $\nu$: the larger $\nu$, the more repulsive the prior is. This is evident, e.g., by looking at the plots corresponding to $n=400$ and $n=1000$ when $s = 0.5$. In these cases, if $\nu = 2$ we estimate 4 clusters a posteriori with a high probability, while if $\nu=5$ we estimate 3 clusters.
Figure~\ref{fig:dpp_tv} shows the total variation distance between the data generating density and the estimated density using the same settings as in Figure~\ref{fig:dpp_nclus}. This confirms our insights on the trade-off between density and cluster estimates with repulsive mixtures: in each plot, the larger $s$, the larger the distance between true and estimated density, even if the increase with $s$ is not dramatic.
Similarly, as the sample size increases, the density estimation error decreases.

\subsubsection{Learning hyperparameters via MCMC}\label{app:dpp_mcmc}

We now turn to the appealing possibility of learning the hyperparameters $(\rho, \nu, s)$ of the DPP prior in a fully Bayesian fashion. Such a strategy was proposed by \cite{bianchini2018determinantal} who suggested $\nu=2$, $s=0.5$ and $\rho - M \sim \mathrm{Gamma}(1, 1)$ where $M \approx 1.16$ is a positive constant. Since they tested their model on datasets where the number of clusters was fairly large (always larger than 6), we can comment that their default strategy is to set $\rho$ to a small value (recall that $\rho$ is the expected number of points a priori).
Moreover, they showcase the mixture model only on small datasets with $n \leq 100$ observations, and then focus on more complex models where covariates are included.

\begin{figure}[t]
    \centering
    \includegraphics[width=\linewidth]{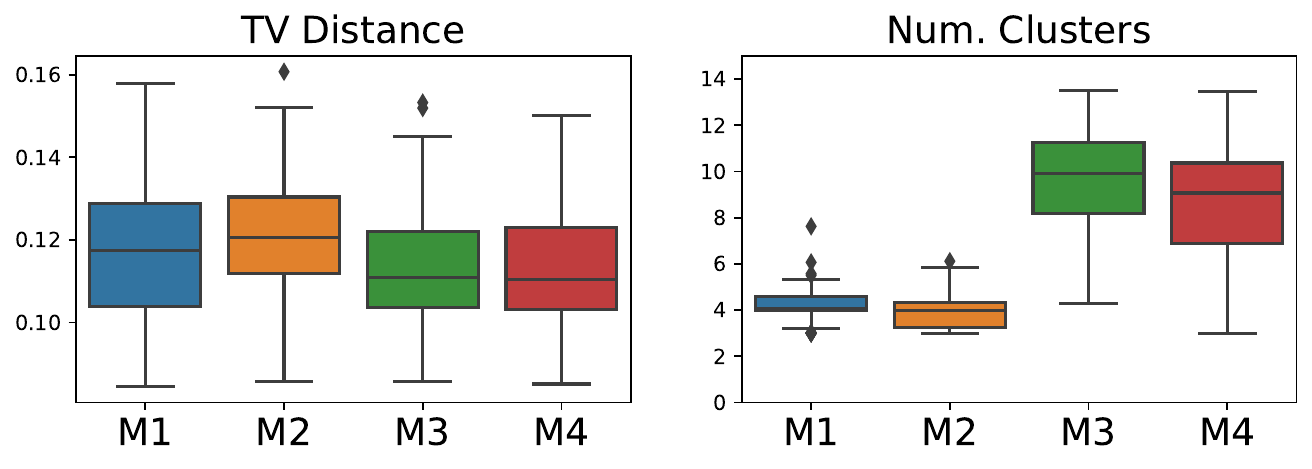}
    \includegraphics[width=\linewidth]{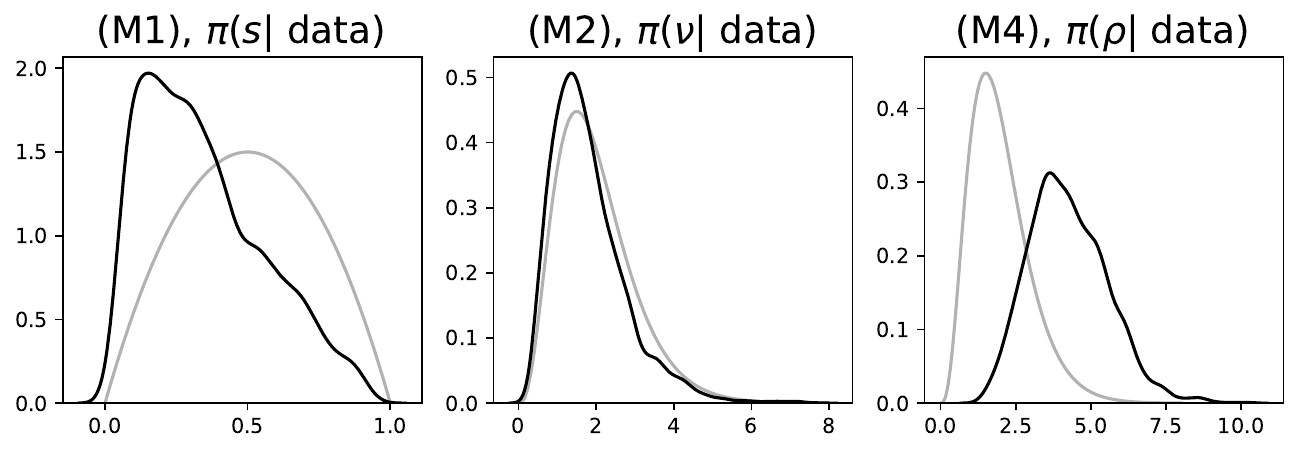}
    \caption{Total variation distance between the true and estimated densities (top left) and posterior mean for the number of clusters (top right). Boxplots refer to 100 independent experiments.
    Bottom row: kernel density estimates of the posterior distribution of relevant parameters under different models (black line) and associated prior densities (grey line).}
    \label{fig:dpp2}
\end{figure}

\cite{sun2022bayesian} consider a different prior specification based on repulsive Mat\'ern processes. Their numerical illustrations suggest that learning the parameters governing the repulsiveness in a fully-Bayesian requires very concentrated prior densities. Indeed, if they assume ``vague'' priors, the posterior puts mass to parameters favoring less repulsiveness in the model, thus overestimating the number of clusters.
Our experiments confirm that this is not a specific drawback to Mat\'ern processes, but is shared also by DPP priors and, we argue, by all repulsive priors.

In particular, we consider the same simulated dataset of Appendix \ref{app:dpp_density} and fix the sample size $n=200$. Then, we consider four different prior specification for the hyperparameters of the DPP: 
\begin{itemize}
    \item[(M1)] $\rho =  \nu = 2$, $s \sim \mathrm{Beta}(2, 2)$
    \item[(M2)] $\rho = 2$, $s = 0.75$, $\nu \sim \mathrm{Gamma}(4, 2)$
    \item[(M3)] $\nu = 2$, $\rho \sim \mathrm{Gamma}(4, 2)$, $s \sim \mathrm{Beta}(2, 2)$
    \item[(M4)] $\rho \sim \mathrm{Gamma}(4, 2)$, $\nu \sim \mathrm{Gamma}(4, 2)$, $s \sim \mathrm{Beta}(2, 2)$
\end{itemize}

In all cases, the random parameters are updated as part of the MCMC algorithm via a Metropolis-Hastings move. Specifically, for $\rho$ and $\nu$, we perform an update in the log scale and use a centered Gaussian proposal with standard deviation $0.25$, while for $s$ we work on the logit scale.

Figure~\ref{fig:dpp2} summarizes our findings. Looking at the top row, we clearly see that models (M3) and (M4) are associated with a higher number of clusters and lower density estimation errors, thus highlighting once again the tradeoff between clustering and density estimation. The posterior for $s$ under (M1) clearly shifts towards smaller values, thus resulting in a model that forces less repulsion a posteriori. Interestingly, the posterior on $\nu$ under model (M2) does not exhibit a significant difference from the prior. Indeed, model (M2) shows the smallest number of estimated clusters and the larger density estimation error.
Finally, the posterior for $\rho$ under model (M4) shows a clear shift towards higher values, still leading to less repulsion a posteriori.

Let us now give an intuition on why assuming the hyperparameters to be random leads to overestimating the number of clusters.
As discussed also in Section~\ref{sec:tradeoff}, we argue that this behavior can be traced back to the asymptotic behavior of nonparametric posteriors under model misspecification.
As shown in \cite{Kle(06)} the posterior of the mixture density contracts to a ``pseudo-true'' density $f^*(z) = \int_{\X \times \W} f(z \mid y, v) p^*(\dd y \, \dd v)$, that minimizes the Kullback-Leibler divergence between the true data generating density, say $f_0$, and the support of the prior.
As shown in \cite{cai2021finite}, if the prior is an IFPP, where the prior of the number of points is supported on the whole $\mathbb N$, then the posterior of the number of clusters diverges as the sample size increases because such mixtures yield consistent estimators for the density, i.e., $f^* = f_0$.
Now, observe that we can recover an IFPP prior as a limiting case of a DPP, namely if $K(x, y) = \indicator[x = y]$. 
In particular, \eqref{eq:pes_mercer} converges to the Fourier series of $\delta_0$ if we let $\rho \rightarrow +\infty$, or equivalently $\alpha_{\max} \rightarrow 0$.
As shown in our simulations above, letting hyperparameters $\rho$ or $s$ be random leads to posterior distributions favoring less repulsive behavior, which can be understood as the model trying to recover $f_0$ as closely as possible.

\subsection{The SNCP Prior}

\subsubsection{Prior elicitation}\label{app:sncp_prior}

\begin{figure}
    \centering
    \begin{subfigure}{0.33\linewidth}
    \centering
        \includegraphics[width=\linewidth]{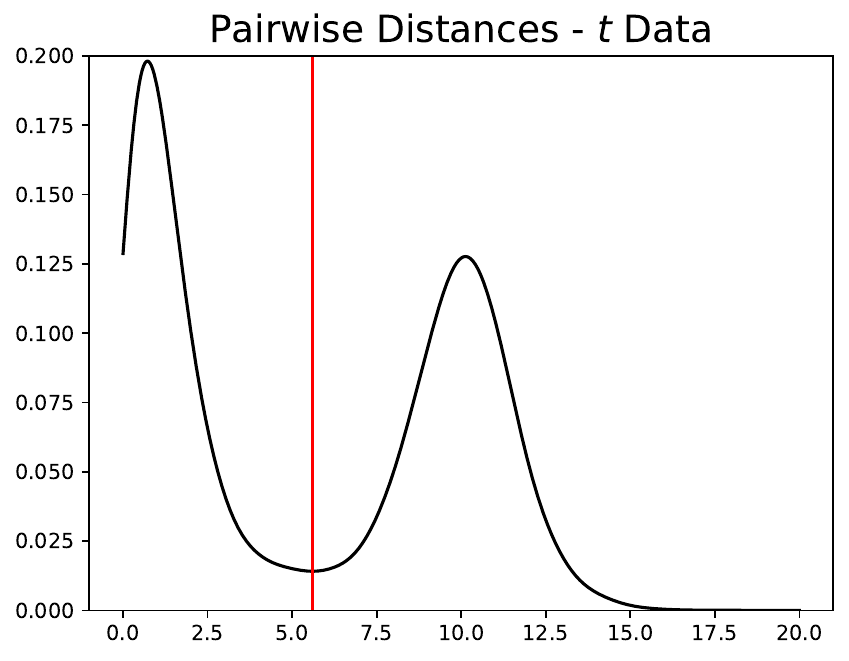}
    \end{subfigure}%
    \begin{subfigure}{0.33\linewidth}
    \centering
        \includegraphics[width=\linewidth]{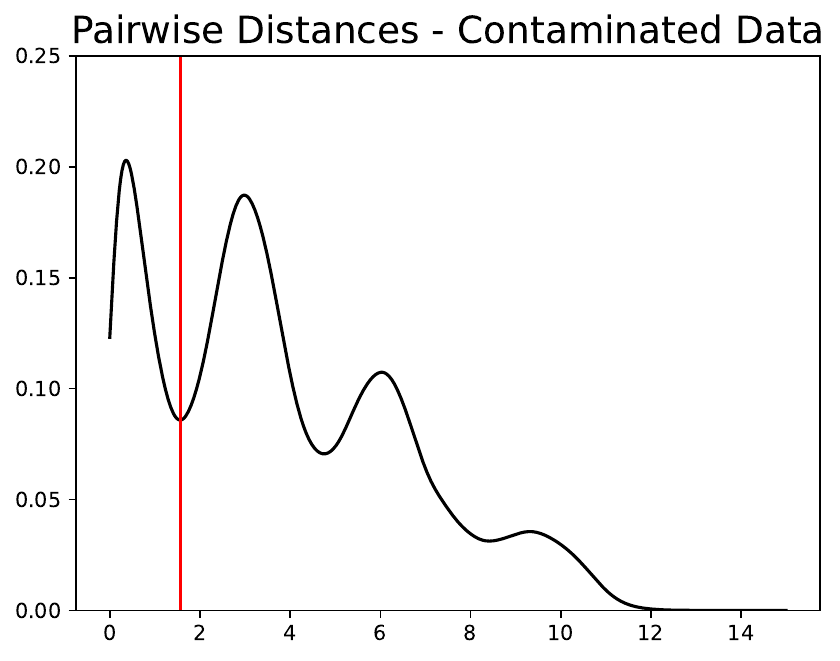}
    \end{subfigure}%
    \begin{subfigure}{0.33\linewidth}
    \centering
        \includegraphics[width=\linewidth]{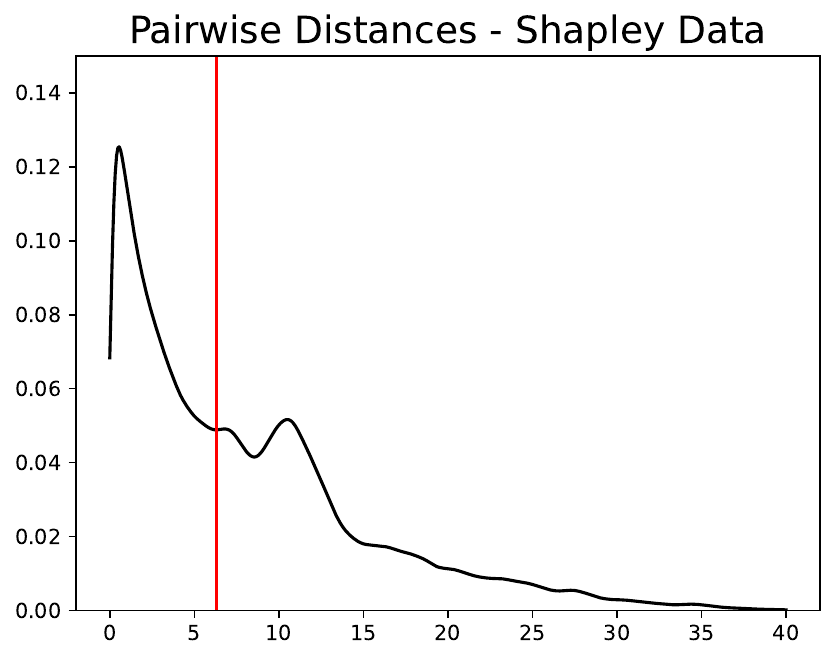}
    \end{subfigure}
    \caption{Kernel density estimates of the pairwise distances of datapoints for the datasets considered in Section~\ref{sec:numeric}.}
    \label{fig:pairwise_dists}
\end{figure}

We assume throughout that the kernel $k_\alpha$ is the univariate Gaussian density with variance $\alpha^2$.
The SNCP mixture model depends on the parameters $(\gamma, \alpha, \omega)$ that define the prior for $\Phi$, $a$ (the shape distribution of the Gamma prior on the unnormalised weights), and the parameters involved in the prior for the component-specific variances.

The parameters $\gamma$ and $a$ have a strong interplay as is usual in the case of finite mixtures \citep{ArDeInf19}. The number of \emph{mixture components increases with $\gamma$, and in particular, it is easy to see that its expected value is $\gamma \omega(\X)$}. 
On the other hand, the number of active components $K_n$ generally differs from $M$ and depends also on $a$. 
If $n \rightarrow +\infty$, $K_n \rightarrow M$ in distribution by de Finetti's theorem, but for finite samples, their distribution can drastically differ between a priori and a posteriori. This is not specific to the choice of the SNCP prior, and common prior elicitation strategies for finite mixtures can be used here.
The parameters controlling the prior for the component-specific variances also play a key role in the induced clustering. Also, in this case, this is common to all mixture models.

Specific to the SNCP prior instead are the parameters $\alpha$ and $\omega$. To understand their role, it is convenient to consider the double-summation formulation of $\Phi$ as in \eqref{eq:shot-noise2} as done in Section \ref{sec:sncp_mixture_of_mixture}. In particular, \eqref{eq:sncp_comp} shows that \emph{component} $h$, representing the $h$-th \emph{group} of data, has a mixture density supported over the points of $\Phi_h = \sum_{j \geq 1} \delta_{ X_{h, j}}$ such that $\Phi_h \mid \Lambda \sim \mathrm{PP}(\gamma k_\alpha(\cdot - \zeta_h))$.
In particular, the \emph{center} of the $h$-th component can be identified with $\zeta_h$.
Recalling that  $\Lambda = \sum_{h \geq 1} \delta_{\zeta_h} \sim \mathrm{PP}(\omega)$, it is clear how $\omega$ plays a similar role to the \emph{centering measure} of traditional mixtures, as it controls the location of the cluster centers $\zeta_j$'s. We assume $\omega$ to be the Gaussian measure and follow the typical prior elicitation strategies in setting its expectation equal to the empirical mean of the observed data and its variance equal to the empirical variance of the sample.

Moreover, recalling that $k_\alpha$ is the univariate Gaussian density with variance $\alpha^2$, the interpretation of $\alpha$ becomes transparent: it controls the dispersion of data within each \emph{group},
meaning that if $\alpha$ is small, each $X_{h, j}$ (for $j \geq 1$) will be close to $\zeta_h$. If we assume that the component means $X_{h, j}$ are somehow fixed by an oracle, small values of $\alpha$ will result in more ``groups'' of data being detected. In contrast, for large values of $\alpha$, we will detect a smaller number of groups, as the model will tend to cluster together the $X_{h, j}$'s.
This led us to consider $\alpha$ in the same ways as the \emph{interaction range} of repulsive point process priors. Then, we can employ the empirical Bayes strategy described in \cite{beraha21} to set $\alpha$.
Indeed \cite{beraha21} suggest fixing the interaction range by considering the pairwise distances of datapoints. In particular, they argue that the smallest local minimum (say $d^*$) of the density of pairwise distances gives a heuristic estimate of the smallest distance between two possible clusters. Hence, they fix the interaction range to $d^*$. Figure~\ref{fig:pairwise_dists} shows the kernel density estimate of the distance between datapoints in the numerical examples analyzed in our paper.
In the SNCP mixture, we cannot force separation between different clusters. We choose $\alpha$ such that $\prob(|X_{h, j} - X_{h, \ell}| < d^*) > 0.99$: that is, for two-component centres belonging to the same group of data, the probability that their distance is larger than the chosen range $d^*$ is less than $1\%$.
In particular, for the $t$ dataset, this procedure yields $\alpha \approx 2.3$, for the contaminated dataset $\alpha \approx 0.85$ and for the Shapley galaxy data $\alpha \approx 3.0$.

\subsubsection{Prior for the Shapley Galaxy Example}\label{app:galaxy}

For the SNCP mixture, following the discussion above, we fix $\omega = \mathcal N(\bar \mu, \bar \sigma^2)$ where $\bar \mu \approx 14.6$ and $\bar \sigma \approx 7.4$ are the empirical mean and standard deviation of the data.
We fix $\gamma=1$ so that each component as in \eqref{eq:sncp_comp} is expected to be represented by one Gaussian only and set $a=2$. As mentioned in the previous paragraph, we set $\alpha = 3.0$.
The variances are i.i.d. as the inverse-Gamma density with parameters $(2.0, 25.0)$. 

For the IFPP mixture, we assume that the component-specific means and variances follow a Normal-inverse-Gamma density such that $X_i \mid W_j \sim \mathcal N(\bar \mu, 10 W_j)$ and $1/W_j \sim \mbox{Gamma}(2, 25)$ and assume that $K-1 \sim \mbox{Poisson}(1)$.

For the DPP mixture, we set the same marginal prior for the variances and select the DPP parameters as in \cite{beraha21}.

\subsubsection[The role of alpha in SNCP mixtures]{The role of $\alpha$ in SNCP mixtures}\label{app:alpha_choice}

We report here a sensitivity analysis for the choice of $\alpha$ in the case of the Shapley data. We consider a subsample of $n=1000$ observations from the full dataset and run a posterior inference with parameters fixed as in Appendix \ref{app:galaxy} and $\alpha = 0.5, 1.0, 2.5, 5, 10, 50$. Figure~\ref{fig:galaxy_alpha} displays the density estimates and posterior distribution for the number of groups.
We can see that the density estimate is robust to the choice of $\alpha$, being virtually indistinguishable across all the different values.
Instead, the number of detected groups varies with $\alpha$, but not dramatically. In particular, when $\alpha \leq 25$, we estimate between 6 and 9 clusters, while for $\alpha=50$, the model identifies only 5 clusters.
Overall, inference appears to be robust to $\alpha$.

\begin{figure}
    \centering
    \includegraphics[width=0.75\linewidth]{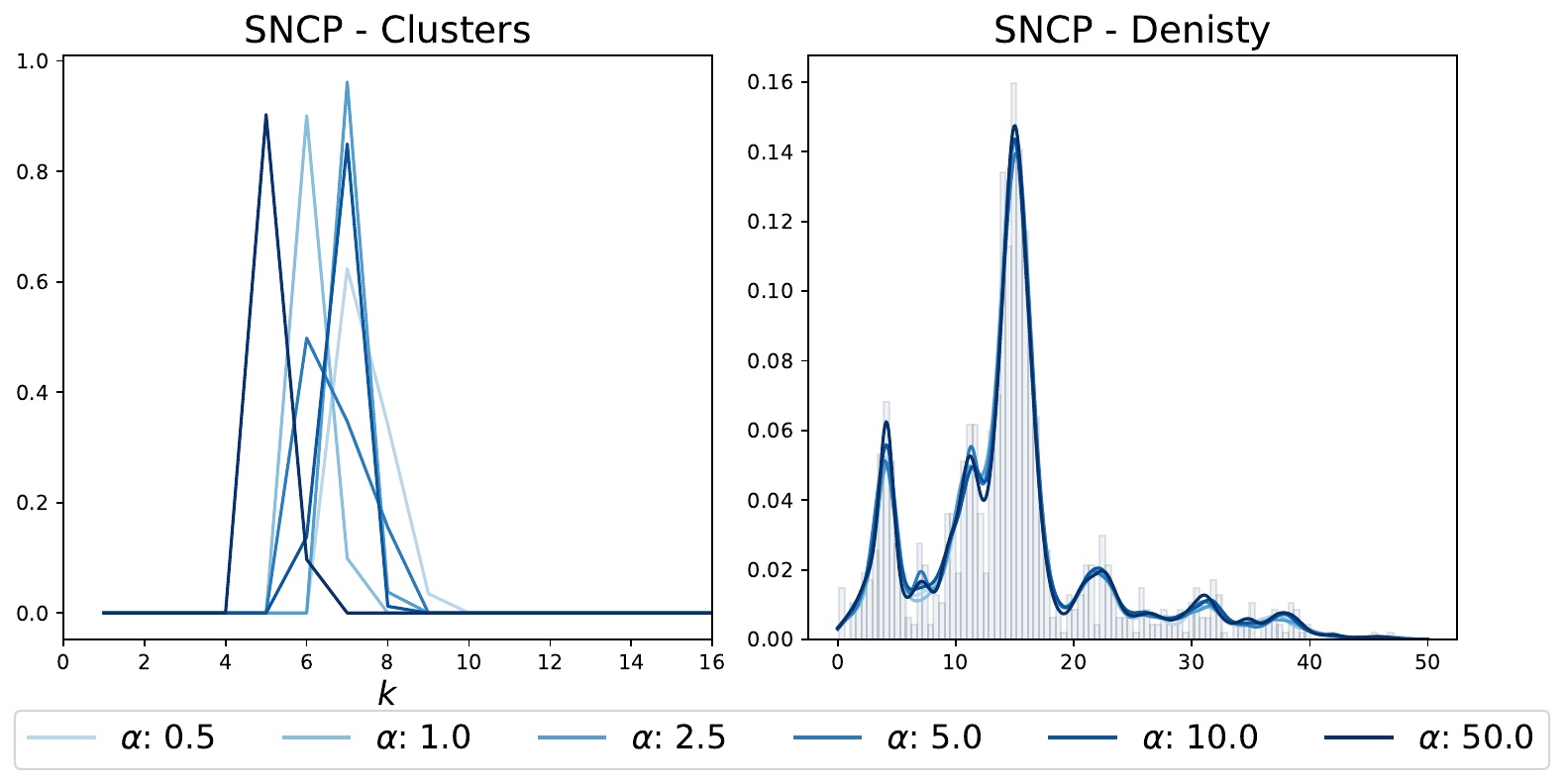}
    \caption{Posterior distribution of the number of clusters and density estimates as $\alpha$ varies.}
    \label{fig:galaxy_alpha}
\end{figure}

\end{document}